\documentclass[11pt]{amsart}
\usepackage{amsmath,amssymb,mathtools}
\usepackage[T1]{fontenc}
\usepackage{lmodern,textcomp}
\usepackage{color}
\usepackage{graphicx}
\usepackage{amsthm}
\usepackage[all]{xy}
\usepackage{caption}
\usepackage{enumitem}
\usepackage{hyperref}
\usepackage{csquotes}

\usepackage{soul}

\usepackage{tikz}
\usepackage{tikz-cd}

\usepackage{fullpage}
 \usepackage{setspace}
\usepackage{cleveref}
\crefformat{equation}{(#2#1#3)}

\AtBeginDocument{\renewcommand{\ref}[1]{\Cref{#1}}}

\newcommand{\bb}[1]{{\mathbb{#1}}}

\newcommand{\ca}[1]{{\mathcal{#1}}}

\def\H{{\mathcal{H}}}
\def\M{{\mathcal{M}}}

\def\C{{\mathcal{C}}}

\def\L{{\mathcal{L}}}

\def\cH{{\mathcal{H}}}
\def\cM{{\mathcal{M}}}

\def\cC{{\mathcal{C}}}
\def\cL{{\mathcal{L}}}
\def\cP{\mathcal{P}}

\def\cO{\mathcal{O}}

\def\cF{\mathcal{F}}
\def\cJ{\mathcal{J}}

\def\oC{\overline{\mathcal{C}}}
\def\oM{\overline{\mathcal{M}}}
\def\oJ{\overline{\mathcal{J}}}
\newcommand{\Mbar}{\oM}

\def\oH{\overline{\mathcal{H}}}
\def\oSQ{\overline{\mathcal{SQ}}}
\def\SQ{{\mathcal{SQ}}}
\def\oGamma{{\overline{\Gamma}}}

\def\wH{\widetilde{\mathcal{H}}}

\def\CC{\mathbb{C}}
\def\ZZ{\mathbb{Z}}
\def\NN{\mathbb{N}}
\def\QQ{\mathbb{Q}}
\def\PP{\mathbb{P}}

\def\bg{{\mathbf{g}}}
\def\bn{{\mathbf{n}}}
\def\ba{{\mathbf{a}}}
\def\bP{{\mathbf{P}}}
\def\ba{{\mathbf{a}}}

\def\fR{{\mathfrak{R}}}

\def\half{{1/2}}

\def\DR{{\rm DR}}
\def\Aut{{\rm Aut}}

\def\DRL{{\rm DRL}}
\def\lcm{{\rm lcm}}

\def\rP{{\mathsf{P}}}
\def\rH{{\mathsf{H}}}

\newcommand{\spec}{\mathrm{Spec}}

\theoremstyle{definition}
\newtheorem{definition}{Definition}[section]

\newtheorem{remark}[definition]{Remark}

\theoremstyle{plain}
\newtheorem{conjecture}[definition]{Conjecture}
\newtheorem{theorem}[definition]{Theorem}
\newtheorem{proposition}[definition]{Proposition}

\newtheorem{lemma}[definition]{Lemma}
\newtheorem{corollary}[definition]{Corollary}

\begin{document}

\title[DR cycles and strata of differentials with spin parity]{Classes of strata of differentials with spin parity}
\author{David Holmes}
\address{Mathematical Institute, Leiden University, PO Box 9512, 2300 RA Leiden, The Netherlands}
\email{holmesdst@math.leidenuniv.nl}
\author{Georgios Politopoulos}
\address{Mathematical Institute, Leiden University, PO Box 9512, 2300 RA Leiden, The Netherlands}
\email{g.politopoulos@math.leidenuniv.nl }
\author{Adrien Sauvaget}
\address{Institut de Mathématiques de Toulouse, Université Paul Sabatier, 31400 Toulouse, France}
\email{adrien.sauvaget@math.cnrs.fr}

\begin{abstract} We study  classes of strata of differentials with fixed spin parity in the Chow ring of moduli spaces of curves.  We show that these classes are tautological and computable. Furthermore, we establish the refined DR cycle formula for these classes. 
\end{abstract}

\maketitle

\setcounter{tocdepth}{1}
\tableofcontents

\section{Introduction}

\subsection{Strata of $k$-differentials and their spin refinement}
Let $g$ and $n$ be non-negative integers such that $2g-2+n>0$.  We denote by $\M_{g,n}$ and $\oM_{g,n}$ the moduli spaces of smooth and stable complex curves of genus $g$ with $n$ markings.  Let $k\in \ZZ_{>0}$, and let  $a=(a_1,\ldots,a_n)\in \ZZ^n$ be a vector of size
$$|a|\left(\overset{\rm def}{=}\sum_{i=1}^n a_i\right)=k(2g-2+n).$$  The {\em stratum of $k$-differentials  of type $a$} is the algebraic sub-stack of  $\M_g(a,k)\subset \M_{g,n}$ of marked curves $(C,x_1,\ldots,x_n)$ for which the relation 
\begin{equation}\label{eq:omegapic}
\omega_{\rm log}^{\otimes k} \simeq \mathcal{O}_C\big(a_1 x_1+\ldots + a_nx_n)
\end{equation}
holds in the Picard group of $C$ (here $\omega_{\rm log}\coloneqq\omega_C\big(x_1+\ldots+x_n)$  is the log-cotangent bundle).  For $k=1$ we simply denote it by $\M_g(a)$, and if $a$ is a vector of positive integers, then we speak of {\em strata of holomorphic differentials}. These  strata will play a special role throughout the paper. For instance, most strata of $k$-differentials are of co-dimension $g$ while strata of holomorphic differentials are of co-dimension $(g-1)$.
\smallskip

We denote by $\oM_g(a,k)$ the closure of $\M_g(a,k)$ in $\oM_{g,n}$, and by $[ \oM_g(a,k) ]\in A^*(\oM_{g,n},\QQ)$ the class of this cycle with the reduced sub-scheme structure. The natural problem is then: {\em are the classes of strata tautological? If yes, then can we compute them in terms of the standard generators of the tautological rings?} These questions go back to the early developments of tautological calculus for moduli spaces of curves. For instance, Faber computed in 1999 the restriction of the classes of strata of holomorphic differentials to $\M_{g,n}$ using the Thom--Porteous formula~\cite{Fab}. As of today, we can identify the following families of methods in the literature to compute classes of strata of differentials: 
 \begin{enumerate}
 \item Computing classes of {\em degeneracy loci in the projectivized Hodge bundle}. This method led to the first algorithm to compute all classes of strata of $1$-differentials ~\cite{Sau} (see~\ref{sec:introspinsection} for details). This method was generalized to higher values of $k$ but only for the restriction to the moduli space of curves with rational tails $\M_{g,n}^{\rm rt}$~\cite{GheTar}.
 \item Expressing $[\oM_g(a,k) ]$ in terms of {\em double ramification cycles}~\cite{FarPan,Sch,HolSch,BHPSS} allowed for their effective computation for all values of $k$ and $a$ (see~\ref{sec:introDR} for details). Early developments of these ideas could be found in the work of Hain~\cite{Hain}.
 \item A variation on the previous approach led to a conjectural expression of classes of strata of holomorphic differentials in terms of {\em Witten's $r$-spin classes}~\cite{PPZ19}. 
 \item Special curves (often called {\em test curves}) in $\oM_{g,n}$ can be used to evaluate the coefficients of classes in the Picard group of $\oM_{g,n}$. This method allows one to obtain a closed formula for the classes of divisors obtained from classes of strata~\cite{CT,Mul}. Faber used this approach to enhance his computation in $\M_{g,n}$ to obtain expressions in $\M_{g,n}^{\rm rt}$~\cite{Fab}. 
 \item Generalizing the previous idea, Wong has proposed an algorithm that allows for the computation of these classes, {\em assuming that all  classes of strata are tautological}. The method is based on the inversion of the linear system of restriction of classes (in any degree) to the union of the boundary divisors of $\oM_{g,n}$~\cite{Wong}. The main drawback of this method is that the value of the classes is hard to track after this inversion of linear system (in particular it is quite difficult to compute tautological integrals on strata with this approach). 
 \end{enumerate}

The main purpose of the paper is to compute the refinement of classes of strata of differentials according to spin parity.  A {\em spin structure} on a curve $C$ is a line bundle $L\to C$ satisfying $L^{\otimes 2}\simeq \omega_C$.  The {\em parity} of a spin structure is defined as the parity of $h^0(C,L)$.  It is invariant along deformations of a spin structure~\cite{Ati,Mum1}.  If $k$ and $a$ are odd (i.e.  all entries of $a$ are odd),  then a point of $\M_g(a,k)$ carries a canonical spin structure defined as
$$
\omega_C^{\otimes \frac{1-k}{2}} \left(\frac{a_1-k}{2} x_1 +\ldots + \frac{a_n-k}{2} x_n  \right).
$$
The local invariance of spin parity implies that $\M_g(a,k)$ splits into two disjoint sub-spaces $\M_g(a,k)^{+}$ and $\M_g(a,k)^-$ of $k$-differentials with even and odd spin parity respectively. We denote by $\oM_g(a,k)^{+}$ and $\oM_g(a,k)^{-}$ their closures in $\oM_{g}(a,k)$. Then again, we are led to ask: {\em are the classes of $\oM_g(a,k)^{+}$ and $\oM_g(a,k)^{-}$ tautological? Can we compute them?}  As the class of $\oM_g(a,k)$ was already computed, it remains to compute the {\em spin class of the stratum} defined as
$$
[\oM_g(a,k)]^{\pm} \coloneqq [\oM_g(a,k)^+] - [\oM_g(a,k)^-] \in A^*(\oM_{g,n},\QQ).
$$

\begin{theorem}\label{thm:main}
The class $[\oM_g(a,k)]^{\pm}$ is tautological and  explicitly computable.
\end{theorem}

\subsubsection*{Previous results} In genus 1 and 2,  strata of holomorphic differentials are connected, so the first interesting computations occur in genus 3. The even loci $\M_{3}((5))^+$ and $\M_3((3,3))^+$ are the loci of hyper-elliptic curves with a marked Weierstrass point, and a marked pair of conjugated points under the hyper-elliptic involution, respectively. The class of these loci (after forgetting the markings) was first computed by Harris--Mumford, by applying a combination of Thom--Porteous formula and test curves computations~\cite{HarMum}.  Since this early result, little progress was made until the work of Wong who produced an algorithm that returns $[\oM_{g}(a)]^\pm$ if all spin classes of strata of $1$-differentials are tautological up to genus $g$, \emph{or} raises an error if this assumption is not valid~\cite[Theorem 1.4]{Wong}. This algorithm provided the first explicit values beyond genus 3. We should emphasize that~\ref{thm:main} 
is established by providing a more direct algorithm, in particular we do not need to compute the full tautological ring but only provide the expression of the desired class in terms of the standard generators. However,  as a corollary, the algorithm of~\cite{Wong} computes spin classes of strata of 1-differentials in all genera.

\subsubsection*{Motivation and future works} Strata of holomorphic differentials are equipped with a measure and an ergodic action of ${\rm SL}(2, \mathbb{R})$, which can be understood by interpreting holomorphic differentials as {\em translation surfaces}. We refer the reader to the surveys~\cite{Zor,Fil} for an overview of this perspective and the associated literature over the past two decades. Recently, the point of view of algebraic geometry and intersection theory has yielded new insights into the numerical invariants associated with the ergodic structure on these strata, such as the {\em Masur--Veech volumes}, {\em Siegel--Veech constants}, and {\em Lyapunov exponents}~\cite{Sau2,CMSZ,IttSau}. From this perspective,~\ref{thm:main} was both natural  and desirable. For instance, the Masur--Veech volumes of the {\em minimal} strata of holomorphic differentials (strata with $n=1$) were computed by intersection theory while the volume of each component defined by spin parity still requires the detour via representation theory and the Bloch--Okounkov formalism~\cite{EskOko,EOP}. The present result could lead to a direct geometric approach.

\smallskip

The second (and maybe more important) motivation comes from FJRW theory and the expression of classes of strata of holomorphic differentials in terms of Witten's classes. In the 90's, Witten proposed to construct a class $W_g(a,r)\in A^*(\oM_{g,n},\QQ)$ using moduli spaces of $r$-spin structures ($r$th roots of the canonical bundle)~\cite{Wit}. The proper mathematical construction of these classes was only achieved years later by Polishchuk and Vaintrob~\cite{PolishchukVaintrob}. Alternative constructions were proposed over time, leading to the more general set-up of FJRW theory or {\em gauged linear sigma models} (GLSM)~\cite{Chi1,CLL,CJRS}. Motivated by the development of the theory of double ramification cycles, Janda--Pandharipande--Pixton--Zvonkine observed that for a fixed value of $g$ and $a$, the class $W_g(a,r)$ is polynomial in $r$ and conjectured the identity \begin{equation}\label{for:witstrata}
    [\oM_g(a)]=\text{constant coefficient in $r$ of $W_g(a,r)$}
\end{equation} 
(see~\cite{PPZ19}). This conjecture provides an explicit method to compute classes of holomorphic strata. Indeed, for a fixed $r$, Witten's classes form a cohomological field theory (CohFT) which is computable through the following observations: this CohFT can be deformed into a semi-simple one (using $A$-Frobenius manifolds), and semi-simple CohFTs are determined by their value in genus 0 by Teleman's reconstruction theorem (see~\cite{Teleman}). If $r$ is even and $a$ is odd, then Witten's class also admits a refinement $W_g^\pm(a,r)$ according to the spin parity, and we have the following natural conjecture.  
\begin{conjecture}
    For a fixed value of $a$, $W_g^\pm(a,r)$ behaves polynomialy and the class  $[\oM_g(a)]^\pm$ is the constant coefficient of this polynomial. 
\end{conjecture}

We emphasize that this conjecture would NOT determine the class $[\oM_{g}(a)]^\pm$, because the classes $W_g^\pm(a,r)$ fail to satisfy the CohFT axioms (so Telemann's reconstruction theorem cannot be applied). However, motivated by the conjectural expression~\ref{for:witstrata}, a new approach to GLSM has been developed in~\cite{CJRS,CJR} that hints that {\em Witten's classes and classes of strata of holomorphic differentials determine each other} (and that the spin analog is also true). If this expectation is valid, then~\ref{thm:main} would allow for a first explicit computation of the classes $W_g^\pm(a,r)$. This result might be of importance as the $W_g^\pm(a,r)$ should determine the potential of the  singular fiber of the $B$-Frobenius manifolds while no other singular manifolds have been described geometrically in the literature on FJRW theory so far.

\subsubsection*{The approach of the present paper} Above, we identified five families of methods for computing strata classes. However, to the best of our knowledge, none of these approaches can be directly adapted to compute the $[\oM_{g}(a,k)]^\pm$ (as exemplified by the discussion on Witten's class). Our proof of~\ref{thm:main} relies on a combination of the first and second approaches, along with a direct description of cones of meromorphic sections of spin structures and their Segre classes. The remainder of the introduction is devoted to presenting the algorithm, which includes the two new identities established in this paper: the spin DR cycle conjecture (\ref{thm:DR}), and the expression of Segre classes of cones of spin sections in terms of quadratic Hodge symbols (\ref{thm:segre}).

\subsection{Spin double ramification cycles}\label{sec:introDR}

 Let $\mathcal{J}_{g,n}\to \M_{g,n}$ be the universal Jacobian  of degree 0 line bundles.  The line bundles $\omega_{\rm log}^{\otimes k}(-a_1x_1-\ldots)$ and $\mathcal{O}_C$ define sections $\sigma_{a,k}$ and $\sigma_0$ of the universal Jacobian. With this notation, the space $\M(a,k)$ defined by the relation~\ref{eq:omegapic} is the fiber product $\M_{g,n}\times_{\mathcal{J}_{g,n}} \M_{g,n}$, where the morphisms used for each factor are $\sigma_{a,k}$ and $\sigma_0$ respectively. The universal Jacobian admits a natural extension to $\oM_{g,n}$ as the group of multi-degree zero line bundles, but the section $\sigma_{a,k}$ does not extend to $\oM_{g,n}$. Instead, there exists a birational model $\rho\colon \oM_g^{a}\to \oM_{g,n}$  allowing for the existence of $\sigma_{a,k}$, and which is universal for this property~\cite{Hol,Marcus2017LogarithmicCO}.  Moreover, the schematic intersection of the extended sections $\sigma_0$ and $\sigma_{a,k}$ is proper over $\oM_{g,n}$, thus enabling the definition of the {\em double ramification (DR) cycle} as
\begin{equation}
    \DR_g(a,k)\coloneqq \rho_* \left(\sigma_{0}^! [\sigma_{a,k}] \right) \in A^{g}(\oM_{g,n},\QQ),
\end{equation}
where $\sigma_{0}^!$ stands for Gysin pull-back. In~\cite{FarPan,Sch}, two conjectural expressions of the DR cycle were proposed 
\begin{equation}
    \rP_g(a,k)=\DR_{g}(a,k)=  \rH_g(a,k).
\end{equation}
Here the {\em Pixton's class} $P_g(a,k)$ is an explicit tautological class, while the {\em star graph expression} $\rH_g(a,k)$ is the class of an algebraic cycle constructed from strata of differentials. These equalities were established in~\cite{BHPSS}, and~\cite{HolSch} respectively. Moreover, Farkas--Pandharipande proved that the classes of strata  $[\oM_{g}(a,k)]$ can be recovered explicitly from the classes ${\rH}_g(a,k)$~\cite{FarPan}. Put together, these results provide an explicit expression for $[\oM_{g}(a,k)]$, in terms of tautological classes for all $k$ and $a$. 
\smallskip

\subsubsection{Construction of the spin DR cycle} 

We work over the {\em moduli space of  spin structures} $\M_{g,n}^{1/2}$, i.e. the moduli space of tuples $(C,x_1,\ldots,x_n,L, \phi\colon L^{\otimes 2}\xrightarrow{\sim} \omega_C)$. This space admits a compactification $\oM_{g,n}^{1/2}$ constructed by allowing certain semi-stable degenerations of the base curve (see~\cite{Cornalba1989ModuliOC} and~\cite{Caporaso2004ModuliOR}). The space $\oM_{g,n}^\half$ is a DM stack, and there exists a finite morphism  $\epsilon \colon \oM_{g,n}^{1/2}\to \oM_{g,n}$ defined by forgetting the spin structure. Moreover, Cornalba showed that the parity of $h^0(C,L)$ is still locally constant in $\oM_{g,n}^{1/2}$ so the space has two connected components $\oM_{g,n}^{1/2,+}$ and $\oM_{g,n}^{1/2,-}.$ The {\em parity cycle} is then defined as the difference
\begin{equation*}
    [\pm]:=[\overline{\cM}^{1/2,+}_{g,n}]-[\overline{\cM}^{1/2,-}_{g,n}]\in A^{0}(\overline{\cM}^{1/2}_{g,n}).
\end{equation*}

Let $\mathcal{L}\to \oC_{g,n}^{1/2}\overset{\pi}{\to} \oM_{g,n}^{1/2}$ denote the universal spin structure over the universal curve. If $k$ and $a$ are odd, then we may consider the section $\sigma^{\half}_{a,k}$ of the universal Jacobian over $\M_{g,n}^{1/2}$ defined by the line bundle 
\begin{equation}\label{eqn:section on smooth part}
    \mathcal{L}\otimes\omega_{C}^{\otimes \frac{k-1}{2}}\left(-\frac{a_1-k}{2}x_1-\ldots-\frac{a_n-k}{2}x_n\right),
\end{equation}
where $x_i$ stands for the section of the universal curve associated with the $i$-th marking. Again, the section $\sigma^{\half}_{a,k}$ extends to $\oM_{g,n}^\half$ if we pass to a (canonical) birational model $\rho^\half\colon \oM_g^{a,\half} \to \oM_{g,n}^{\half}$, then the {\em spin DR cycle} is defined as \begin{align*}
\DR_g^{\pm}(a,k)=& 2 \epsilon_*\left([\pm] \cdot \DR_g^\half(a,k)\right)\in A^g(\oM_{g,n},\QQ), \text{ where }  \\ \nonumber
     \DR_g^\half(a,k)= & \rho^\half_*\left(\sigma_0^![\sigma_{a,k}^{1/2}]\right) \in A^g(\oM_{g,n}^\half,\QQ).
\end{align*}

Conjectural spin analogs of Pixton's class and star graph expressions were proposed in~\cite{CosSauSch}.

\subsubsection{Pixton's spin class} Let   $\Gamma=\left(V,H, g\colon V\to \NN, i\colon H\to H,  H\to V, H^i\cong [\![1,n]\!]\right)$ be a pre-stable graph of genus $g$ with $n$ marked legs (we use the standard notation of~\cite{GraPan}).  

\begin{definition} \label{def:weight}  An $r/k$-{\em weighting} on $\Gamma$ is a function $w\colon H \to \{0,\dots,r-1\}$    
satisfying:
\begin{enumerate}
\item [(i)] For all edges $e=(h,h')$, we have $
w(h)+w(h')\equiv 0\mod r.$
\item [(ii)] For all vertices $v$, we have 
\begin{equation*}
    \sum_{h\mapsto v} w(h) \equiv k(2g(v)-2+n(v)) \mod r.
\end{equation*}
\end{enumerate}
We say that this weighting is {\em compatible with the vector $a$} if $w(i)\equiv a_i\mod r$.  Finally, we say that $w$ is an $r$-weighting if $k=1$.
\end{definition}

We denote by ${\rm Stab}_{g,n}$ the set of stable graphs of genus $g$ with $n$ legs. If $\Gamma$ is a stable graph, then we denote by $W_\Gamma^{2r/k}(a)^{\rm odd}$ the set of odd $2r/k$-weightings (i.e. taking only odd values) compatible with $a$.   For all positive integers $c$, we denote by $\rP_{g}^{ \pm,c,r}(a,k)$ the degree $c$ part of the class
\begin{align*}
   \sum_{\Gamma\in\mathcal{G}_{g,n}}
    \sum_{w\in W^{2r/k}_{\Gamma}(a)^{\rm odd}}&\frac{r^{-h^1(\Gamma)}}{|\Aut(\Gamma)|}
    \zeta_{\Gamma*}\left[\prod_{v\in V(\Gamma)}\exp(-\frac{k^2}{4}\kappa_1(v))\prod_{i=1}^n\exp(\frac{a_i^2}{4}\psi_i)\right.\  \cdot \\
    &\left.\prod_{e=(h,h')}\frac{1-\exp(-\frac{w(h)w(h')}{4})(\psi_h+\psi_{h'}))}{\psi_h+\psi_{h'}} \right]. \\
\end{align*}

In this expression we used standard notation: $\psi_i\in A^1(\oM_{g,n},\QQ)$ is the Chern class of the co-tangent line at the $i$-th marking for all $i\in \{1,\ldots,n\}$,  $\kappa_m=\pi_*(\psi_{n+1}^{m+1})\in A^m(\oM_{g,n},\QQ)$ for all $m\geq 1$ (here $\pi$ is the morphism forgetting the $(n+1)$-st marking), and 
$$
\zeta_\Gamma\colon \oM_{\Gamma}=\prod_{v\in V} \oM_{g(v),n(v)}\to \oM_{g,n}
$$
is the gluing morphism defining the boundary stratum associated to $\Gamma$. The class $\rP_{g}^{\pm,g,r}(a,k)$ is a polynomial in $r$ of degree at most $2c$~\cite{JanPanPixZvo}. The {\em Pixton's spin class} $\rP^\pm_g(a,k)$ is the constant term of this polynomial.
\begin{proposition}\label{spDR = spPix} We have the identity $\DR_{g}^\pm(a,k)=\rP_g^\pm(a,k)$. 
\end{proposition}

Cornalba described degenerations of spin structures in terms of semi-stable curves  (instead of orbifold curves for instance), so this proposition is obtained by direct application of~\cite{BHPSS} to the family of line bundles~\ref{eqn:section on smooth part}. We refer to \ref{ssec:spinpixton} for this computation.\footnote{An alternative proof was outlined  in~\cite[Section 2.3]{CosSauSch} by combining~\cite[Theorem 7]{BHPSS} (or rather the alternative formulation~\cite[Theorem~a]{HolCh} in terms of Chiodo classes), and the computation of ``spin Chido classes'' of~\cite{GKL}. To make this proof rigorous one needs to show that the moduli space of $r$-th root of~\ref{eqn:section on smooth part} is isomorphic to $2r$-root of $\omega_{\rm log}^k(-a_1\sigma_1\ldots)$. Although  natural, the proof of this statement would require one to work with different approaches to roots of line bundles on singular curves.}

\begin{remark} The constant term of the spin Pixton's class $\rP^{\pm}_{g}(a,k)$
is equal to the codimension $g$ part of the constant term of the polynomial
defined in~\cite[Section 2.3]{CosSauSch}. This can be seen by applying
the exact same arguments as in~\cite[Proposition 5]{JanPanPixZvo}
to reduce the contribution of the constant term defined in~\cite{CosSauSch} to simpler terms. 
\end{remark}

\subsubsection{Star graph expression} Let $\Gamma$ be a semi-stable graph. \begin{definition} \label{def:twist} A \emph{twist}  $I$
on a stable graph $\Gamma$ is a function $I\colon H\to \ZZ$ satisfying
\begin{enumerate}
    \item[(i')] For all edges $(h,h')$ we have $I(h)+I(h')=0$
    \item[(ii')] For all  $v,v'$ in $V(\Gamma)$, if two edges $(h_1,h'_1)$ and $(h_2,h'_2)$ connect $v$ to $v'$ then
    $$I(h_1)\geq0\Leftrightarrow I(h_2)\geq0.$$
In which case we denote $v\geq v'$.
\item[(iii')] The relation $\geq$ on the set of vertices of $\Gamma$ is transitive (i.e. it is a partial ordering). 
\end{enumerate}
We say that it is a {\em $k$-twist} if moreover
\begin{enumerate}
    \item[(iv')]  For all vertices $v$, we have 
\begin{equation*}
    \sum_{h\in H(v)} I(h) = k(2g(v)-2+n(v)).
\end{equation*}
\end{enumerate}
We say that it is compatible with a vector $a$ if $I(i)=a_i$ for all $i$.
\end{definition}

We see that the conditions (i) and (ii) in the definition of $r/k$-weights are the reduction modulo $r$ of the definition of the conditions (i') and (iv') in the definition of $k$-twists. 

\smallskip

The pair $(\Gamma,I)$ is called
a \emph{$k$-twisted graph}. If $v$ is a vertex of $\Gamma$, then we denote by $I(v)$ the vector of values of the $I(h)$ for $h$ incident to $v$. If $e=(h,h')$ is an edge, then we denote by $I(e)=|I(h)|$.

\begin{definition}\label{def:simple-star} A \emph{$k$-simple star graph} is a $k$-twisted
graph $(\Gamma,I)$ 
such that: 
\begin{enumerate}
    \item The graph $\Gamma$ is stable, and has a distinguished vertex $v_0$ (the \emph{central} vertex) such that all edges connect $v_0$ to a vertex in $V_{\rm out}:=V\setminus{\{v_0\}}$ (called an \emph{outlying} vertex).
    \item  If  $h$ is an half-edge incident to an outlying vertex then $I(h)\in k\NN$. 
\end{enumerate} 
\end{definition}
We denote by $\text{SStar}_{g}(a,k)$ (respectively $\text{SStar}_{g}^{\rm odd}(a,k)$) the set of
$k$-simple star graphs compatible with $a$ (respectively with odd twist). Then,  the {\em spin star-graph expression} is the class
\begin{align*}
    {\rH}_g^\pm(a,k)=\,\, & \!\!\!\sum_{(\Gamma,I)\in \text{SStar}_{g}^{\rm odd}(a)}\frac{\prod_{e\in E(\Gamma)}I(e)}{k^{\#V_{\rm out}}|\Aut(\Gamma,I)|}\zeta_{\Gamma*}\left([\oM_{\Gamma,I}]^\pm\right), \text{ where} \\  \nonumber  \oM_{\Gamma,I} =\,\, & \oM_{g(v_0)}(I(v_0),k) \times \prod_{v\in V_{\rm out}} \oM_{g(v)}(I(v)/k), \text{ and} \\ \nonumber
     [\oM_{\Gamma,I}]^\pm =\,\, & [\oM_{g(v_0)}(I(v_0),k)]^\pm  \bigotimes_{v\in V_{\rm out}}[\oM_{g(v)}(I(v)/k)]^\pm.
\end{align*}

\begin{theorem}[\cite{CosSauSch}, Conjecture~2.5]\label{thm:DR} If $a$ is not in $(k\NN)^n$, then we have ${\rH}_g^\pm(a,k)=\DR_g^\pm(a,k)$.
\end{theorem}
As explained in~\cite{CosSauSch}, this theorem allows for the computation of integrals of products of $\psi$-classes on $\DR_g^\pm(a,k)$ and $[\oM(a,k)]^\pm$.

\subsubsection{How to use spin DR-cycles} By applying~\ref{thm:DR} and~\ref{spDR = spPix} we obtain the following expression
$$
[\oM_{g}(a,k)]^\pm = \rP_g^\pm(a,k) - \Delta,
$$
where $\Delta$ is a class expressed in terms  of spin classes  of strata for smaller values of $(g,n)$. However,~\ref{thm:DR} only holds for vectors $a$ which are not in $(k\NN)^n$. For non-signed classes, Farkas--Pandharipande explained that we can go around this problem by observing that the residue theorem implies that $\oM_g((a_1,\ldots,a_n,0))$ is empty when all $a_i$ are positive so $[\oM_g((a_1,\ldots,a_n,0))]=0$. In particular the first non-trivial term in ${\rH}_g((2g,0))$ involves $[\oM_g((2g-1))]$. This observation does not apply in the spin setting as $[\oM_g((a_1,\ldots,a_n,0))]^\pm$ is not defined (because 0 is even), but~\ref{thm:DR} has the following corollary.
\begin{lemma}\label{lem:computation1}
    Let $g_0\in \NN$. If all classes $[\oM_{g}(a)]^\pm$ are tautological and computable for $g\leq g_0$ and $a$ positive, then the classes $[\oM_{g}(a,k)]^\pm$ are tautological and computable for $g\leq g_0$.
\end{lemma}
\begin{proof}
    If $a$ is not in $(k\NN)^n$ then we use the star graph expression as explained above. If $a$ is in $(k\NN)$, then the star graph expression cannot be used as it is stated, but we can use the same argument as in~\cite{Sauflat}:  \ref{thm:DR} and the polynomiality of spin Pixton's class (proved in~\cite{SpePolynomiality,PixPolynomiality}) to prove that the following identity holds $$\rP_{g}^\pm(a,k)= -a_1\psi_1 [\oM(a/k)]^\pm + [\overline{\M(a,k)\setminus \M(a/k)}]^\pm   \, +\!\!\!\!\!\!\! \sum_{(\Gamma, I)\in {\rm SStar}_{g,1}^{\rm odd}(a,k)}  \frac{\prod_{e\in E(\Gamma)}I(e)}{k^{\#V_{\rm out}}|\Aut(\Gamma,I)|}\zeta_{\Gamma*}[\oM_{\Gamma,I}]^\pm,
    $$
    where ${\rm SStar}_{g,1}^{\rm odd}(a,k)\subset {\rm SStar}_{g}^{\rm odd}(a,k)$ is the subset where the first leg is on the central vertex. The second term in the right-hand side of this expression is the class of the closure of $k$-canonical divisor which are not obtained from a $1$-canonical divisor multiplied by $k$. Using this expression, we see again that the class $[\overline{\M(a,k)\setminus \M(a/k)}]^\pm$ can be expressed in terms of spin classes of strata of holomorphic differentials. 
\end{proof}

\subsection{Cone of spin sections}\label{sec:introspinsection} The DR cycle approach has reduced the problem to computing spin classes of strata of holomorphic differentials. To do so we will follow the approach of~\cite{Sau}. The {\em Hodge bundle} $\oH_{g,n}\to \oM_{g,n}$ is the vector bundle of holomorphic differentials and we denote by $p\colon \PP\oH_{g,n}\to \oM_{g,n}$ its projectivization. If $a\in \ZZ_{>0}^n$, then we can consider the {\em degeneracy locus} $\PP\H_g(a)\subset \PP\oH_{g,n}$ of differentials on smooth curves and with order $a_i-1$ at the marking $x_i$ for all $i\in \{1,\ldots,n\}$. Note that in this definition,  we do not assume that $|a|=2g-2+n$. The main formula used to compute the classes of closures of degeneracy loci has the shape
$$
[\PP\oH_g(a)]=(\xi+(a_i-1)\psi_i) [\PP\oH_g(a_1,\ldots,a_i-1,\ldots)] - \text{boundary terms},
$$
where $\xi=c_1(\mathcal{O}(1))$. The boundary terms in this formula are expressed in terms of strata of meromorphic differentials with residue conditions in either lower genus or lower number of markings. Therefore, the computation of classes of the form $p_*\left(\xi^\ell[\PP\oH_{g}(a)]\right)$ can be performed by induction on the size of $a$ (and the genus, and number of markings). The base of this inductive computation are the $\lambda$-classes defined as
$$
\lambda_c \coloneqq (-1)^{c} s_{c}(\oH_{g,n}) = (-1)^{c}  p_*\left(\xi^{c+g-1}[\PP\oH_{g}(1,\ldots,1)]\right).
$$
These classes are tautological and computable as was shown by Mumford in his seminal introduction to tautological calculus~\cite{Mum} (or alternatively by~\cite{PolSau}). 

\smallskip

In order to adapt this approach to spin classes,  we consider the sub-cone $\mathcal{SQ}_g(a)\subset \H_g(a)$ 
of {\em squares}, i.e. differentials with even zeros. This cone is partitioned into two components according to parity, and we denote by $[\PP\overline{\mathcal{SQ}}_g(a)]^\pm$ the class of the difference between the closures of the projectivizations of these components.  
With this notation we have
$$
2[\PP\overline{\mathcal{SQ}}_g(a)]^\pm =(\xi+(a_i-2)\psi_i) \, [\PP\overline{\mathcal{SQ}}_g(a_1,\ldots,a_i-2,\ldots)]^\pm + \text{boundary terms}
$$
where the boundary terms are described in terms of spin classes of strata in lower genera or number of markings (see~\ref{ssec:intersectionformulas}). Therefore, we can reduce the computation of classes of holomorphic strata to the computation the classes 
$$
s_{g}^{\pm,c} = \nonumber p_*\left(\xi^c [\PP\overline{\mathcal{SQ}}_g(1,\ldots,1)]^\pm\right) \in A^c(\oM_{g,n},\QQ)\footnote{As for Hodge classes, we did not include the number of markings in this notation as $\pi^*s_{g}^{\pm,c }=s^{\pm,c}_{g}$ if $\pi$ is the forgetful morphism of a marking.}.$$
 To express these classes, we consider the set ${\rm Tree}_{g,n}\subset {\rm Stab}_{g,n}$ of stable graphs with no loops, and its subset ${\rm Tree}_{g,n}^\star\subset {\rm Tree}_{g,n}$ of graphs with no genus 0 vertices. For all $t\in \CC$, we denote 
\begin{align}
\label{for:notations} s_{g}^{\pm}(t)=\,\, & \sum_{c\geq 0} t^c s_{g}^{\pm,c}, \text{ and}   \\
L_{g}(t)=\,\, &2^{g-1}\Lambda(2t)\Lambda(-t)-2^{2g-1}, 
\end{align}
where  $\Lambda(t)=1+t\lambda_1+t^2\lambda_2+\ldots$. The following theorem expresses the classes $s_g^\pm$ in term of the $L$-classes,  which are tautological and computable because  $\lambda$-classes are tautological and computable. 

\begin{theorem}\label{thm:segre}
    For all $g$ and $n$ we have 
    \begin{align}
        \label{for:mainseg} L_g(t)=&\sum_{\Gamma \in {\rm Tree}^\star_{g,n}} \frac{t^{|E(\Gamma)|}}{|{\rm Aut}(\Gamma)|}\zeta_{\Gamma *}\left(\bigotimes_{v\in V} s_{g(v)}(t) \right), \text{ and}\\
        \label{for:mainseg2} s_g(t)=&\sum_{\Gamma \in {\rm Tree}^\star_{g,n}} \frac{(-t)^{|E(\Gamma)|}}{|{\rm Aut}(\Gamma)|}\zeta_{\Gamma *}\left(\bigotimes_{v\in V} L_{g(v)}(t) \right).
    \end{align}
\end{theorem}
This theorem completes the proof of~\ref{thm:main} outlined in this introduction. The proof of formula~\ref{for:mainseg} will be performed by expressing both sides as the Segre classes of the cone over $\oM_{g,n}^\half$ of  sections  of the spin structure. The presence of the quadratic Hodge symbol in the left-hand side of this formula is due to the identity
$$\epsilon_*\left([\pm]\cdot s_*\left(R^*\pi_*\mathcal{L}\right)\right)= 2^{g-1}\Lambda(-1) \Lambda(1/2)$$
that was already used in the expression of the spin Hurwitz numbers and spin Gromov--Witten invariants of $\PP^1$~\cite{GKL,GKLSau}. Finally, formula~\ref{for:mainseg2} will be derived from formula~\ref{for:mainseg} by performing a Möbius inversion.

\subsection*{Acknowledgements} The authors would like to thank Francesca Carocci, Alessandro Chiodo, Matteo Costantini, Alessandro Giacchetto, Martin Möller, Rahul Pandharipande, Johannes Schmitt, Pim Spelier, Thibault Poiret, and Anton Zorich for several discussions on strata of differentials and spin parity over the years. We are particularly grateful to Francesca Carocci for pointing out a mistake in section 3 of an earlier draft. 

A.S. was supported by the 
ERC Starting Grant  {\em SpiCE} (no. 101164820). D.H. and G.P. were supported by NWO grant VI.Vidi.193.006. The project was also supported by the Franco-Dutch program Hubert Curien - Van Gogh.

\section{Moduli of spin structures and spin classes}

In this section, we summarize the definition of \emph{spin curves}
and their moduli constructed by Cornalba in~\cite{Cornalba1989ModuliOC}. We recall the description of the stratification of the boundary indexed by {\em 2-weighted} graphs and prove a vanishing statement for the classes of strata with at least one 0 weight (see~\ref{prop:cancellation}). We could not find a proper proof of this ``folklore'' proposition in the literature, although the argument is essentially borrowed from Cornalba's original paper and used in the computations of~\cite{GKL}.  Finally, we compute Pixton's classes
constructed from the universal spin structure. 

\subsection{Moduli space of spin structures}
\begin{definition} A pre-stable curve  of genus $g$ with $n$ 
marked points over an algebraically closed field $(\mathcal{C},p_1,\dots,p_n)$
is \emph{quasi-stable} if all unstable components $E_1,\dots, E_{t}$
\begin{enumerate}
    \item [(i)] are smooth, rational, and disjoint, that is, if $E_i$ and $E_j$ are 
    unstable, then $E_i\cap E_j=\emptyset$, and
    \item [(ii)]  meet the rest of the curve in exactly \emph{two} 
    points. 
\end{enumerate}
The unstable components of $\mathcal{C}$ 
will be referred to as \emph{exceptional}. 
Note that, by definition, the marked points $p_1,\dots,p_n$ lie on non-exceptional components. 
The sub-curve 
\begin{equation*}
    \widetilde{\cC}:= \overline{\mathcal{C}\setminus{\bigcup}_{i=1}^tE_i}
\end{equation*}
will be called the \emph{non-exceptional sub-curve}. In the definition
of $\widetilde{\cC}$ above the upper bar denotes the topological closure.
Finally, a quasi-stable curve of genus $g$
with $n$ marked points $(\mathcal{C},p_1,\dots p_n)$ 
is called a \emph{quasi-stable model} of $(C,p_1,\dots,p_n)$ 
if $C$ is the stabilization of $\mathcal{C}$.
\end{definition}

\begin{definition}
    A line bundle $\cL$ on a quasi-stable curve $\mathcal{C}$ is 
    \emph{admissible} if the restriction to any exceptional
    component has degree $1$.
\end{definition}

\begin{definition}\label{def:spstr} A \emph{spin structure} on a stable curve $C$ is a
triple $(\mathcal{C}, \cL, \phi)$ such that:
\begin{enumerate}
    \item [(i)] $\mathcal{C}$ is a quasi-stable model of $C$,

    \item [(ii)] $\cL$ is an admissible line bundle on $\mathcal{C}$
    of total degree $g-1$,

    \item [(iii)] $\phi\colon \cL^{\otimes2}\to \omega_{\mathcal{C}}$ is a morphism 
    of sheaves which is generically non-zero on non-exceptional 
    components.
\end{enumerate}

\end{definition}

\begin{remark}\label{honroot} 
The fact that $\cL$ is admissible implies, by degree considerations,
that $\phi$ vanishes on exceptional components.
The morphism $\phi$ induces a morphism 
\begin{equation*}
    \widetilde{\phi}\colon \cL^{\otimes2}|_{\widetilde{\cC}}
    \to \omega_{\widetilde{\cC}},
\end{equation*}
which, by conditions $(ii)$ and $(iii)$, is an isomorphism (see~\cite[beginning of Section 2]{Cornalba1989ModuliOC}). In particular, if 
$\mathcal{C}$ is smooth, a spin structure coincides with a \emph{theta characteristic} in the 
sense of~\cite{Mum1}. Furthermore, let $n\colon \widetilde{\cC}'\to \widetilde{\mathcal{C}}$
be the normalization morphism of the non-exceptional sub-curve. Since $\widetilde{\phi}$ is an isomorphism, 
pulling back $\cL|_{\widetilde{\cC}}$ via the normalization morphism and restricting to an irreducible component $Y\subset\widetilde{\cC}$ yields a square root of 
$n^*\omega_{\widetilde{\mathcal{C}}}|_{Y}$.
 In particular,
a necessary condition for $\mathcal{C}$ to be the support 
of a spin structure on a stable curve is that 
the degree of the restriction of $\omega_{\widetilde{\cC}}$ 
to every irreducible component must be even (see~\cite[Section 3]{Cornalba1989ModuliOC}).
\end{remark}

\begin{remark}\label{rem:exc-parity} Given a global section $\widetilde{s}$ of the 
restriction $\cL|_{\widetilde{\cC}}$, one can uniquely extend it over
the exceptional components to obtain a global section $s$ of
$\cL$. Indeed, let $E$ be any exceptional component. Then, since $E\cong \PP^{1}$ and $\cL$ is admissible, we obtain that 
$\cL|_{E}\cong \cO_{\PP^{1}}(1)$. Furthermore, as there is a unique section of 
$\cO_{\mathbb{P}^{1}}(1)$ 
taking any given two values at 
$0$ and $\infty$, we can extend $\widetilde{s}$ over each 
exceptional component by the unique section of $\cO_{\mathbb{P}^1}(1)$ 
which takes the values of the section 
$\widetilde{s}$ at the corresponding pre-images of exceptional nodes. 
Conversely, given a global section of $\cL$, we can always restrict
it to the non-exceptional sub-curve $\widetilde{\cC}$. In this way,
we obtain an isomorphism
\begin{equation*}
    H^0(\mathcal{C},\cL)\cong H^0(\widetilde{\cC},\cL|_{\widetilde{\cC}}).
\end{equation*}
In particular, the dimension of sections of a spin structure  
$(\mathcal{C},\cL,\phi)$ on a stable curve $C$
is equal to that of the space of sections of $\cL|_{\widetilde{\cC}}$.
\end{remark}

\begin{definition}Let 
$(\pi\colon C\to S,p_1.\dots,p_n\colon S\to C)$
be a family of $n$-marked stable curves of genus $g$ over a base scheme $S$,
and let $a\in\ZZ^n$ a vector of integers such that 
\begin{equation*}
    \left(2g-2+n-\sum_{i=1}^{n} a_i\right)\equiv0\mod2.
\end{equation*} 
A family of \emph{square roots of}
$\omega_{{\rm log},\pi}(-\sum_{i=1}^na_ip_i(S))$ is a triple 
$(\xymatrix{
\mathcal{C} \ar@{->}[r]^{\overline{\pi}} & S \ar@/^/@{->}[l]^{{p_1,\dots,p_n}}
}, \cL, \phi)$,
where 
\begin{enumerate}
    \item [(i)] $\xymatrix{
\mathcal{C} \ar@{->}[r]^{\overline{\pi}} & S \ar@/^/@{->}[l]^{{p_1,\dots,p_n}}
}$ 
is a family of quasi-stable curves such that, for all geometric points $s\in S$,
the fiber $(\mathcal{C}_{s},p_1(s),\dots,p_n(s))$ is a quasi-stable model of $(C_{s},p_1(s),\dots,p_{n}(s))$,
    \item [(ii)] $\cL$ is a line bundle on $\mathcal{C}$ such that, for all geometric points $s\in S$, the fiber $\cL_{s}$ is an admissible line  bundle on 
    $\mathcal{C}_{s}$ of total degree 
    $$\frac{2g-2+n-\sum_{i=1}^{n} a_i}{2}$$,
    \item [(iii)] 
    $\phi\colon \cL^{\otimes2}\to \omega_{{\rm log},\overline{\pi}}
(-\sum_{i=1}^na_ip_i(S))$ 
is a morphism of sheaves on $\mathcal{C}$, such that for
all geometric points $s\in S$, the induced morphism 
\begin{equation*}
    \phi_{s}\colon \cL_{s}^{\otimes 2}\to \omega_{{\rm log},\mathcal{C}_{s}}
(-\sum_{i=1}^na_ip_i(S))
\end{equation*}
is generically non-zero on non-exceptional components of $\mathcal{C}_{s}$.
\end{enumerate}
\end{definition}
\begin{remark}
Let $C\to S$ be a stable curve over an algebraically closed field.
The definition above generalizes~\ref{def:spstr}.
In particular, if $a=(1,\dots,1)$, 
the triple $(\mathcal{C},\cL,\phi)$ in the definition above
recovers the definition of a spin structure on $C\to S$. 
However, since many computations to follow will involve
spin structures, we include both definitions for clarity.
\end{remark}

\begin{theorem}~\cite{Cornalba1989ModuliOC} 
The moduli spaces $\oM_{g,n}^{1/2}$ of spin structures, 
 and more generally $\overline{\cM}^{1/2,a}_{g,n}$ (of square roots of $\omega_{\rm log}(-\sum_{i=1}^na_ip_i)$), on 
stable curves
are smooth Deligne-Mumford stacks over ${\rm Spec}(\CC)$ such that the boundary is a normal-crossings divisor. Furthermore, 
forgetting the root and stabilizing yields a finite flat morphism 
$\epsilon\colon\overline{\cM}^{1/2,a}_{g,n}\to\overline{\cM}_{g,n}$ 
of degree $2^{2g-1}$. 
\end{theorem}

\subsection{Stratification of $\oM_{g,n}^{1/2}$}\label{subsec:strat-of-roots}
Let $(\Gamma,w)$ be a stable graph together
with a $2$-weighting (see~\ref{def:weight}). For each edge $e=(h,h')\in E(\Gamma)$, we define:
\begin{equation*}
    w(e):=w(h)=w(h').
\end{equation*}
We then define $\Gamma_w$ to be the unique quasi-stable graph
obtained by inserting a vertex of genus $0$ on every edge $e=(h,h')$ 
such that $w(h)=1$. The stabilization 
$\Gamma$ of $\Gamma_{w}$, i.e., contracting all unstable vertices of
$\Gamma_{w}$, induces a morphism $\Gamma_{w}\to\Gamma$.
In this way, the sets $V(\Gamma)$ and $H(\Gamma)$ can be seen as subsets of 
$V(\Gamma_w)$ and $H(\Gamma_w)$. 

The collection of $2$-weightings encodes all 
possible ways of inserting exceptional components at nodes of
a stable curve $C$ such that the resulting quasi-stable curve 
can be the support of 
a spin structure on $C$. To see this, first recall that,
by~\ref{honroot},  
a quasi-stable curve $\cC$ is the support of a spin structure on 
$C$ only if the multi-degree of $\omega_{\widetilde{\cC}}$ is even on all
stable components.
Now suppose that we are 
given a stable curve with associated dual graph $\Gamma$, a $2$-weighting $w$ on $\Gamma$, and a quasi-stable curve 
$\mathcal{C}$ whose dual graph is $\Gamma_{w}$. 
Then the multi-degree of $\omega_{\widetilde{\mathcal{C}}}$ on a vertex 
$v\in V(\Gamma)\subseteq V(\Gamma_{w})$ is given by
\begin{equation*}
    2g(v)-2+n(v)-\sum_{h\to v}w(h),
\end{equation*}
which by definition of the $2$-weighting is an even number, and which,
by the above, implies that $\cC$ can be the support of a spin structure on $C$.

We recall that the stable graph $\Gamma$ determines a stratum 
$\zeta_\Gamma\colon \oM_{\Gamma}\to \oM_{g,n}$. By the discussion above,
we see that pulling back $\zeta_{\Gamma}$ 
along the morphism $\epsilon\colon \oM_{g,n}^{1/2}\to \oM_{g,n}$
yields a disjoint union of spaces 
\begin{equation*}
    \xymatrix{ \coprod_{w}\overline{\cP}_{\Gamma,w}\ar[d]
    \ar[r]  & \overline{\cM}^{1/2}_{g,n}\ar[d]^{\epsilon} 
    \\ \overline{\cM}_{\Gamma}\ar[r]^{\zeta_{\Gamma}}  & 
  \overline{\cM}_{g,n} }
\end{equation*}
\noindent indexed by $2$-weightings on the stable graph $\Gamma$. Restricting to a fixed $2$-weighting $w$, we obtain a gluing morphism 
\begin{equation*}
    \zeta_{\Gamma,w}\colon \overline{\cP}_{\Gamma,w}\to \oM_{g,n}^{1/2}.
\end{equation*}
By definition of the pullback diagram of the spaces above, we see that
the space $\overline{\cP}_{\Gamma,w}$ parametrizes points in $\oM_{g,n}^{1/2}$
whose stabilization is given by the gluing
of a point in $\oM_{\Gamma}$, along with an identification of
the half edges corresponding to $\Gamma$. We refer to the image of $\zeta_{\Gamma,w}$
as the \emph{boundary stratum} associated to $(\Gamma,w)$. We define the space
\begin{equation*}
    \overline{\cM}_{\Gamma,w}:=  \prod_{v\in V}
    \overline{\cM}^{1/2,w(v)}_{g(v),n(v)},
\end{equation*}
where 
$w(v):=(w(h))_{h\in H(v)}$. Unlike the case of $\oM_{g,n}$, the space 
$\overline{\cP}_{\Gamma,w}$ is not isomorphic to $\oM_{\Gamma,w}$.
However, there exists a morphism  
\begin{equation*}
    \nu \colon \overline{\cP}_{\Gamma,w}\to \oM_{\Gamma,w}
\end{equation*}
given by normalizing all nodes corresponding to edges of $\Gamma$
of a point in $\overline{\cP}_{\Gamma,w}$ and pulling back the spin structure 
to the partial normalization. In general, there are two possible ways of gluing
line bundles $(\cL_{v})_{v\in V(\Gamma)}$ of points of $\oM_{\Gamma,w}$
on exceptional and non-exceptional nodes contained in a cycle of $\Gamma$,
such that the resulting object lies in
$\overline{\cP}_{\Gamma,w}$ (see~\cite[Section 1.3]{CapCas}). 
Thus, it would be tempting to 
argue that $\nu\colon \overline{\cP}_{\Gamma,w}\to \oM_{\Gamma,w}$
is a torsor under the group $H^1(\Gamma,\ZZ/2\ZZ)$. However, $\nu$
fails to endow $\overline{\cP}_{\Gamma,w}$ with a torsor structure over $\oM_{\Gamma,w}$ 
due to ramification. This ramification occurs precisely when there exists an edge 
$e\in E(\Gamma)$ contained in a cycle of $\Gamma$ such that $w(e)=1$. In such cases, the
two different choices of gluing along these nodes yield isomorphic spin structures. 

\begin{proposition}\label{degree:nu} The morphism $\nu \colon \overline{\cP}_{\Gamma,w}\to \oM_{\Gamma,w}$ is finite and flat of degree
$2^{h^{1}(\Gamma)+|V(\Gamma)|-1}$.
\end{proposition}
\begin{proof} The authors in~\cite[Proposition 5]{CapCas} compute
the degree of $\nu$ at the level of coarse spaces. 
In particular, they show that in the fiber of $\nu$ there are $2^{h^{1}(\widetilde{\Gamma})}$ points, where $\widetilde{\Gamma}$ is the 
graph obtained from $\Gamma$ by removing all edges $e$ such that $w(e)=1$.
Additionally, they show that each point in the fiber of $\nu$ has
multiplicity $2^{h^{1}(\Gamma)-h^{1}(\widetilde{\Gamma})}$, i.e., the degree
of $\nu$ is $2^{h^{1}(\Gamma)}$ on the level of coarse spaces.
However, a general 
point in the domain and codomain of $\nu$ has $2$ and $2^{|V(\Gamma)|}$ 
automorphisms, respectively, given
by multiplication by $\mu_2$ in the fibers of the spin structures. 
\end{proof}

\begin{proposition}\label{prop:cancellation} Let $(\Gamma,w)$ be a pair consisting of a stable graph $\Gamma$
and a $2$-weighting $w$
admitting an even value on some half edge $h\in H(\Gamma)$. Then 
\begin{equation*}
    \nu_*(\zeta_{\Gamma,w}^*[\pm])=0
\end{equation*}
in $A^*(\oM_{\Gamma,w})$. 
In particular, we have
\begin{equation*}
    \epsilon_*(\zeta_{\Gamma,w*}(\phi\circ \nu)^*(\beta)\cdot[\pm])=0
\end{equation*}
in $A^*(\oM_{g,n})$, where $\phi\colon\oM_{\Gamma,w}\to \oM_{\Gamma}$
and $\beta\in A^*(\oM_{\Gamma})$.
\end{proposition}

\begin{proof} We will make use of a well-known argument found 
in~\cite[Example 6.2]{Cornalba1989ModuliOC}. In order to prove the first assertion, 
we see from the proposition above 
that it suffices to show that in each fiber of $\nu$ there 
is an equal number of even and odd spin structures. 
Since the parity is constant in families, it is enough to argue on 
geometric points.   

First, we fix some notation. Let $e_0=(h,h')$
be the edge of $\Gamma$ in the assumption: i.e., $w(h)=0$.
We note that since $w(h)=0$, then by degree considerations 
$e_0$ necessarily corresponds 
to a non-separating edge. Let
$(\cC_v,\cL_v)_{v\in V(\Gamma)}$ be a geometric point of $\oM_{\Gamma,w}$
and $\cC_{e_0}$ denote the curve obtained by gluing all non-exceptional 
nodes except for the one corresponding to $e_0$. Note that
a priori, $\cC_{e_0}$
may be disconnected. However, since $e_0$ corresponds to a non-separating node
we may restrict to the connected component of $\cC_{e_0}$, 
say $\widetilde{\cC}_{e_0}$, carrying the markings corresponding to the half edges $(h,h')$ of $e_0$. 

Suppose that we fix a choice of gluings
for all non-exceptional nodes except for $e_0$ also for 
the line bundles $\cL_v$. 
Then we obtain a line bundle $\cL_{e_0}$ on $\widetilde{\cC}_{e_0}$ such that 
    \begin{equation*}
        \cL_{e_0}^{\otimes2}\cong \omega_{\widetilde{\cC}_{e_0}}(p_{h}+p_{h'}).
    \end{equation*}
An application of Riemann-Roch shows that 
$\cL_{e_0}$ admits a section $s$ which does not vanish simultaneously 
at $p_h$ and $p_{h'}$. 
This yields two possible identifications
of the node for $\cL_{e_0}$ with opposite parities. More 
precisely, the two identifications are given by 
\begin{equation*}
    s(p_h)\mapsto s(p_{h'})\ \ \text{and}\ \ s(p_{h})\mapsto -s(p_{h'}).
\end{equation*}
In the former the section $s$ descends to a section of the spin structure whereas in the latter it does not, thus the resulting line bundles
will have opposite parities. Note that, up to this point, we have obtained
a line bundle only on the non-exceptional sub-curve of a point in the fiber 
of $\nu$. However, different 
choices of gluings on the remaining exceptional nodes will
neither change the isomorphism class of the resulting spin structure $(\mathcal{C},\cL,\phi)$
nor change the parity (see~\ref{rem:exc-parity}).
\end{proof}

In other words, the preceding lemma implies that if we fix a stable graph 
$\Gamma$ and a $2$-weighting $w$ admitting an even value, then
the \enquote{degree} of the class $[\pm]$ along the stratum corresponding to 
$(\Gamma,w)$ is $0$.
However, this is not true in the complementary case: when $w$ admits
only odd values. In that situation, the computation of the degree along the stratum 
$(\Gamma,w)$ appears in the proof of~\cite[Proposition 9.21]{GKL} 
in a more general context. For completeness, we recall this 
computation here, specialized to the context relevant to our setting.

\begin{proposition}\label{degree:even-odd} Let $(\Gamma,w)$ be a pair of stable 
graph and a $2$-weighting admitting only odd values. 
Then, using the notation of~\ref{prop:cancellation} 
\begin{equation*}
    \epsilon_*(\zeta_{\Gamma,w*}(\phi\circ \nu)^*(\beta)\cdot[\pm])=2^{g-1}
    \zeta_{\Gamma*}(\beta)
\end{equation*}
in $A^*(\oM_{g,n})$, where $\beta\in A^*(\oM_{\Gamma})$.
\end{proposition}
\begin{proof} First, by~\ref{rem:exc-parity}, a 
geometric point $(\mathcal{C},\cL,\phi)$ lying in the stratum $(\Gamma,w)$, 
where $w$ takes only odd values, satisfies
\begin{equation*}
    H^0(\mathcal{C},\cL)\cong\bigoplus_{v\in V(\Gamma)} H^0(\cC_v,\cL_v).
\end{equation*}
In particular, the parity of $\cL$ depends only on the parity of each stable 
component of $\mathcal{C}$ corresponding to vertices $v\in V(\Gamma)$.
This implies that the cycle $\zeta_{\Gamma,w}^*([\pm])$ is the pullback of the product of
the parity cycles on each component of $\oM_{\Gamma,w}$ via $\nu$. We denote this cycle
by $[\pm]_{V}\in A^0(\oM_{\Gamma,w})$.
Furthermore, for each vertex $v\in V(\Gamma)$ the difference in the number  of even and 
odd spin structures on a stable curve is $2^{g(v)}$. Taking into 
account the fact that a general point in
$\oM_{\Gamma,w}$ has $2^{|V(\Gamma)|}$ automorphisms, given by multiplication by $\mu_2$
on the fibers of $\cL_{v}$ for each $v\in V(\Gamma)$, we obtain that
\begin{equation}\label{eq:proof-even-odd}
    \phi_{*}([\pm]_{V})=2^{\sum_{v}(g(v)-1)}=2^{g-h^{1}(\Gamma)-|V(\Gamma)|}.
\end{equation}
Consider
the following diagram
\begin{equation}\label{diagram:degree-proof}
    \xymatrix{\mathcal{P}_{\Gamma,w}\ar[r]^{\zeta_{\Gamma,w}}\ar[d]_{\nu} & \oM_{g,n}^{1/2}\ar[dd]^{\epsilon} \\
    \oM_{\Gamma,w}\ar[d]_{\phi} & \\
    \oM_{\Gamma} \ar[r]_{\zeta_{\Gamma}} & \oM_{g,n.}}
\end{equation}
Using the above, together with~\ref{degree:nu}, we compute:
\begin{align*}
    \epsilon_*\left(\zeta_{\Gamma,w*}(\phi\circ\nu)^*(\beta)\cdot[\pm]\right)&=
    \epsilon_*\left(\zeta_{\Gamma,w*}\left((\phi\circ\nu)^*\beta\cdot
    \zeta_{\Gamma,w}^*[\pm]\right)\right) &&(\text{projection formula on $\zeta_{\Gamma,w}$}) \\
    &=\zeta_{\Gamma*}\left(\phi\circ\nu_{*}\left(\nu^{*}(\phi^{*}\beta\cdot[\pm]_{V})\right)\right)
    &&(\text{commutativity of~\ref{diagram:degree-proof}})\\
    &=2^{h^1(\Gamma)+|V(\Gamma)|-1}\zeta_{\Gamma*}\left(\phi_{*}(\phi^{*}\beta\cdot[\pm]_{V})\right) &&(\text{\ref{degree:nu}})\\
    &=2^{h^1(\Gamma)+|V(\Gamma)|-1}2^{g-h^1(\Gamma)-|V(\Gamma)|}\zeta_{\Gamma *}(\beta) 
    &&(\text{Equation~\ref{eq:proof-even-odd}})\\
    &=2^{g-1}\zeta_{\Gamma*}(\beta).
\end{align*}
\end{proof}

\subsection{Pixton's spin class}\label{ssec:spinpixton} 
Let $d\in\ZZ$, and let $a\in \ZZ^{n}$ be a vector of integers such that $|a|=d$. Following~\cite{BHPSS}, we denote by $\mathfrak{Pic}_{g,n}$ the Picard stack
over the moduli space of pre-stable curves
$\mathfrak{M}_{g,n}$, and by $\mathfrak{Pic}_{g,n,d}$ its connected component, parametrizing degree $d$ line bundles on pre-stable curves. We consider the Abel-Jacobi section
\begin{align*}
    \sigma\colon \mathfrak{M}_{g,n}&\to \mathfrak{Pic}_{g,n,d} \\
                    (C,p_1,\dots,p_{n})&\mapsto \cO_{C}\left(\sum_{i=1}^{n}a_{i}p_{i}\right).
\end{align*}
In~\cite[Section 3]{BHPSS}, the 
authors construct an operational Chow class
called the \emph{universal double ramification cycle}
denoted by $\DR_{g,a}^{\rm op}$ corresponding to the class
of the closure
of the image of $\sigma$. One of the main results
in \emph{loc.cit.} is an expression for $\DR_{g,a}^{\rm op}$ in $A^*_{\rm op}(\mathfrak{Pic}_{g,n})$. 
In order to state this result we introduce several operational
classes on $\mathfrak{Pic}_{g,n}$. First, 
we denote  by
\begin{equation*}
    p\colon\mathfrak{C}_{g,n}\to \mathfrak{Pic}_{g,n}, \ \ \mathfrak{L}\to \mathfrak{C}_{g,n}
\end{equation*}
the universal curve over $\mathfrak{Pic}_{g,n}$ and the universal line bundle on $\mathfrak{C}_{g,n}$, respectively.
We define 
\begin{equation*}
    \eta:=p_*\left(c_1(\mathfrak{L})^2\right).
\end{equation*}
To define the remaining classes of interest, we recall the notion of \emph{pre-stable graphs of degree $d$} 
(see~\cite[Section 0.3.2]{BHPSS}), denoted by $(\Gamma,\delta)$.
Here $\Gamma$ is a pre-stable graph of type $(g,n)$,
and $\delta\colon V(\Gamma)\to \ZZ$ is
a multi-degree function of degree $d$, i.e., $\delta$ satisfies
\begin{equation*}
    \sum_{v\in V(\Gamma)}\delta(v)=d.
\end{equation*}
We denote by $\mathcal{G}^{pst}_{g,n,d}$ the set of such pairs. 
There exist proper representable gluing maps for pairs $(\Gamma,\delta)$
\begin{equation*}
    j_{(\Gamma,\delta)}\colon\mathfrak{Pic}_{(\Gamma,\delta)}\to\mathfrak{Pic}_{g,n},
\end{equation*}
where $\mathfrak{Pic}_{(\Gamma,\delta)}$ parameterizes
curves with degenerations forced by $\Gamma$
and with line bundles which have degree $\delta(v)$ when restricted to the components corresponding to the vertex 
$v\in V(\Gamma)$. 
Although $\mathfrak{Pic}_{(\Gamma,\delta)}$ is not isomorphic to $\prod_{v\in V(\Gamma)}\mathfrak{Pic}_{g(v),n(v),\delta(v)}$, it admits a natural map 
\begin{equation*}
    \mathfrak{Pic}_{(\Gamma,\delta)}\to \prod_{v\in V(\Gamma)}\mathfrak{Pic}_{g(v),n(v),\delta(v)},
\end{equation*}
defined by partially normalizing the curve along the nodes of $\Gamma$
and pulling back the line bundle to the normalization. This morphism
is a $\mathbb{G}^{h^{1}(\Gamma)}_m$-torsor. 
This allows us to pullback
classes from the codomain of the map above and push them forward to $\mathfrak{Pic}_{g,n}$ via the proper morphism $j_{(\Gamma,\delta)}$.

\begin{definition}\label{def:weights-for-delta} Let $g,n,d\in\ZZ_{\geq0}$, and let $a\in\ZZ^{n}$ 
be a vector of integers such that $|a|=d$. 
Let $(\Gamma,\delta)$ be a pre-stable graph of type $(g,n)$
and of degree $d$. A weighting
\emph{modulo $r$ for $\delta$} compatible with $a$ is a function $u\colon H(\Gamma)\to \{0,\dots,r-1\}$ such that
\begin{enumerate}
    \item [(i)] For all edges $e=(h,h')\in E(\Gamma)$ we have
    \begin{equation*}
        u(h)+u(h')\equiv0\mod r.
    \end{equation*}
    \item [(ii)] For all $\ell_i\in L(\Gamma)$ we have 
    \begin{equation*}
        u(\ell_i)\equiv a_i\mod r.
    \end{equation*} 
    \item [(iii)] For all vertices we have 
    \begin{equation*}
        \sum_{h\in H(v)}u(h)\equiv\delta(v)\mod r.
    \end{equation*}
\end{enumerate}
We write $U_{\Gamma,\delta}^r(a)$ for the set of weightings modulo $r$
for $\delta$ compatible with $a$.
\end{definition}

\begin{definition}\label{op:Pix}Let $a\in\ZZ^n$ such that $|a|=d$, and let $\mathfrak{L}_{d}$ 
denote the universal line bundle of $\mathfrak{Pic}_{g,n,d}$. 
For all positive integers $r,c$, we 
denote by $\mathsf{P}^{c,r}_{g,a}$ the codimension 
$c$ part of the following class
\begin{equation*}
    \exp\left(-\frac{1}{2}\eta_{a}\right)\sum_{\substack{(\Gamma,\delta)\in
    \mathcal{G}^{pst}_{g,n,d}\\ u\in U^r_{\Gamma,\delta}(a)}}\frac{r^{-h^{1}
    (\Gamma)}}{|\Aut(\Gamma,\delta)|}j_{(\Gamma,\delta)*}
    \left[\prod_{e=(h,h')}\frac{1-\exp\left(\frac{-u(h)u(h')}{2}
    (\psi_h+\psi_{h'})\right)}{\psi_h+\psi_{h'}}\right]
\end{equation*}
in $A^{c}(\mathfrak{Pic}_{g,n,d})$, 
where 
\begin{equation*}
    \eta_{a}:=p_{*}\left(c_1\left(\mathfrak{L}_{d}(-\sum_{i=1}^na_ip_i)\right)^2\right).
\end{equation*}
The class $\mathsf{P}^{c,r}_{g,a}$ is polynomial in $r$ for $r$
sufficiently large. We denote by $\mathsf{P}^{c}_{g,a}$
the constant term of this polynomial.
\end{definition}
\begin{theorem}~\cite[Theorem 7]{BHPSS}\label{pix-Pic-stack} Let $a\in\ZZ^{n}$ be a vector
of integers such that $|a|=d$. Then we have 
\begin{equation*}
    \DR^{\rm op}_{g,a}=\mathsf{P}^{g}_{g,a}.
\end{equation*}
in $A^*_{\rm op}(\mathfrak{Pic}_{g,n,d})$.
\end{theorem}
\begin{remark} The general Pixton's formula in 
$\mathfrak{Pic}_{g,n,d}$ presented in~\cite[Section 0.3]{BHPSS}
appears with a different shape. In particular, our formula
is the pullback of the degree $0$ Pixton's formula 
(see~\cite[Section 0.7]{BHPSS}) under the translation morphism
\begin{equation*}
    \tau_{a}\colon\mathfrak{Pic}_{g,n,d}\to \mathfrak{Pic}_{g,\emptyset,0}
\end{equation*}
described in the same section of \emph{loc.cit.}.
We choose this form in order
to simplify the notation in our computations later.
\end{remark}

Throughout the rest of this section, we fix an odd integer 
$k\in\ZZ_{\geq1}$, and a vector of odd integers $a\in\ZZ^n$ such that 
$|a|=k(2g-2+n)$. Moreover, we define $\widetilde{a}\in\ZZ^n$ 
to be the vector of integers whose $i$-th entry is given by 
$\widetilde{a}_i:=(a_i-1)/2$.
We denote by
\begin{equation*}
    \cL\to \oC^{1/2}_{g,n}\xrightarrow{\pi} \oM^{1/2}_{g,n}
\end{equation*} 
the universal quasi-stable curve of $\oM_{g,n}^{1/2}$ and the universal spin structure. We define two line bundles
\begin{equation*}
    \cF_{k}:=\cL\otimes\omega^{\otimes \frac{k-1}{2}}_{\rm log}\ \ {\rm and}\ \ \cL_{a,k}:=\cF_{k}\left(-\sum_{i=1}^{n}\widetilde{a}_ip_i\right)  
\end{equation*} 
on $\oC_{g,n}^{1/2}$, of degree $|\widetilde{a}|$ and 
of degree $0$, respectively. We denote by 
$\phi_{\cF_{k}}\colon \oM_{g,n}^{1/2}\to \mathfrak{Pic}_{g,n}$
the corresponding classifying morphism. 
Arguing as in~\cite[Section 3.7]{BHPSS}, we obtain
\begin{equation*}
    \DR^{\rm op}_{g,\widetilde{a}}(\phi_{\cF_{k}})[\oM_{g,n}^{1/2}]=\DR_{g}^{1/2}(a,k),
\end{equation*}
where $\DR^{1/2}_{g}(a,k)$ is the class defined in the introduction. Therefore, by~\ref{pix-Pic-stack},
computing the action of $\mathsf{P}_{g,\widetilde{a}}^{g}$ 
on the same data will produce a formula for $\DR^{1/2}_{g}(a,k)$.

\begin{remark}\label{rem:weights-for-prestable/stable} 
Let $(\Gamma,w)$ be a stable graph with a $2$-weighting. 
We note that a spin structure in the stratum corresponding to 
$(\Gamma,w)$ has generically constant multi-degree given by 
\begin{equation*}
    \delta_{w}(v):=\frac{k(2g(v)-2+n(v))-\sum_{h\in H(v)}w(h)}{2}
\end{equation*}
on $v\in V(\Gamma)$, and by $\delta_{w}(v):=w(e)$ on $v\in V(\Gamma_{w})\setminus V(\Gamma)$. Thus, 
from the pair $(\Gamma,w)$ we can define $(\Gamma_w,\delta_{w})$,
where $\delta_{w}$ denotes the multi-degree function of 
spin structures in the stratum $(\Gamma,w)$. 
Now, for every partial stabilization $(\Gamma',\delta')$ of $(\Gamma_{w},\delta_{w})$ we have a stabilization morphism
$\Gamma_{w}\to \Gamma'$, which induces an 
injection $H(\Gamma')\subseteq H(\Gamma_{w})$. 
Restricting a weighting modulo $r$ for $\delta_{w}$
compatible with $\widetilde{a}$ to the subset
$H(\Gamma')$ yields a bijection
\begin{equation*}
    U_{\Gamma_{w},\delta_{w}}^{r}(\widetilde{a})\to 
   U^{r}_{\Gamma',\delta'}(\widetilde{a}).
\end{equation*}
When there 
is no risk of confusion, we will use the same notation 
$u$ for all aforementioned weightings on different graphs.
Finally, the first Betti numbers and the 
cardinality of the automorphism groups of all 
$\Gamma,\Gamma_w$ and $\Gamma'$ are equal.
\end{remark}

\begin{definition} For all positive integers $c,r$ we denote by 
$\mathsf{P}^{c,r,\frac{1}{2}}_{g}(a,k)\in A^*(\oM^{1/2}_{g,n})$ the
codimension $c$ component of the class
\begin{align*}
    \sum_{(\Gamma,w)}
    \sum_{u\in U^{r}_{\Gamma_w,\delta_{w}}(\widetilde{a})}&\frac{r^{-h^1(\Gamma)}}{|\Aut(\Gamma,w)|}
    \zeta_{\Gamma,w*}\left[\prod_{v\in V(\Gamma)}\exp\left(-\frac{k^2}{8}\kappa_1(v)\right)\prod_{i=1}^n\exp\left(\frac{a_i^2}{8}\psi_i\right)\right.\  \times \\
    &\left.\prod_{e=(h,h')\in E(\Gamma)}\frac{1-\exp((\frac{u(h)^2+u(h')^2}{4}-\frac{w(e)}{8})(\psi_h+\psi_{h'}))}{\psi_h+\psi_{h'}} \right]. \\
\end{align*}
The class $\mathsf{P}^{c,r,\frac{1}{2}}_{g}(a,k)$ is a polynomial for $r>>0$. We denote by
$\mathsf{P}^{c,\frac{1}{2}}_{g}(a,k)$ the constant term of this polynomial.
\end{definition}

\begin{remark}\label{rem:r-weight-restriction} In the expression above we are summing over
weightings on the graph $\Gamma_{w}$ which we evaluate
at the half-edges of $\Gamma$. In order to make sense 
of this,
we identify $H(\Gamma)$ as a subset of $H(\Gamma_{w})$
via the stabilization morphism 
$\Gamma_{w}\to \Gamma$ (see the start of~\ref{subsec:strat-of-roots}).
\end{remark}

\begin{proposition}\label{1/2Pix} Let $k\in\ZZ_{\geq1}$ be an odd integer, and let
$a\in \ZZ^n$ be a vector of odd integers such that 
$|a|=k(2g-2+n)$. Then we have 
\begin{equation*}
    \mathsf{P}^{c,r,\frac{1}{2}}_{g}(a,k)=\mathsf{P}^{c,r}_{g,\widetilde{a}}(\phi_{\cF_{k}})[\oM^{1/2}_{g,n}].
\end{equation*}
In particular,
we have $\mathsf{P}^{g,\frac{1}{2}}_{g}(a,k)=\DR^{1/2}_{g}(a,k)$
in $A^*(\oM_{g,n}^{1/2})$.
\end{proposition}
\begin{proof} 
We will prove this formula by using 
the general form of Pixton's formula given in~\ref{op:Pix}. 
First, we define
$\eta_{a,k}:=\pi_{*}(c_1(\cL_{a,k})^2)\in A^*(\oM_{g,n}^{1/2})$, and we observe that
$\eta_{\widetilde{a}}(\phi_{\cF_{k}})[\oM_{g,n}^{1/2}]=\eta_{a,k}$.
Now, there exists a line bundle $\cO(\beta)$ such that 
there is an isomorphism
\begin{equation*}
    \cL^{\otimes2}_{a,k}\cong \omega^{\otimes k}_{\rm log}\left(-\sum_{i=1}^{n}a_ip_i\right)\otimes\cO(\beta)
\end{equation*}
on the universal quasi-stable curve $\oC_{g,n}^{1/2}\to \oM_{g,n}^{1/2}$.
The line bundle $\cO(\beta)$ is constructed from a PL
function that will be introduced in the next section (see~\ref{def:PLfun} and~\ref{def:betaPL}). However, its 
definition is not needed in the present proof.
A similar computation as in ~\cite[the beginning of Section 6.2]{Holmes2022LogarithmicMO} 
or~\cite[Proposition 38]{BHPSS} gives
\begin{equation*}
    \frac{1}{2}\eta_{a,k}=\frac{k^2\kappa_1-\sum_{i=1}^{n}a_{i}^{2}\psi_i+
    \pi_{*}(c_{1}(\cO(\beta))^2)}{8}.
\end{equation*}
Furthermore, we have to compute the contribution obtained for each stratum corresponding
to $(\Gamma,w)$. Recall that points in $\overline{\cP}_{\Gamma,w}$ have generically constant
multi-degree. We denote by $\delta_{w}$ the multi-degree 
function on $\Gamma_{w}$. 
Then, for any partial stabilization $(\Gamma',\delta')$ of $(\Gamma_w,\delta_{w})$,
we have the following pullback diagram
\begin{equation}\label{gysin:pull}
    \xymatrix{\overline{\mathcal{P}}_{\Gamma,w}\ar[r]^j
    \ar[d]_{\zeta_{\Gamma,w}} & 
    \mathfrak{Pic}_{(\Gamma',\delta')}\ar[d]^{j_{(\Gamma',\delta')}}\\
              \oM^{1/2}_{g,n}\ar[r]_{\phi_{\cF_{k}}}& 
              \mathfrak{Pic}_{g,n}}
\end{equation}
This implies that for any decorated pre-stable graph 
$(\Gamma',\delta',\gamma)$,
in the sense of ~\cite[Definition 3]{BHPSS}, we have
\begin{equation*}
    j_{(\Gamma',\delta')*}[\gamma](\phi_{\cL_{k,a}})[\oM^{1/2}_{g,n}]=\zeta_{\Gamma,w*}
    \phi_{\cL_{k,a}}^!(\gamma).
\end{equation*}
Note that if $(\Gamma',\delta')$ is \emph{not} the partial stabilization
of $(\Gamma_{w},\delta_{w})$ for some $2$-weighting $w$ on a
stable graph $\Gamma$, we have
$j_{(\Gamma',\delta')*}[\gamma](\phi_{\cL_{k,a}})[\oM^{1/2}_{g,n}]=0$.
Now, we fix a pair $(\Gamma,w)$ consisting of a stable graph $\Gamma$
and $2$-weighting $w$ compatible with $a$.
In our case, for any partial stabilization $(\Gamma',\delta')$ of $(\Gamma_{w},\delta_{w})$,
the decorated classes we are interested in are of the form 
\begin{equation*}
   j_{(\Gamma',\delta')*}\left[ \prod_{e=(h,h')}(-\psi_{h}-\psi_{h'})^{d_e}\right].
\end{equation*}
For a given pair $(\Gamma',\delta')$, every arrow 
of the diagram~\ref{gysin:pull} is
a regular closed immersion, which simplifies the computation of the excess bundle. We perform this computation in the following lemma.
\begin{lemma}\label{lem:excess-int}The excess 
intersection formula gives  
\begin{equation*}
    \phi^{!}_{\cF_{k}}\left[\prod_{e=(h,h')\in E(\Gamma')}
    (-\psi_{h}-\psi_{h'})^{d_e}\right]=\frac{1}{2^{|E(\Gamma)|}}\prod_{e=(h,h')\in E(\Gamma)}\left(\frac{-\psi_{h}-\psi_{h'}}{2}\right)
    ^{D_e},
\end{equation*}
where 
\begin{equation*}
     D_e=\left\{\begin{matrix}
 d_{e_{1}}+d_{e_{2}} +1 & \text{if $e$ is subdivided to $(e_1,e_2)$ in
 $\Gamma'$} \\
 d_e& \text{otherwise.}  \\
\end{matrix}\right.
\end{equation*}
\end{lemma} 
\begin{proof}[Proof of Lemma 2.20]Let $(\Gamma',\delta')$
be a partial stabilization of $(\Gamma_w,\delta_{w})$. 
We define a map
\begin{equation*}
    f\colon E(\Gamma_w)\to E(\Gamma)
\end{equation*}
as follows: an edge $e\in E(\Gamma)$ is subdivided to two edges
$e_1$ and $e_2$ in $\Gamma_w$ 
if $w(e)=1$ and remains unchanged if $w(e)=0$. 
We form $f$ by mapping $e_1$ and $e_2$ to $e$ if $w(e)=1$ and $e$ maps to itself if $w(e)=0$.
In particular, $f$ has two pre-images at $e$ if $w(e)=1$ and one otherwise.
Furthermore, the partial stabilization morphism $\Gamma_w\to \Gamma'$
induces an injective map 
\begin{equation*}
    i\colon E(\Gamma')\to E(\Gamma_w).
\end{equation*}
We define $f':=f\circ i$, and we denote by $E_{f'}(\Gamma)$
the set of edges of $\Gamma$ with two pre-images under $f'$, 
i.e., $E_{f'}(\Gamma):=\{e\in E(\Gamma)\ |\ |f'^{-1}(e)|=2\}$. With this setup at hand, the normal bundle
$\mathcal{N}_{j_{(\Gamma',\delta')}}$ on $\mathfrak{Pic}_{(\Gamma',\delta')}$
is given by 
\begin{equation*}
    \bigoplus_{\underline{e}=(\underline{h},\underline{h}')\in E(\Gamma')}L_{\underline{h}}\otimes L_{\underline{h}'},
\end{equation*}
where $L_{\underline{h}}$ 
is the cotangent line bundle on the marking corresponding to the 
half edge $\underline{h}$. 
Now, for each edge $\underline{e}=(\underline{h},\underline{h}')\in E(\Gamma')$ and
$e=(h,h')\in E(\Gamma)$
such that $f'(\underline{e})=e$, we have $j^*(L_{\underline{h}}\otimes L_{\underline{h}})=L_{h}\otimes L_{h'}$.
Since all arrows in~\ref{gysin:pull} are regular embeddings, we 
obtain that the excess bundle $T$ on $\overline{\mathcal{P}}_{\Gamma,w}$
is given by $T\cong j^*\mathcal{N}_{j_{(\Gamma',\delta')}}/\mathcal{N}_{\zeta_{(\Gamma,w)}}$. 
Then we have 
\begin{equation*}
    T\cong\bigoplus_{e\in E_{f'}(\Gamma)}L_{h}\otimes 
    L_{h'}
\end{equation*}
Finally, for a given pair $(\Gamma',\delta')$, the excess 
intersection formula gives  
\begin{equation*}
    \phi^{!}_{\cF_{k}}\left[\prod_{e=(h,h')\in E(\Gamma')}
    (-\psi_{h}-\psi_{h'})^{d_e}\right]=\frac{1}{2^{|E(\Gamma)|}}\prod_{e=(h,h')\in E(\Gamma)}\left(\frac{-\psi_{h}-\psi_{h'}}{2}\right)
    ^{D_e},
\end{equation*}
where $D_e$ is as described in the claim. Finally, the factor
$2^{-|E(\Gamma)|}$ occurs from the root 
structure of $\mathfrak{Pic}_{g,n}$ at the nodes.
\end{proof}

We now complete the proof of the proposition.
As explained in~\ref{rem:weights-for-prestable/stable},
weightings modulo $r$ for $\delta_{w}$
compatible with $\widetilde{a}$ and weightings modulo $r$
for $\delta'$ compatible with $\widetilde{a}$ for any partial
stabilization $(\Gamma',\delta')$ of $(\Gamma_{w},\delta_{w})$
are in bijection. Under this bijection,
we fix a weighting modulo $r$ for $\delta_{w}$ compatible with 
$\widetilde{a}$, say $u$,
and the corresponding ones for every partial 
stabilization $(\Gamma',\delta')$, 
which, by abuse of notation, we denote again by $u$. Then, using~\ref{lem:excess-int}, we obtain
\begin{align*}
    \zeta_{\Gamma,w*}&\left[\sum_{\Gamma',\delta'}\phi^{!}_{\cF_{k}}\left(\prod_{e\in E(\Gamma')}\frac{1-\exp(-\frac{u(h)u(h')}{2})(\psi_{h}+\psi_{h'})}{\psi_{h}+\psi_{h'}}\right)\right] \\
    =&\, \zeta_{\Gamma,w*}\,\left[ \prod_{e\in E(\Gamma)}\frac{1-\exp(\frac{u(h)^{2}+u(h')^{2}}{4})(\psi_{h}+\psi_{h'})}{\psi_{h}+\psi_{h'}}\right],
\end{align*}
where in the RHS of the equation above 
we evaluate $u\in U^{r}_{\Gamma_{w},\delta_{w}}(\widetilde{a})$ on the subset $H(\Gamma)$ (see~\ref{rem:r-weight-restriction}). Moreover, we have the following formula:
\begin{align*}
    \zeta_{\Gamma,w}^{*}\exp\left(-\frac{k^2\kappa_1-\sum_{i=1}^{n}a_{i}^{2}\psi_i}{8}\right)=\prod_{v\in V(\Gamma)}\exp\left(-\frac{k^2}{8}\kappa_1(v)\right)\prod_{i=1}^n\exp\left(\frac{a_i^2}{8}\psi_i\right).
\end{align*}
The proof now finishes by using the projection formula to 
include the $\psi$- and $\kappa$- terms of
$\exp(-\frac{1}{2}\eta_{a,k})$
in the sum, and the translation of piecewise polynomials, 
exactly as presented in~\cite[Section 6]{Holmes2022LogarithmicDR} on 
the pair $(\oM_{g,n}^{1/2},\partial\oM_{g,n}^{1/2})$,
to translate the class 
$\exp\left(-\frac{1}{8}\pi_{*}(c_1(\cO(\beta))^2)\right)$.
Using this translation we obtain

\begin{align*}
    \exp\left(-\frac{\pi_{*}(c_1(\cO(\beta))^2)}{8}\right)&\sum_{(\Gamma,w)}
    \sum_{u\in U^{r}_{\Gamma_w,\delta_{w}}(\widetilde{a})}\frac{r^{-h^1(\Gamma)}}{|\Aut(\Gamma,w)|}
    \zeta_{\Gamma,w*}\left[\prod_{v\in V(\Gamma)}\exp\left(-\frac{k^2}{8}\kappa_1(v)\right)\prod_{i=1}^n\exp\left(\frac{a_i^2}{8}\psi_i\right)\right.\  \times \\
    &\left.\prod_{e=(h,h')}\frac{1-\exp((\frac{u(h)^2+u(h')^2}{4})(\psi_h+\psi_{h'}))}{\psi_h+\psi_{h'}} \right] \\
    =& \sum_{(\Gamma,w)}
     \sum_{u\in  U^{r}_{\Gamma_w,\delta_{w}}(\widetilde{a})}\frac{r^{-h^1(\Gamma)}}{|\Aut(\Gamma,w)|}
     \zeta_{\Gamma,w*}\left[\prod_{v\in V(\Gamma)}\exp\left(-\frac{k^2}{8}\kappa_1(v)\right)\prod_{i=1}^n\exp\left(\frac{a_i^2}{8}\psi_i\right)\right.\  \times \\
     &\left.\prod_{e=(h,h')}\frac{1-\exp((\frac{u(h)^2+u(h')^2}{4}-\frac{w(e)}{8})(\psi_h+\psi_{h'}))}{\psi_h+\psi_{h'}} \right].
\end{align*}
\end{proof}
\begin{proposition}~\label{prop:spinPix=Pix1/2+-} The following equality holds in $A^*(\oM_{g,n}^{1/2})$
\begin{equation*}
    \epsilon_{*}\left(2\mathsf{P}^{c,r,\frac{1}{2}}_{g}(a,k)\cdot[\pm]\right)=\mathsf{P}^{c,r,\pm}_{g}(a,k).
\end{equation*}
\end{proposition}
\begin{proof}
We observe that $\mathsf{P}^{c,r,\frac{1}{2}}_{g}(a,k)$ is 
expressed as a sum of classes of the form 
\begin{equation*}
    \zeta_{\Gamma,w*}\nu^*\phi^*(\beta)
\end{equation*}
for pairs $(\Gamma,w)$, 
where $\nu\colon\overline{\cP}_{\Gamma,w}\to \oM_{\Gamma,w}$, and
$\phi\colon \oM_{\Gamma,w}\to \oM_{\Gamma}$ 
are the morphisms described in~\ref{subsec:strat-of-roots}, 
and $\beta\in A^*(\oM_{\Gamma})$. 
If $w$ admits an
even value on an edge,~\ref{prop:cancellation} 
implies that all classes of this form will vanish when capped with $[\pm]$ and
pushed along $\epsilon$. Thus, the only non-trivial
contribution of $\mathsf{P}^{c,r,\frac{1}{2}}_{g}(a,k)$ 
after capped with $[\pm]$ and pushed along $\epsilon$
will arise from 
\begin{align*}
    \sum_{\Gamma\in\mathcal{G}_{g,n}}
    \sum_{u\in U^{r,{\rm odd}}_{\Gamma_w,\delta_{w}}(\widetilde{a})}&\frac{r^{-h^1(\Gamma)}}{|\Aut(\Gamma,w)|}
    \zeta_{\Gamma,w*}\left[\prod_{v\in V(\Gamma)}\exp\left(-\frac{k^2}{8}\kappa_1(v)\right)\prod_{i=1}^n\exp\left(\frac{a_i^2}{8}\psi_i\right)\right.\  \times \\
    &\left.\prod_{e=(h,h')\in E(\Gamma)}\frac{1-\exp((\frac{u(h)^2+u(h')^2}{4}-\frac{1}{8})(\psi_h+\psi_{h'}))}{\psi_h+\psi_{h'}} \right],
\end{align*}
where $U_{\Gamma_{w},\delta_{w}}^{r,{\rm odd}}$ is
the set of weightings mod $r$ for $\delta_{w}$, where $w$
takes only odd values, that is, $w$ is the constant function $1$. In that case, we use~\ref{degree:even-odd}
and we obtain a factor of $2^{g-1}$ when we cap with $[\pm]$ and push along 
$\epsilon$. This, together with the extra factor of $2$ in the equation of
the proposition yields a global factor $2^{g}$. 
Furthermore, for all $u\in U^{r,{\rm odd}}_{\Gamma_{w},\delta_{w}}(\widetilde{a})$, by~\ref{def:weights-for-delta} and~\ref{rem:weights-for-prestable/stable}, we have the following identities
for the restriction of $u$ to $H(\Gamma)$: let $e=(h,h')\in E(\Gamma)$. Then 
\begin{align*}
    u(h)^{2}&\equiv-u(h)-u(h)u(h')\mod r, \\ 
    u(h')^{2}&\equiv -u(h')-u(h)u(h')\mod r.
\end{align*}
Therefore, we obtain
\begin{equation*}
    2(u(h)^2+u(h')^2)-1\equiv -(2u(h)+1)(2u(h')+1)\mod r.    
\end{equation*}
It is a straightforward computation to see that the assignment
\begin{align*}
    U_{\Gamma_{w},\delta_{w}}^{r,{\rm odd}}
    (\widetilde{a})&\to W^{2r/k}_{\Gamma}(a)^{\rm odd} \\
                        u&\mapsto 2u|_{H(\Gamma)}+1,
\end{align*}
where 
$u|_{H(\Gamma)}$ denotes the restriction of $u$
to the subset $H(\Gamma)$ of $H(\Gamma_{w})$, is a bijection.
Thus, we can replace the term $2u+1$ with the corresponding $2r/k$-weightings compatible with $a$ which takes only odd values. 
Finally, we also have $|\Aut(\Gamma,w)|=|\Aut(\Gamma)|$ 
when $w$ is a constant function which proves the desired result.
\end{proof}

\noindent We now have all the necessary tools to
prove~\ref{spDR = spPix} of the Introduction.

\begin{proof}[Proof of~\ref{spDR = spPix}.] 

 By the propositions we proved in this section and the
definition of $\DR^{\pm}_{g}(a,k)$, we obtain
\begin{align*}
    \mathsf{P}^{\pm}_{g}(a,k)&=\epsilon_{*}\left(2\mathsf{P}^{g,\frac{1}{2}}_{g}(a,k)\cdot[\pm]\right) &&(\text{\ref{prop:spinPix=Pix1/2+-}}) \\
    &=\epsilon_{*}\left(2\DR^{1/2}_{g}(a,k)\cdot[\pm]\right) &&(\text{\ref{1/2Pix}}) \\
    &=\DR^{\pm}_{g}(a,k) &&(\text{Definition})
\end{align*}
\end{proof}

\section{DR cycles: Abel-Jacobi and compatibility with roots}

We begin this section by recalling the construction of the double ramification (DR) cycles 
from~\cite{Marcus2017LogarithmicCO,Hol}, along with a slight variation, introduced 
in~\cite{HolSch}. To describe
these constructions,
we begin by outlining the space $\mathbf{Div}_{g,n}$, originally 
introduced in~\cite{Marcus2017LogarithmicCO}. 

\subsection{The $\mathbf{Div}$ stack}
Unless otherwise stated, all log structures considered will be assumed to 
be \emph{fine} and \emph{saturated}. 
Following F. Kato, a \emph{log curve} over a log scheme $S$
is a morphism of log schemes $\pi\colon C\to S$ that is proper, integral, 
saturated, log smooth, and whose geometric fibers are reduced and connected
of pure dimension $1$ (see~\cite[Definition 1.2]{Kato}). When the underlying scheme
of $S$ is the spectrum of an algebraically closed field, we refer to $C\to S$ simply as a log curve over an algebraically closed field.
While the definition of a log curve will not play a central role throughout the text,
we will make substantial use of the following structural result proven in~\cite{Kato}.

We recall the local description of the log structures of
log curves $\pi\colon C\to S$ at a geometric point $c$ of the fiber $C_s$ 
over a geometric point $s\in S$.
We distinguish the following cases:
\begin{enumerate}
    \item [(i)] $c$ is a smooth point of $C_s$ and $\overline{M}_{C_s,c}\cong \overline{M}_{S,s}$, or
    \item [(ii)] $c$ is a marked point of $C_s$ and $\overline{M}_{C_s,c}\cong \overline{M}_{S,s}\oplus\NN$, or
    \item [(iii)] $c$ is a nodal point of $C_s$ and there exists
    an element $\ell_c\in\overline{M}_{S,s}$ such that
    \begin{equation*}
        \overline{M}_{C_s,c}\cong \overline{M}_{S,s}\oplus_{\NN}\NN^2,
    \end{equation*}
    where the morphism $\NN\to \NN^2$ is the diagonal, and 
    $\NN\to\overline{M}_{S,s}$ is given by $1\mapsto \ell_c$.
\end{enumerate}
Over a geometric point $s$ of $S$, the dual graph of $C_s$ is well defined, and a 
nodal point $c$ corresponds to an edge $e$ of its dual graph. We will use the 
notation $\ell(e)$ for the element $\ell_{c}\in\overline{M}_{S,s}$
described above. 
The element $\ell(e)$ is widely known as the \emph{length} of $e$ or the 
\emph{smoothing parameter} of $c$.

\begin{definition} Let $C\to S$ be a log curve over an
algebraically closed field. The \emph{tropicalization} of $C$
is the data $(\Gamma,\ell)$, where $\Gamma$ is the dual graph of $C$, and 
\begin{equation*}
    \ell\colon E(\Gamma)\to \overline{M}_{S}
\end{equation*}
is the length function induced by the local description of the log structure 
discussed above.
\end{definition}

\begin{definition} Let $(X,M_X)$ be a logarithmic scheme. Then, we define the \emph{logarithmic multiplicative group} and the
\emph{tropical multiplicative group} to be the groups
\begin{align*}
    \mathbb{G}_m^{\rm log}(X)&=\Gamma(X,M^{gp}_X) \\
    \mathbb{G}_m^{\rm trop}(X)&=\Gamma(X,\overline{M}^{gp}_X)
\end{align*}
respectively. Furthermore, a $\mathbb{G}_m^{\rm log}$ (resp. $\mathbb{G}_m^{\rm trop}$)-torsor is called a logarithmic (resp. tropical) line bundle.
\end{definition}

\begin{definition}\label{def:divstack} We denote by $\mathbf{Div}$ 
the stack in the strict \'etale topology
on logarithmic schemes whose $S$-points are given by
\begin{enumerate}
    \item [(a)] a log curve $\pi\colon C\to S$, and
    \item [(b)] a tropical line bundle $\mathcal{P}$ 
    on $S$ and an $S$-morphism $\alpha\colon C\to \mathcal{P}$.
\end{enumerate}
Furthermore, we will denote by $\mathbf{Div}_{g,n}$ the component of 
$\mathbf{Div}$ such that the log curves are \emph{stable} of fixed genus
$g$, and where the marked points are labelled by $\{1,\dots,n\}$.
\end{definition}

Let $\pi\colon C\to S$ be a log curve. 
Then, as proven in~\cite[Proposition 4.2.3]{Marcus2017LogarithmicCO},
the $S$-points of $\mathbf{Div}_{g,n}$ whose underlying log curve is given by $C\to S$
are parametrized by sections of $\pi_*(\overline{M}_{C}^{gp})/\overline{M}_S^{gp}$. 
Evaluating an element 
$\alpha\in H^0(S,\pi_*(\overline{M}_{C}^{gp})/\overline{M}_S^{gp})$ 
at a marked point $p_i$  yields
an integer via the isomorphism 
\begin{equation*}
    \overline{M}^{gp}_{C,p_i}/\overline{M}^{gp}_{S,s}\cong\ZZ,
\end{equation*}
which is induced by the local description at markings. This integer is
called the \emph{outgoing slope at the marking} $p_i$.

\begin{definition}\label{def:PLfun} Let $C\to S$ be a log curve over an algebraically closed field. 
A \emph{combinatorial piecewise linear} (PL) 
\emph{function} on the tropicalization $(\Gamma, \ell)$ is a pair 
of functions
\begin{align*}
    F\colon V(\Gamma)&\to \overline{M}_{S}^{gp}, \\
    S\colon L(\Gamma)&\to \ZZ, 
\end{align*}
such that for every edge $e$ connecting two adjacent vertices 
$v,u$, we have
\begin{equation}\label{slopes}
    F(u)-F(v)=s_{v}(e)\ell(e)
\end{equation}
for some integer $s_v(e)$. The integer $s_v(e)$
is called the \emph{outgoing slope at} $e$ on $v$.
\end{definition}

\begin{remark} If $u, v$ are
vertices connected by an edge $e$, then we have 
\begin{equation*}
    s_v(e)=-s_u(e).
\end{equation*}
At this point, we have not imposed any condition on the 
outgoing slopes at the markings. Such conditions will be introduced
in the next subsection when 
we define $\oM^{a}_{g}$.
\end{remark}
\begin{remark}
If $C\to S$ is a log curve over an algebraically closed field, then by 
\cite[Lemma 2.12]{Chen2022ATO}, sections of
$\overline{M}_C^{gp}$ correspond 
bijectively with combinatorial PL functions on 
the tropicalization of $C$. In particular, in that case, 
sections of $\overline{M}_C^{gp}$ which differ by a constant combinatorial PL function are equivalent in $\pi_*(\overline{M}^{gp}_C)/\overline{M}^{gp}_S$.
Hence, geometric $S$-points
of $\mathbf{Div}_{g,n}$ correspond to a log curve $C\to S$
together with a choice of an equivalence class of a combinatorial 
PL function.
    
\end{remark}

In~\cite{Marcus2017LogarithmicCO}, the authors 
construct an Abel-Jacobi map 
\begin{equation*}
aj\colon\mathbf{Div}_{g,n}\to \mathbf{Pic}_{g,n},
\end{equation*}
where $\mathbf{Pic}_{g,n}$ denotes the relative Picard space, 
i.e., the quotient of the Picard stack 
$\mathfrak{Pic}_{g,n}$ with its relative inertia over 
$\oM_{g,n}$. We refer the reader to~\cite[Section 4.3]{Marcus2017LogarithmicCO} for a detailed analysis of this map. The construction is based on the observation that pushing forward the exact sequence
\begin{equation*}
    0\to \cO_{C}^*\to M_{C}^{gp}\to \overline{M}_{C}^{gp}\to 0,
\end{equation*}
along $\pi$ and taking quotients induces a map 
\begin{align*}
    \pi_*(\overline{M}_{C}^{gp})/\overline{M}_S^{gp}&\to R^1\pi_*\cO^*_C \\
       \alpha&\mapsto \cO(\alpha),
\end{align*}
where $\cO(\alpha)$ denotes the line bundle class produced by $\alpha$. 
Moreover, in~\cite[Proposition 3.3.3]{Marcus2017LogarithmicCO}, 
the authors give a more explicit description of
the line bundles $\cO(\alpha)$.  
Let $(C\to S, \alpha)$ be a geometric point
of $\mathbf{Div}_{g,n}$, and let $\nu\colon C^{\nu}\to C$ 
denote the normalization map. The restriction of $\nu^*\cO(\alpha)$ 
to an irreducible component $C_v\subset C^{\nu}$ of the normalization 
is isomorphic to
\begin{equation*}
    \cO_{C_v}(-\sum_{q}\lambda_qq)
\end{equation*}
up to a pullback from $S$,
where the sum runs over all the pre-images $q$ of nodal points 
and markings on $C_v$. The integer $\lambda_q$ is the outgoing slope 
along the edge or  marking corresponding to $q$.
\vspace{2mm}

\subsection{Extending the Abel-Jacobi map}

\begin{definition} Let $k\in\ZZ_{\geq0}$, and let $a\in\ZZ^n$ be a vector of integers such that $|a|=k(2g-2+n)$. We denote by 
$\oM^{a}_{g}$ the open substack 
of $\mathbf{Div}_{g,n}$ parametrizing points $(C\to S, \alpha)$ such that:
\begin{enumerate}
    \item [(i)] For every geometric fiber, the outgoing slopes of $\alpha$ 
    define a $k$-twist compatible with $a$ on the dual graph of the fiber.   
    \item [(ii)] The log structure is minimal with respect to the existence of $\alpha$.
\end{enumerate}
Moreover, we denote by $\oM_{g,f}^{a}$ the space defined in the same way, but considered
as a substack of $\mathbf{Div}_{g,n}$ over the category of fine (and so not necessarily saturated) log schemes.
\end{definition}

\begin{remark} Here minimality is used in the sense of~\cite{Gillam}. 
For a detailed discussion of minimality of points of $\textbf{Div}_{g,n}$
see the proof of~\cite[Theorem 4.2.4]{Marcus2017LogarithmicCO}.
\end{remark}

\begin{remark} We explain the relation of the space defined above
and the spaces constructed in~\cite{Hol} and~\cite{HolSch}. First, 
the space $\oM^{a}_{g}$ is isomorphic to $\oM^{\diamond}$ defined in the former
citation. Indeed, the minimal
log structures of points in $\mathbf{Div}_{g,n}$ yield
\'etale local charts. On the other hand, in~\cite{Hol}, the
author constructs a space by gluing affine patches built using 
toric geometry. These affine patches naturally carry a toric log 
structures, which coincides with the minimal one of points of 
$\oM^{a}_{g}$. 
Furthermore, if one considers the space $\mathbf{Div}_{g,n}$ 
over the category of \emph{fine} log schemes, then
the minimal log structures yield \'etale local charts, 
which are exactly the same as the ones considered in~\cite{HolSch}
(the space constructed in \emph{loc.cit} is denoted by $\oM^{\mathbf{m}}$).
In particular, the space $\oM^{a}_{g,f}$ is isomorphic to the space 
constructed in~\cite{HolSch}.
Finally, as we will also see later, the space 
$\oM^{a}_{g}$ is normal, as it is built using toric affine patches,
and is the normalization of $\oM^{a}_{g,f}$.
\end{remark}

Furthermore, by the definition of $\oM_{g}^{a}$, 
for any $S$-point of $\oM^{a}_{g}$, and every fiber $C_s$ over a geometric point 
$s\in S$, the line bundle
\begin{equation*}
    \omega^{\otimes k}_{C_s,\text{log}}\otimes\cO_{C_s}(\alpha)
\end{equation*}
has multi-degree $\underline{0}$. 
We define the Abel-Jacobi section on $S$-points of $\oM^{a}_{g}$ by
\begin{align*}
    \sigma_{a,k}\colon \oM^{a}_{g}(S)&\to \oJ_{g,n}(S) \\
              (C/S,\alpha)&\mapsto (C/S,\omega_{\rm log}^{k}\otimes\cO(\alpha)),
\end{align*}
where $\oJ_{g,n}\to \oM_{g,n}$ is the universal semi-abelian
Jacobian over $\oM_{g,n}$ parametrizing multi-degree zero line bundles on stable curves.
We define the $0$-section on $S$-points to be the map
\begin{align*}
    \sigma_{0}\colon\oM_{g}^{a}(S)&\to \oJ_{g,n}(S) \\
       (C\to S,\alpha)&\mapsto(C\to S,\cO_{C}).
\end{align*}

\begin{definition} The \emph{double ramification locus}
$\DRL$ in $\oM^{a}_{g}$ is the schematic pullback of the $0$-section along $\sigma_{a,k}$.
\end{definition}
\noindent By~\cite[Proposition 5.3]{Hol}, the morphism $\rho$ 
restricted to $\DRL$ is proper over $\oM_{g,n}$. Now, consider the pullback diagram
\begin{equation*}
    \xymatrix@C=1.6cm{\oJ_{g,n}\times_{\oM_{g,n}}\oM^{a}_{g}\ar[d]
    \ar[r]^{\,\,\,\,\, \rho'\!\!\!\!} &  \oJ_{g,n}\ar[d]  
    \\ \oM^{a}_{g}\ar[r]^{\rho}\ar@/^2.0pc/@[black][u]^{\sigma'_{a,k}} &  
  \overline{\cM}_{g,n}\ar@/_2pc/@{-->}@[black][u]}
\end{equation*}
\noindent where the section $\sigma'_{a,k}$ is induced by the pair 
$(\sigma_{a,k},\text{id}_{\oM^{a}_{g}})$. Let $e$ denote the zero section in $\oJ_{g,n}$, and $e'$ the zero section in $\oJ_{g,n}\times_{\oM_{g,n}}\oM^{a}_{g}$. 
\begin{definition} The \emph{double ramification cycle} 
in $\oM^{a}_{g}$ is defined as
\begin{equation*}
    \text{DRC}=\sigma'^{!}_{a,k}([e']),
\end{equation*}
which is a cycle supported on $\DRL$.
\end{definition}
\begin{definition} The \emph{double ramification cycle} 
in $\oM_{g,n}$ is defined as 
\begin{equation*}
    \DR_{g}(a,k)=\rho|_{\DRL*}\text{DRC}.
\end{equation*} 
\end{definition} 

One may similarly define the double ramification locus
and cycle in $\oM^{a}_{g,f}$~\cite{HolSch}.
We denote this locus by $\DRL_{f}$, which is again proper over $\oM_{g,n}$.

\subsection{\'Etale local description of $\oM_{g}^{a}$ and $\oM^{a}_{g,f}$} 
In the following subsections,
we study the normalization morphism 
$\oM^{a}_{g}\to \oM_{g,f}^{a}$
and its
restriction $\DRL\to \DRL_{f}$. We will then use this 
setup to prove that the lengths of the Artin local rings at generic points of 
$\DRL$ over a generic point of $\DRL_{f}$ are equal (see~\ref{cor:equal-lengths}). 
To do this, we provide a description of the \'etale local charts of the spaces $\oM^{a}_{g}$, $\oM^{a}_{g,f}$, and of a partial normalization of $\oM_{g,f}^{a}$. We obtain these charts
from the minimal log structures of points of $\textbf{Div}_{g,n}$
following~\cite{HolSch}. Let 
\begin{equation*}
   \mathbf{Div}_{g,n}\ \text{and}\ \mathbf{Div}_{g,n}^{f}
\end{equation*}
denote the stack $\mathbf{Div}_{g,n}$ defined in~\ref{def:divstack}
over the category of \emph{fine} and \emph{saturated} (fs) log schemes 
and over the category of \emph{fine} log schemes, respectively.
Furthermore, for a given stable graph $\Gamma$,
we denote by $E$ the set of edges and omit $\Gamma$ for 
notational convenience. 

\subsubsection{Charts for $\oM_{g,f}^{a}$}

Let ${\rm Spec}(k)\to \oM_{g,n}$ be a geometric point of $\oM_{g,n}$ with associated dual graph $\Gamma$.
From the data of $\Gamma$, 
we consider a \emph{combinatorial chart} of $\oM_{g,n}$ 
\begin{equation*}
    \oM_{g,n}\xleftarrow{f} U\xrightarrow{g} \mathbb{A}^{E}
\end{equation*}
in the sense of~\cite[Definition 1.10]{HolSch}, where $U$ is
a scheme, and $\mathbb{A}^{E}$
denotes the affine space of dimension $|E|$ over $k$. Recall that $g$
is smooth and that the origin $0\in \mathbb{A}^E$ lies in the 
image of $g$. The existence of $U$ and the maps $g$ and $f$
follow from the log smoothness of the divisorial log structure
on $\oM_{g,n}$ induced by the boundary $\partial\oM_{g,n}$.
Moreover, by the results of~\cite{Kato}, the space $\oM_{g,n}$ can be covered by combinatorial charts.
From the data of a $k$-twist $I$ on $\Gamma$ compatible with $a$, 
the authors define a monoid 
\begin{equation*}
    N:=\mathbb{N}^{E}\langle\sum_{e\in\gamma}
    I^{\gamma}_e e\rangle_{\gamma \in Y}\subseteq \ZZ^{E},
\end{equation*}
where $Y$ is the set of all cycles $\gamma$ in $\Gamma$, and $I^{\gamma}_e$ 
denotes the value $I(h)$ for $e=(h,h')$ such that $h\in\gamma$
in the direction of the cycle. Using these monoids, they define 
\begin{equation*}
    \mathbb{A}^E_{I,f}:=\text{Spec}(k[N])
\end{equation*}
over $\mathbb{A}^E$. The natural morphism $\NN^{E}\to N$
yields charts for the morphism
$\oM_{g,f}^{a}\to \oM_{g,n}$. In particular, $N$ gives a chart 
for the minimal log structure of a point in $\mathbf{Div}^{f}_{g,n}$
(see the proof of~\cite[Theorem 4.2.4]{Marcus2017LogarithmicCO}).
Fixing a combinatorial chart, we denote by $\oM^{f}_{U,I}$ the pullback of the diagram
\begin{equation}\label{diag:oM_{U,I}^{f}}
    \xymatrix{ & \mathbb{A}^E_{I,f}\ar[d] \\
    U\ar[r] & \mathbb{A}^{E}}.
\end{equation} 
Then, as $I$ varies over the 
$k$-twists of $\Gamma$ compatible with $a$, the spaces $\oM_{U,I}^{f}$ 
glue to a space $M_{U}^{f}$ over $U$, 
which can be upgraded
to a descent datum as $U$ varies over all
combinatorial charts of $\oM_{g,n}$ constructing the space 
$\oM_{g,f}^{a}$ (see~\cite[Section 2.2]{HolSch}). 
\subsubsection{Charts for $\oM_{g}^{a}$}
The saturated case can be treated analogously.
We start again, as in the non-saturated case, 
with a combinatorial chart of $\oM_{g,n}$ associated to a geometric point 
${\rm Spec}(k)\to \oM_{g,n}$
\begin{equation*}
    \oM_{g,n}\xleftarrow{f} U\xrightarrow{g} \mathbb{A}^{E}.
\end{equation*} 
For a given $k$-twist $I$ on the dual graph $\Gamma$ 
of this geometric point compatible with $a$,
we define the space
\begin{equation*}
    \mathbb{A}_{I,fs}^E:=\text{Spec}(k[N^{\rm sat}])
\end{equation*}
over $\mathbb{A}^E$, where $N^{\rm sat}$ denotes the saturation of $N$.
The monoid $N^{\rm sat}$ gives a chart for 
the minimal log structure of $\mathbf{Div}_{g,n}$, and 
the natural morphism $\NN^{E}\to N^{\rm sat}$ is a chart for the 
morphism $\oM^{a}_{g}\to \oM_{g,n}$. We denote by 
$\oM_{U,I}^{fs}$ the space obtained by
pulling back 
$\mathbb{A}_{I,fs}^E\to \mathbb{A}^E$ along $U\to \mathbb{A}^E$. 
Finally, as it is proven in 
~\cite[Corollary 3.13]{Hol}, the spaces $\oM_{U,I}^{fs}$ 
glue together to a space $M_{U}^{fs}$ as $I$ varies over all 
the $k$-twists on $\Gamma$ compatible with $a$.
This can again be upgraded to a descent datum 
as $U$ varies over all combinatorial charts of $\oM_{g,n}$.
Now, in order to have a more explicit description of $\oM^{a}_{g}$, 
we compute $N^{\rm sat}$.
\begin{lemma}\label{lem: minimal-logstr} 
The saturation $N^{\rm sat}$ of $N$ is given by
\begin{equation*}
    \mathbb{N}^{E}\langle\sum_{e\in\gamma}
    \frac{1}{\gcd(\gamma)}I^{\gamma}_e e\rangle_{\gamma \in Y},
\end{equation*}
where $\gcd(\gamma)$ denotes the $\gcd(I(e))_{e\in\gamma}$ of 
all the values $I(e)$ for $e\in\gamma$.
\end{lemma}
\begin{proof} First, we note that $N^{gp}=\ZZ^{E(\Gamma)}$. We denote by
$C_N\subseteq \mathbb{Q}^{E(\Gamma)}$ the $\QQ$-cone spanned by $N$. By 
~\cite[Corollary 2.3.20]{Ogus_2018}, we have that 
\begin{equation*}
    N^{\rm sat}=\ZZ^{E(\Gamma)}\cap C_N,
\end{equation*}
where the intersection is taken inside $\QQ^{E(\Gamma)}$. Then we obtain the
result by a direct computation.
\end{proof}

\subsubsection{Explicit description of the rings $k[N]$ and $k[N^{\rm sat}]$}
For notational convenience, in this subsection we identify the finite set of edges $E$ with $\{e_1,\dots,e_n\}$. 
For each cycle $\gamma\in Y$, we define the following ideal
\begin{equation*}
    J_{\gamma}=\left(\prod_{\substack{e_i\in \gamma \\ I_{e_{i}}^{\gamma}<0}} e_i^{-I_{e_{i}}^{\gamma}}\gamma-
    \prod_{\substack{e_i\in \gamma \\ I_{e_{i}}^{\gamma}>0}} e_i^{I_{e_{i}}^{\gamma}}
    \right).
\end{equation*}
Then
\begin{equation*}
    k[N]\cong\frac{k[ e\ : \ e\in E,\ \gamma^{\pm1}\ :\ \gamma\in Y]}{(A_Y, J_{\gamma}\ :\ \gamma\in Y) },
\end{equation*}
where $A_Y$ denotes the ideal generated by
all the homology relations among the cycles in 
$\Gamma$. Furthermore, since the groupifications $N^{gp}=(N^{\rm sat})^{gp}=\ZZ^E$ are
torsion-free,
both $k[N]$ and $k[N^{\rm sat}]$ are integral 
domains (see~\cite[Proposition 3.3.1]{Ogus_2018}).
The inclusion 
$N\to N^{\rm sat}$ induces a morphism 
\begin{equation*}
    \phi\colon \mathbb{A}_{I,fs}^E\to \mathbb{A}_{I,f}^E.
\end{equation*} 
Moreover, again by~\cite[Proposition 3.3.1]{Ogus_2018}, we see that 
$k[N^{\rm sat}]$ is the integral closure
of $k[N]$ in $\text{Frac}\left(k[N]\right)$, and that $\phi$ 
is the normalization morphism.
We observe that every cycle $\gamma\in Y$ 
induces an integral element, namely
\begin{equation*}
    u_{\gamma} := \frac{\prod_{\substack{e_i\in \gamma \\ 
    I_{e_{i}}^{\gamma}>0}} e_i^{I_{e_{i}}^{\gamma}/\gcd(\gamma)}}
    {\prod_{\substack{e_i\in \gamma \\ I_{e_{i}}^{\gamma}<0}} 
    e_i^{-I_{e_{i}}^{\gamma}/\gcd(\gamma)}}\in k[N^{\rm sat}],\ 
    \ \text{with}\ \ u_{\gamma}^{\gcd(\gamma)}=\gamma.
\end{equation*}
For each such $u_\gamma$, we obtain a relation in $k[N^{\rm sat}]$ 
of the form 
\begin{equation}\label{relation}
    u_{\gamma}\prod_{\substack{e_i\in \gamma \\ I_{e_{i}}^{\gamma}<0}} 
    e_i^{-I_{e_{i}}^{\gamma}/\gcd(\gamma)}-\prod_{\substack{e_i\in \gamma \\ 
    I_{e_{i}}^{\gamma}>0}} e_i^{I_{e_{i}}^{\gamma}/\gcd(\gamma)}=0.
\end{equation}
Note that the elements $u_\gamma$ correspond to the generators 
\begin{equation*}
    \sum_{e_i\in\gamma} \frac{1}{\gcd(\gamma)}I^{\gamma}_{e_i}e_i
\end{equation*} 
in $N^{\rm sat}$. The elements in $\text{Frac}(k[N])$ which 
are integral over $k[N]$ are exactly the elements $u_{\gamma}$ for all $\gamma\in Y$. 
However, in order to further describe $k[N^{\rm sat}]$, we have to 
describe the relations between the elements $u_{\gamma}$. 
We have the following relations:
\begin{enumerate}
    \item [(i)] The relations $u_{\gamma}^{\gcd(\gamma)}=\gamma$.
    \item [(ii)] The relations given by~\ref{relation}. We denote  by $U_{\gamma}$ the ideal
    of $k[N^{\rm sat}]$ generated by these relations.
    \item [(iii)]  The relations among $u_\gamma$ arising from the relations of 
$\gamma$ in the homology of the graph, which we denote by $U_{Y}$.
\end{enumerate}
The above description of the generators and relations of $k[N^{\rm sat}]$ yield
the following isomorphism
\begin{equation*}
    k[N^{\rm sat}]\cong\frac{k[e: e\in E,\ \gamma, \ u_{\gamma}: \gamma\in Y]}
    {(U_{Y}, u_{\gamma}^{\gcd(\gamma)}-\gamma,U_{\gamma} : \gamma\in Y)}.
\end{equation*}

\begin{remark}
Given a $k$-twisted graph $(\Gamma,I)$ compatible 
with $a$, one can recover the spaces $\oM_{U,I}^{fs}$ for the various combinatorial charts of $\oM_{g,n}$ as the pullback of
\begin{equation*}
    \xymatrix{ & \mathbb{A}_{I,fs}^E\ar[d]^{\phi} \\
    \oM_{U,I}^{f}\ar[r] &\mathbb{A}_{I,f}^E}
\end{equation*}
By construction, the morphism $\oM_{U,I}^{f}\to \mathbb{A}_{I,f}^E$
is smooth (see~\ref{diag:oM_{U,I}^{f}}). Moreover, since 
$\phi$ is the normalization morphism, it follows by smooth base change that the induced morphism
\begin{equation*}
    \phi'\colon \oM_{U,I}^{fs}\to \oM_{U,I}^{f}    
\end{equation*}
is also the normalization morphism. In particular, $\oM_{g}^{a}$ is the 
normalization of $\oM_{g,f}^{a}$. 
\end{remark}

\subsubsection{Irreducible components of $\DRL_{f}$ and $\DRL$}\label{subsec:irred-comp}

As shown in~\cite[Lemma 2.13-14]{HolSch}, the map
$\oM^{a}_{g,f}\to \oM_{g,n}$ restricts to a quasi-finite map 
\begin{equation*}
    \DRL_{f}\to \widetilde{\cH}^{k}_{g}(a), 
\end{equation*}
where $\widetilde{\cH}^{k}_{g}(a)$ is the moduli space of twisted
canonical divisors, introduced in~\cite{FarPan,Sch}. 
Moreover, the same authors showed that the dual 
graphs $\Gamma$ of the generic points 
of irreducible components $Z$ of 
$\widetilde{\cH}_{g}^{k}(a)$ are star graphs
admitting a $k$-twist $I$ making the $k$-twisted graph
$(\Gamma,I)$
a $k$-simple star graph, and such that $Z$ is an irreducible
component of $\oM_{\Gamma,I}$. In the proof of~\cite[Proposition 2.19]{HolSch} it is shown that there exists 
exactly one point $p$ in $\DRL_{f}$
lying over each generic point of $\widetilde{\cH}_{g}^{k}(a)$.
Let $\DRL_{f}^{(\Gamma,I)}$ denote an irreducible component of $\DRL_f$, and let $(\Gamma,I)$ be the corresponding 
$k$-simple star graph. Finally, further pulling back 
$\phi\colon \mathbb{A}^E_{I,fs}\to \mathbb{A}^E_{I,f}$ 
along the inclusion $\DRL_{f}^{(\Gamma,I)}\to \oM_{U,I}^{f}$ yields the following cartesian square
\begin{equation*}
    \xymatrix{\DRL^{(\Gamma,I)}\ar[r]\ar[d]_{\phi'|_{\DRL_{fs}^{\Gamma,I}}} & \mathbb{A}_{I,fs}^E\ar[d]^{\phi} \\
           \DRL_{f}^{(\Gamma,I)}\ar[r] & \mathbb{A}_{I,f}^E.}
\end{equation*}
Thus, since $\phi$ is finite and surjective, we conclude
that $\DRL^{(\Gamma,I)}\to \DRL_{f}^{(\Gamma,I)}$ is a finite surjective
morphism. The pullback $\DRL^{(\Gamma,I)}$ is the union 
of all irreducible components corresponding to generic
points of $\DRL$ over the generic point of $\DRL_{f}^{(\Gamma,I)}$.

\begin{corollary}\label{irrcomp:DRL}
The irreducible components of $\DRL$ are 
parametrized by a generic point of $\DRL_{f}$
together with a choice of a point $\widetilde{p}$ of $\DRL$
lying over the generic point $p$.
\end{corollary}

\begin{remark}\label{rem:number-of-points} It can be proven~\cite{PhDGeorgios}, 
using the \'etale local picture of $\oM_{g}^{a}$ and 
$\oM_{g,f}^{a}$, that over the generic point $p$
of $\DRL_{f}^{(\Gamma,I)}$ there are exactly 
\begin{equation*}
    \frac{\prod_{e\in E(\Gamma)}I(e)}{\prod_{v\in V_{\rm out}}\lcm_{e\to v}(I(e))}
\end{equation*}
points in $\DRL$.
\end{remark}

We complete this subsection with a lemma about the
charts $\oM_{U,I}^{fs}$ that will be used in later 
parts of the text to identify the cycle associated to $\DRL$ with
${\rm DRC}$. 
\begin{lemma}\label{lem:l.c.i-charts}~\cite[Lemma 2.7]{HolSch}
Let $(\Gamma,I)$ be a 
$k$-simple star graph. Then $\oM_{U,I}^{fs}$ is a local complete
intersection over $k$.
\end{lemma}

\subsection{Computation of lengths of DRL} 
This subsection is devoted to showing that the lengths
of Artin local rings at generic points of irreducible components of $\DRL$ over a generic point of $\DRL_{f}$
are equal. To do so, we will employ a variant of 
$\oM_{g,f}^{a}$ introduced in~\cite{HolSch}, which 
was used to compute the lengths of Artin local rings
at generic points of $\DRL_{f}$.

We consider the open substack of $\oM^{a}_{g,f}$ 
whose $S$-points are given by $(C\to S,\alpha)$, 
where, fiberwise, the PL function $\alpha$ has slopes 
divisible by $k$ on the edges of the graph. We denote this 
substack by $\oM^{a}_{g,f,k|I}$, and we consider a 
partial normalization of $\oM^{a}_{g,f,k|I}$
\begin{equation*}
    \oM^{a,1/k}_{g,f}\to \oM^{a}_{g,f,k|I}.
\end{equation*}
The space $\oM^{a,1/k}_{g,f}$ can be described \'etale locally by charts similar to the ones presented in the previous sections. 
In particular, given a $k$-twisted graph $(\Gamma,I)$ where $k$ divides
all values of $I$,
the charts of $\oM^{a,1/k}_{g,f}$ are given by
    \begin{equation*}
   \mathbb{A}^{E}_{I',f}:=\text{Spec}k[N'],\ {\rm where}\  
   N':=\mathbb{N}^{E}\langle\sum_{e\in\gamma}
    I^{'\gamma}_e e\rangle_{\gamma \in Y},
\end{equation*}
and where $I^{'\gamma}_{e}:=I^{\gamma}_{e}/k$. 
If $e\in E$ is an edge of $\Gamma$, we define 
$I'(e):=I(e)/k$.
The structure of $\oM^{a,1/k}_{g,f}$ mirrors that of $\oM^{a}_{g,f}$.
In particular, this space admits an Abel-Jacobi section 
\begin{equation*}
    \oM^{a,1/k}_{g,f}\xrightarrow{\sigma_{1/k}} \oJ_{g,n},
\end{equation*}
a double ramification locus, which we denote by $\DRL^{1/k}_{f}$, defined as the schematic pullback of the
$0$-section along $\sigma_{1/k}$, and a cycle
$\text{DRC}^{1/k}_{f}$ supported on $\DRL^{1/k}_{f}$. 
The authors in~\cite{HolSch}
prove that the lengths of the Artin local rings at 
generic points of $\DRL_{f}^{1/k}$ over the generic point of $\DRL^{(\Gamma,I)}_{f}$,
are all equal to
\begin{equation*}
    \prod_{e\in E}I'(e)=\frac{\prod_{e\in E}I(e)}{k^{|E|}}.
\end{equation*}
In fact, they prove even more by giving an explicit description of the 
Artin local rings at such generic points.
\begin{theorem}~\cite[Theorem 5.2]{HolSch}\label{thm: holsch-length} Let $p$ be a generic point of
$\DRL_{f}^{1/k}$ over the generic point of $\DRL^{(\Gamma,I)}_{f}$.
Let $k(p)$ denote the 
residue field at $p$. Then the $k(p)$-algebra structure can be chosen so that
    \begin{equation*}
        \cO_{\DRL^{1/k}_{f},p}\cong \frac{k(p)[e\ \colon\ e\in E]}{(e^{I'(e)}\ \colon\ e\in E)}.
    \end{equation*}
\end{theorem}

Working as in the previous section, we can normalize this space and obtain 
a space $\oM_{g}^{a,1/k}$ and a morphism $\oM^{a,1/k}_{g}\to \oM^{a,1/k}_{g,f}$,
which is the normalization morphism. The construction of
$\oM^{a,1/k}_{g}$ is given \'etale locally by considering the
saturation of the charts of $\oM_{g,f}^{a,1/k}$.
Moreover, we denote by $\DRL^{1/k}$
the double ramification locus in $\oM^{a,1/k}_{g}$.
\begin{remark}\label{rem:comparisonDRL} Up to this point, we have defined several moduli
spaces and double ramification loci. It is thus meaningful to compare
these loci and cycles. In~\cite[Section 2.6]{HolSch} the 
authors give such a comparison for the various DR cycles and loci on the 
spaces $\oM_{g,f}^{a,1/k}$, $\oM_{g,f}^{a}$, and $\oM^{a}_{g}$. In addition,
we can also compare the $\DRL^{1/k}$ with the rest by 
using~\cite[Lemma 2.12]{HolSch} on the proper birational morphism
\begin{equation*}
    \oM^{a,1/k}_{g}\to\oM^{a,1/k}_{g,f}.
\end{equation*}
Moreover, we remark that, according to our discussion on 
the previous subsection, the charts of
$\oM^{a,1/k}_{g}$ are given by
\begin{equation*}
    \mathbb{A}^{E}_{I',fs}:=\text{Spec}k[N^{'\rm sat}],
\end{equation*}
where 
\begin{equation*}
    {N}^{'sat}=\NN^E\langle\sum_{e\in\gamma}
    \frac{1}{\gcd(\gamma')}I^{'\gamma}_e e\rangle_{\gamma \in Y},
\end{equation*}
and where $\gcd(\gamma')$ denotes the value $\gcd(I'(e))_{e\in\gamma}$. 
Furthermore, we observe that $\gcd(\gamma)=k\gcd(\gamma')$.
As a result, the charts of $\oM^{a}_{g}$
for $k$-twists $I$ compatible with $a$, whose values are
divisible by $k$, coincide with those of 
$\oM^{a,1/k}_{g}$.
In that way, we can identify $\oM^{a,1/k}_{g}$ as an open substack
of $\oM^{a}_{g}$. Additionally, the inclusion 
$\oM^{a,1/k}_{g}\to \oM^{a}_{g}$ restricts
to an  open immersion $\DRL^{1/k}\to \DRL$. 
As it is discussed in~\ref{subsec:irred-comp}, 
the dual graph and $k$-twist of  
a generic point of $\DRL$ form a $k$-simple star graph $(\Gamma,I)$, and thus generically the values of the  
$k$-twists $I$ 
are divisible by $k$ on the edges. Therefore, 
the open substack 
$\DRL^{1/k}$ contains all generic points of $\DRL$, and
hence, the computation of
the lengths of the Artin local rings at generic points of $\DRL^{1/k}$ 
is equivalent to computing those of $\DRL$. We provide the following commutative
diagram to summarize the relations of our spaces and the respective double ramification loci:
\begin{equation*}
  \xymatrix{
 & \DRL^{1/k} \ar@{->}[rr]^{{\rm open \ immesion}} \ar@{^{(}->}[ld] \ar@{->}[dd] &  & \DRL \ar@{^{(}->}[ld] \ar@{->}[dd] &  \\
\oM^{a,1/k}_{g} \ar@{^{(}->}[rr] \ar@{->}[rdd] &  & \oM^{a}_{g} \ar@{->}[rdd] &  &  \\
 & \DRL^{1/k}_{f} \ar@{->}[rr] \ar@{^{(}->}[d] &  & \DRL_{f} \ar@{^{(}->}[d] \ar@{->}[r] & \widetilde{\cH}_{g}^{k}(a) \ar@{^{(}->}[d] \\
 & \oM^{a,1/k}_{g,f} \ar@{->}[r] & \oM^{a}_{g,f,k|I} \ar@{^{(}->}[r]^{} & \oM^{a}_{g,f} \ar@{->}[r] & \oM_{g,n}
}
\end{equation*}
\end{remark}

Let $\DRL_{f}^{1/k,(\Gamma,I)}$ denote an irreducible component of  
$\DRL_{f}^{1/k}$ corresponding to a generic point $p$ lying over the generic point of
$\DRL_{f}^{(\Gamma,I)}$\footnote{In~\cite[Lemma 2.5]{HolSch}
the authors show that over the generic point $p$ of $\DRL_{f}^{(\Gamma,I)}$ there are $k^{h^{1}(\Gamma)}$ generic points in $\DRL_{f}^{1/k}$ lying over $p$}. We consider the natural morphism 
\begin{equation*}
    \text{Spec}\cO_{\DRL_f^{1/k},p}\xrightarrow{i} \DRL_{f}^{1/k,(\Gamma,I)},
\end{equation*}
and we post-compose with 
\begin{equation*}
    \DRL_{f}^{1/k,(\Gamma,I)}\xrightarrow{j} \mathbb{A}^E_{I',f}\xrightarrow{g}\mathbb{A}^E,
\end{equation*}
where $j$ is the restriction of $\oM_{U,I'}^{f}\to \mathbb{A}^{E}_{I',f}$. 
We denote the aforementioned composition with $\psi$ (i.e., 
$\psi:=g\circ j\circ i$).
Furthermore, we consider the pullback diagram 
\begin{equation*}
    \xymatrix{ \mathbb{A}^{E}_{\DRL^{1/k}_f,I'}\ar[d]_{g'}\ar[r]^{\psi'}& \mathbb{A}^E_{I',f}\ar[d]^{g} \\
              \text{Spec}\cO_{\DRL^{1/k}_f,p}\ar[r]^{\psi} & \mathbb{A}^E},
\end{equation*}
where 
\begin{equation*}
    \mathbb{A}^{E}_{\DRL^{1/k}_f,I'}:= \text{Spec}k(p)\frac{[e\ \colon\ e\in E, \gamma\ \colon \gamma\in Y]}{(e^{I'(e)}\ : e\in E, A_Y)}.
\end{equation*}
In the equation above, $A_Y$ denotes the ideal generated by the homology relations among cycles in $\Gamma$.
Then we have a section
$\sigma\colon 
\text{Spec}\cO_{\DRL^{1/k}_f,p}\to \mathbb{A}^{E}_{\DRL^{1/k}_f,I'}$
induced by the pair $(\text{id},j\circ i)$. By definition, 
we have 
\begin{equation*}
    \psi'\circ\sigma=j\circ i\ \text{and}\ g'\circ \sigma = 
\text{id}_{\text{Spec}\cO_{\DRL^{1/k}_f,p}}.    
\end{equation*}
\begin{lemma} The morphism ${\rm Spec}\cO_{\DRL^{1/k}_f,p}
\xrightarrow{\sigma} \mathbb{A}^{E}_{\DRL^{1/k}_f,I'}$ is a closed immersion and it is cut out by the ideal
\begin{equation*}
    (e^{I'(e)}\ : e\in E, \gamma-w_\gamma\ :\ \gamma\in Y, A_Y),
\end{equation*} 
where $w_{\gamma}\in \cO^{*}_{\DRL_{f}^{1/k},p}$. Moreover, the invertible elements $(w_{\gamma})_{\gamma\in Y}$ can be chosen to be
$(1)_{\gamma\in Y}$.
\end{lemma}
\begin{proof} First, since $\sigma$ is a section of the separated morphism $g'$, it is a closed immersion. We denote by $\mathbb{A}^{Y}$ the affine scheme
$\spec k(p)[\gamma\ \colon \gamma\in Y]$. 
We identify $\mathbb{A}^{E}_{\DRL^{1/k}_f,I'}$ 
as the closed 
sub-scheme of 
\begin{equation*}
    \text{Spec}\cO_{\DRL^{1/k}_f,p} \times \mathbb{A}^Y,
\end{equation*} 
cut out by the ideal $A_Y$.
Then, the morphism $g'$ is
the projection to the first factor and hence, to give a section of $g'$ is the 
same as giving a morphism
\begin{equation*}
     \text{Spec}\cO_{\DRL^{1/k}_f,p}\to \mathbb{A}^Y,
\end{equation*}
where the images are compatible with the ideal $A_Y$. That is, 
to define a section we have to give an element $w_\gamma$
of $ \cO^{*}_{\DRL^{1/k}_f,p}$ for each $\gamma\in Y$ compatible
with $A_Y$. This proves the first assertion of the lemma. Furthermore, for the second assertion,
we observe that 
different choices, $(w_{\gamma})_{\gamma\in Y}$
and $(w'_{\gamma})_{\gamma\in Y}$ defining sections $\sigma$ and 
$\sigma'$ yield isomorphic images in 
$\mathbb{A}^{E}_{\DRL^{1/k}_f,I'}$. Indeed, the automorphism
\begin{align*}
    \mathbb{A}^{E}_{\DRL_{f}^{1/k},I}&\to \mathbb{A}^{E}_{\DRL_{f}^{1/k},I} \\
    \gamma&\mapsto\gamma w_{\gamma}w^{'-1}_{\gamma} \\
    e&\mapsto e
\end{align*}
maps the image of $\sigma$ to the image of $\sigma'$. Finally,
we observe that elements in $\cO_{\DRL_{f}^{1/k},p}^*$
are given by polynomials in the edges $e\in E$ with non-zero 
constant term. Thus,
we may assume that $w_{\gamma}=1$
for all $\gamma\in Y$.
\end{proof}

\begin{corollary}\label{cor:equal-lengths} Let 
$\DRL^{(\Gamma,I)}_{f}$ be an irreducible component of $\DRL_{f}$ whose generic dual graph and $k$-twist correspond to a fixed $k$-simple star graph $(\Gamma,I)$, and let $p$ denote its generic point. Then all generic points of $\DRL$ lying over $p$ have isomorphic Artin local rings. In particular, the lengths of the Artin local rings at generic points
of such irreducible components of $\DRL$ are equal.
\end{corollary}
\begin{proof} First, we recall that, by~\ref{rem:comparisonDRL}, it suffices to prove the corollary above for $\DRL^{1/k}$. Let 
$\DRL_{f}^{1/k,(\Gamma,I)}$ be an irreducible component
of $\DRL_{f}^{1/k}$ whose generic dual graph and $k$-twist
are correspond to a 
$k$-simple star graph $(\Gamma,I)$ and let $p$
denote its generic point.
We consider the following pullback diagram
\begin{equation*}
    \xymatrix{\coprod_{\Tilde{p}\mapsto p}\text{Spec}\cO_{\DRL^{1/k},\Tilde{p}}\ar[r]\ar[d] & \DRL^{1/k,(\Gamma, I)}\ar[d]\ar[r] & \mathbb{A}^E_{I',fs}\ar[d] \\
    \text{Spec}\cO_{\DRL_{f}^{1/k},p}\ar[r] & \DRL^{1/k,(\Gamma, I)}_f\ar[r] & \mathbb{A}^E_{I',f}}
\end{equation*}
On the other hand, we can compute the pullback of the diagram,
\begin{equation}\label{fig:pbODRL_p}
    \xymatrix{ & \mathbb{A}^E_{I',fs}\ar[d] \\
    \text{Spec}\cO_{\DRL_{f}^{1/k},p}\ar[r]_{\ \ \ \ \ \psi'\circ\sigma} & \mathbb{A}^E_{I',f}}
\end{equation}
say $X$, which will be isomorphic to 
$\coprod_{\Tilde{p}\mapsto p}\text{Spec}\cO_{\DRL,\Tilde{p}}$.
All schemes in diagram~\ref{fig:pbODRL_p} are affine,
so, in order to compute the pullback we only have to compute 
the tensor product. We obtain that $X$ is given by 
\begin{equation*}
    X:=\text{Spec}\frac{k(p)[e\ : e\in E, u_{\gamma}\ : \gamma\in Y]}{(U_Y,U_{\gamma}, u_{\gamma}^{\gcd(\gamma')}-1:\ \gamma\in Y, e^{I'(e)}\ :\ e\in E)}.
\end{equation*}
We recall the factorization
\begin{equation*}
    u_{\gamma}^{\gcd(\gamma')}-1=\prod_{i=1}^{\gcd(\gamma')}\left(u_{\gamma}-\zeta_{\gcd(\gamma')}^i\right)   ,
\end{equation*}
where $\zeta_{\gcd(\gamma')}$ is a primitive $\gcd(\gamma')$-root of unity. 
For a given $\gamma$, the ideals generated by 
$(u_\gamma-\zeta_{\gcd(\gamma')}^i)$ for all $i$ are pairwise coprime. Thus, using Sunzi's remainder theorem iteratively
to decompose $X$ into $\prod_{\gamma\in Y}\gcd(\gamma')$ components. Of course, not all components will be non-zero due to the relations of the ring. 
However, all the non-zero rings
appearing in the decomposition of $X$ are clearly isomorphic.
\end{proof}

\begin{remark}\label{rem:DRL-length} In~\cite{HolSch}
the authors show that the lengths of the Artin local rings
at generic points of $\DRL_{f}$ with generic dual graph and $k$-twist
corresponding to a fixed $k$-simple star graph $(\Gamma,I)$
is given by
\begin{equation*}
    \frac{\prod_{e\in E(\Gamma)}I(e)}{k^{|V_{\rm out}|}}.
\end{equation*}
Now, using~\cite[Lemma 2.12]{HolSch} 
on the proper birational
morphism $\oM_{g}^{a}\to \oM_{g,f}^{a}$, together 
with~\ref{rem:number-of-points} and the corollary above,
we obtain that the lengths of the Artin local rings at generic points of $\DRL$ over generic points of $\DRL_{f}$
as described above are given by 
\begin{equation*}
    \frac{\prod_{v\in V_{\rm out}}\lcm_{e\to v}(I(e))}{k^{|V_{\rm out}|}}.
\end{equation*}
\end{remark}

\subsection{Extending the Abel Jacobi map for roots} Throughout
this section, we fix a vector of \emph{odd} integers $a\in\ZZ^n$, and
an odd integer $k\in\ZZ_{\geq1}$ such that 
$$|a|=k(2g-2+n).$$ 
We denote by $\cC_{g,n}^{1/2}\to \cM_{g,n}^{1/2}$ the universal curve
over $\cM_{g,n}^{1/2}$, and by 
$\cL\to \cC_{g,n}^{1/2}$ the universal spin structure. Under these assumptions, and similarly to
the case of $\cM_{g,n}$,
we define the section
\begin{align*}
    \sigma_{a,k}^{1/2}&\colon \cM_{g,n}^{1/2}\to \mathcal{J}_{g,n}\\
            (C,\cL)&\mapsto \left(C,\cL\left(-\sum_{i=1}^n\frac{a_i-1}{2}p_{i}\right)\otimes
            \omega_{\rm log}^{\frac{k-1}{2}}\right),
\end{align*}
where $\cJ_{g,n}$ is the universal Jacobian over $\cM_{g,n}^{1/2}$. 
We denote by $\cJ^{qst}_{g,n}$ the unique extension of $\cJ_{g,n}$
over $\oM_{g,n}^{1/2}$ parametrizing multi-degree $\underline{0}$ line bundles on quasi-stable curves. 
As explained in the introduction, the section $\sigma_{a,k}^{1/2}$ does not extend to $\oM_{g,n}^{1/2}$.
Similarly to the non-spin case, to define the birational model 
$$\rho_{\frac{1}{2}}\colon\oM_{g}^{1/2,a}\to \oM_{g,n}^{1/2}$$ which extends $\sigma_{a,k}^{1/2}$ we
consider
the component of $\mathbf{Div}$ parametrizing \emph{quasi-stable} log
curves of a fixed genus $g$, and where the marked points are labeled
by $\{1,\dots,n\}$. 
We denote this space by $\mathbf{Div}_{g,n}^{qst}$. 
\begin{definition}~\label{def:twist-on-gamma_w} Let $(\Gamma,w)$ be a stable graph and 
a $2$-weighting compatible with $a$. Then a \emph{half} $k$-\emph{twist on $\Gamma_w$ compatible with 
$a\in\ZZ^n$}
is a function $J\colon H(\Gamma_w)\to \ZZ$ such that:
\begin{enumerate}
    \item [(i)] For all edges $e=(h,h')$ we have 
    \begin{equation*}
        J(h)+J(h')=0.
    \end{equation*}

    \item [(ii)] For all legs $\ell_i$ we have 
    \begin{equation*}
        J(\ell_i)=\frac{a_i-1}{2}.
    \end{equation*}
    \item [(iii)]\label{cond2} For all $v\in V(\Gamma)
    \subseteq V(\Gamma_w)$ we have 
\begin{equation*}
    \sum_{h\in H(v)}J(h)=\frac{k(2g(v)-2+n(v))-\sum_{h\in H(v)}w(h)}{2}.
\end{equation*}
Here we recall that for vertices $v\in V(\Gamma)$ we can identify $H(v)$ of $\Gamma_w$ with $H(v)$ of $\Gamma$.
    \item [(iv)] For all $v\in V(\Gamma_{w})\setminus V(\Gamma)$ we have 
    \begin{equation*}
        \sum_{h\in H(v)}J(h)=w(e_v),
    \end{equation*}
\end{enumerate}
\noindent 
where $e_{v}$ denotes the edge of $\Gamma$ subdivided by inserting 
$v$ in the construction of $\Gamma_{w}$ (see the start of~\ref{subsec:strat-of-roots} for the definition of $\Gamma_{w}$).
\end{definition}

\begin{remark}\label{roottwist} 
Let $\Gamma$ be a stable graph, $w$ a $2$-weighting compatible with $a$, 
and $J$ a half $k$-twist on $\Gamma_w$ compatible with 
$a$. 
Then, using $J$, we can define a $k$-twist 
on $\Gamma$ compatible with $a$ as follows.
Under the identification 
of $H_{\Gamma}(v)$ with the set $H_{\Gamma_w}(v)$ for all $v\in V(\Gamma)$,
we write $h$ for a half-edge in $H(\Gamma)$
and $\overline{h}$ for the corresponding half edge in $H(\Gamma_w)$.
Then the assignment
\begin{equation*}
    h\mapsto 2J(\overline{h})+w(h)
\end{equation*}
defines a $k$-twist $I$ on $\Gamma$ compatible with $a$. Indeed, let 
$e=(h,h')$ be an edge of $\Gamma$. Then we have 
\begin{equation*}
    2J(\overline{h})+w(h)+2J(\overline{h}')+w(h')
    =2(J(\overline{h})+J(\overline{h}')+w(h))=0.
\end{equation*}
Here we used that $w(h)=w(h')$ and the fact that 
$J(\overline{h})+J(\overline{h}')=-w(e)$. For all legs
$\ell_i\in L(\Gamma)$, we have 
\begin{equation*}
    I(\ell_i)=2J(\ell_i)+w(\ell_i)=a_i.    
\end{equation*}
Finally,
by condition $(iii)$ of~\ref{cond2}, the following equation holds  
for all vertices $v\in V(\Gamma)$
\begin{equation*}
    \sum_{h\in H(v)}(2J(\overline{h})+w(h))=k(2g(v)-2+n(v)).
\end{equation*}
We will call
such $k$-twists \emph{associated to} $(\Gamma,w,J)$, or to $(J,w)$
when $\Gamma$ is clear from the context. One can observe 
that for a given stable graph $\Gamma$,
the above assignment yields a bijection of sets
\begin{equation*}
    \coprod_{w}\{\text{half }k\text{-}\text{twists}\ J\  \text{on}\ \Gamma_{w}\ \text{compatible with}\ a\}\to \{k\text{-twists}\ I\ \text{on}\ \Gamma\ \text{compatible with}\ a \}.
\end{equation*}
\end{remark} 

\begin{definition}\label{def:half-Div} Let $a\in\ZZ^{n}$ be a vector of odd integers. 
We denote by $\mathbf{Div}^{1/2}_{g,a}$ 
the open substack of $\mathbf{Div}^{qst}_{g,n}$ parametrizing points 
$(\mathcal{C}\to S, \alpha)$ with the minimal log structure and
such that, fiberwise, for all geometric points $s\in S$, 
the dual graph of $\mathcal{C}_{s}$ is given by a
quasi-stable graph $\Gamma_{w}$ for some $2$-weighting 
$w$, and the outgoing slopes of $\alpha$ are given by a half $k$-twist on the dual graph of 
$\mathcal{C}_{s}$ compatible with $a$.
\end{definition}
\begin{definition}\label{def:half-bir-model} We denote by $\oM^{1/2,a}_{g}$ the space obtained
as the pullback\footnote{Here the pullback is taken in the category of stacks.} of the diagram
\begin{equation*}
    \xymatrix{& \oM^{1/2}_{g,n}\ar[d]^{\pi}\\
           \mathbf{Div}^{1/2}_{g,a}\ar[r]& \oM^{qst}_{g,n},}
\end{equation*}
\noindent where the horizontal and the vertical arrows are defined by
forgetting, respectively, the section $\alpha$ and the spin structure.  
We denote by 
$\rho_{\frac{1}{2}}\colon \oM^{1/2,a}\to \oM^{1/2}_{g,n}$ the pullback 
morphism.     
\end{definition}
\noindent Let $S$ be a geometric point of $\oM^{1/2,a}_{g}$, that is,
a tuple $((\mathcal{C}/S,\cL,\phi),\alpha)$. 
Then, by the definition of $\alpha$, one can see that the line bundle
\begin{equation*}
    \cL\otimes
            \omega_{\rm log}^{\frac{k-1}{2}}\otimes\cO(\alpha)
\end{equation*}
has multi-degree $\underline{0}$. Thus, as in the 
previous section, we have an Abel-Jacobi map defined on 
$S$-points as
\begin{align*}
    \sigma_{a,k}^{1/2}\colon \oM^{1/2,a}_{g}(S)&\to \mathcal{J}^{qst}_{g,n}(S) \\
              ((\mathcal{C}\to S,\cL,\phi),\alpha)&\mapsto \left(\mathcal{C}\to S,\cL\otimes
            \omega_{\rm log}^{\frac{k-1}{2}}\otimes\cO(\alpha)\right),
\end{align*}
\noindent where $\mathcal{J}^{qst}_{g,n}$ parametrizes multi-degree $\underline{0}$ line bundles on quasi-stable curves.

\begin{definition} We define the double ramification 
locus $\DRL^{1/2}$ in $\oM^{1/2,a}_{g}$ to be the schematic 
pullback of the $0$-section along $\sigma_{a,k}^{1/2}$.
\end{definition}
\noindent By~\cite[Proposition 4.5.3]{Marcus2017LogarithmicCO}, the morphism $\rho_{\frac{1}{2}}$ 
restricted to $\DRL^{1/2}$ is proper over $\oM^{1/2}_{g,n}$. Now, consider the cartesian diagram 
\begin{equation*}
    \xymatrix@C=1.6cm{
  \mathcal{J}_{g,n}^{qst}\times_{\oM^{1/2}_{g,n}}\oM^{1/2,a}_{g}\ar[d]
    \ar[r]^{\,\,\,\,\, \rho_{\frac{1}{2}}'\!\!\!\!\!}  & \mathcal{J}_{g,n}^{qst}\ar[d]  
    \\ 
  \oM^{1/2,a}_{g}\ar[r]^{\rho_{\frac{1}{2}}}\ar@/^2.0pc/@[black][u]^{\sigma_{a,k}^{1/2'}}  & 
  \overline{\cM}^{1/2}_{g,n}\ar@/_2pc/@{-->}@[black][u]}
\end{equation*}
\noindent Here again, $\sigma_{a,k}^{1/2'}$ denotes the section induced by 
$(\sigma_{a,k}^{1/2},\text{id}_{\oM^{1/2,a}})$. Let
$e$ denote the image $0$-section in $\mathcal{J}^{qst}_{g,n}$, and $e'$ the pullback of the $e$ in $\mathcal{J}^{qst}_{g,n}\times_{\oM^{1/2}_{g,n}}\oM_{g}^{1/2,a}$.

\begin{definition} We define the \emph{double ramification cycle}
in $\oM^{1/2,a}_{g}$ to be 
\begin{equation*}
    \text{DRC}^{1/2}=\sigma^{1/2'!}_{a,k}([e']),
\end{equation*}
which is a cycle supported on $\DRL^{1/2}$.
\end{definition}

\begin{definition}\label{def:spinDR} The \emph{double ramification cycle} 
in $\oM_{g,n}^{1/2}$ is defined as 
\begin{equation*}
    \DR_{g}^{1/2}(a,k)=\rho_{\frac{1}{2}}|_{\DRL^{1/2}*}
    \text{DRC}^{1/2}. 
\end{equation*} 
Moreover, we define the \emph{spin double ramification cycle} as
\begin{equation*}
    \DR_{g}^{\pm}(a,k)=2\epsilon_*(\DR_{g}^{1/2}(a,k)\cdot[\pm]).
\end{equation*}
\end{definition}

\subsection{Relation between \texorpdfstring{$\oM^{a}_{g}$}{} and 
\texorpdfstring{$\oM^{1/2,a}_{g}$}{}.}
Now that we have defined the spaces $\oM_{g}^{a}$ and 
$\oM^{1/2,a}_{g}$ and their associated 
double ramification loci, it is
important to study their relationship. At the level
of smooth curves we have a commutative diagram 
\begin{equation*}
    \xymatrix{\cM_{g,n}^{1/2}\ar[r]^{\sigma_{a,k}^{1/2}}\ar[d]_{\epsilon}& \cJ_{g,n}\ar[d]^{\otimes 2} \\
    \cM_{g,n}\ar[r]_{\sigma_{a,k}}& \cJ_{g,n}.}
\end{equation*}
Moreover, after a straightforward computation one verifies that
$\epsilon$ restricts to $\DRL^{1/2}\to \DRL$, and
that this restriction is actually a $\mu_{2}$-gerbe. 
In this section, we extend this picture to the birational
models $\oM_{g}^{1/2,a}\to \oM_{g}^{a}$ (see~\ref{mu2gerbe} and~\ref{mu_2DR:gerbe}).

We begin with 
the following key observation. Let $\mathcal{C}\to S$ be a 
family of quasi-stable
models of $C\to S$. We denote by 
\begin{equation*}
    \mathbf{st}\colon \mathcal{C}\to C
\end{equation*}
the stabilization morphism over $S$. Then any line bundle $\cL$ on 
$\mathcal{C}$, which is fiberwise of degree $0$ 
on every exceptional component, is a pullback from $C$ in
a unique way. Now, let $(\mathcal{C}\to S, \cL,\phi)$ be a geometric
point of $\oM^{1/2}_{g,n}$, and let $(\Gamma,w)$ be the corresponding 
stable graph and $2$-weighting compatible with $a$. 
By~\cite[Lemma 2.2.5]{Chiodo2006TowardsAE}
we obtain that 
\begin{equation*}
    \cL^{\otimes2}\cong \omega_{\rm log}\otimes \cO(D)
\end{equation*}
where $D$ is a Cartier divisor on $\mathcal{C}$. In~\cite[Chapter 6]{Holmes2022LogarithmicMO} 
the authors give a translation of $D$ to a
piecewise linear function in the following sense; they define
a piecewise linear function $\beta$ on the tropicalization of 
$\mathcal{C}$ such that 
\begin{equation*}
    c_1(\cO(\beta))=D.
\end{equation*}
Additionally, they prove that up to the addition of a constant, $\beta$ is given by
\begin{equation}\label{def:betaPL}
    \beta(v)=\left\{\begin{matrix}
 0 & \text{for all}\  v\in V(\Gamma)\\ 
 -w(e_v)\frac{\ell(e_v)}{2} & \text{otherwise,}
\end{matrix}\right.
\end{equation}
where $e_v$ is the edge of $\Gamma$ corresponding to the exceptional vertex 
$v$ (see~\ref{def:twist-on-gamma_w} (iv) for the definition
of $e_{v}$). It follows that the sum of outgoing slopes of $\beta$ on 
a vertex $v$ corresponding to an exceptional component
is $-2w(e_v)$. Now, let $((\mathcal{C}\to S, \cL,\phi),\alpha)$
be a geometric point of $\oM^{1/2,a}_{g}$, let 
$\Gamma_{w}$ be the dual graph of $\cC$ for some  
$2$-weighting $w$, and let $J$ be the half $k$-twist compatible with $a$ corresponding to the slopes of $\alpha$ (see~\ref{def:half-Div,def:half-bir-model}). Then we have 
\begin{equation*}
    \left(\cL\otimes
            \omega_{\rm log}^{\frac{k-1}{2}}\otimes\cO(\alpha)\right)^{\otimes2}\cong\omega^{\otimes k}_{\rm log}\otimes\cO(\beta)\otimes\cO(2\alpha).
\end{equation*}

\noindent Note that the sum of outgoing slopes 
of $2\alpha + \beta$  on every vertex corresponding to an exceptional
component is $0$. Indeed,
let $v$ be such a vertex, 
and let $h,h'$ be the two half edges incident to $v$. The
sum of the outgoing slopes of $2\alpha + \beta$ on $v$ is
given by 
\begin{equation*}
    2J(h)-w(e_v)+2J(h')-w(e_v)=2(J(h)+J(h')-w(e_v))=0.
\end{equation*}
The equation above holds by condition $(iv)$ of~\ref{cond2}.
Hence, $\cO(2\alpha+\beta)$ is a pullback from $C$ in a unique way.
In particular, there exists a PL function $\alpha'$ on the tropicalization of $C$
such that
\begin{equation*}
    \mathbf{st}^*(\alpha')=2\alpha +\beta,
\end{equation*}  
and which is defined, up to an addition of a constant, by 
\begin{equation*}
    \alpha'(u)=2\alpha(u)
\end{equation*}
for all vertices $u\in V(\Gamma)$. Note also that if the slopes of 
$\alpha$ are given by a half $k$-twist $J$ compatible with $a$, then the slopes of $\alpha'$ are given
by the $k$-twist $I$ associated to the pair $(J,w)$ (see~\ref{roottwist}).
Moreover, the following isomorphism holds
\begin{equation*}
    \mathbf{st}^*\left(\omega_{\rm log}^{\otimes k}\otimes\cO(\alpha')\right)\cong
  \omega^{\otimes k}_{\rm log}\otimes\cO(\beta)\otimes\cO(2\alpha).
\end{equation*}
We define a map on $S$-points 
\begin{align*}
    \epsilon_{a}\colon \oM^{1/2,a}_{g}&\to \oM^{a}_{g}\\
          ((\mathcal{C}\to S,\cL,\phi),\alpha)&\mapsto(C\to S, \alpha'),
\end{align*}
where $C$ is the stabilization of $\mathcal{C}$. 

\begin{proposition}\label{mu2gerbe} The space $\oM^{1/2,a}_{g}$ fits in a commutative diagram
\begin{equation}\label{diagram: prop3.26}
    \xymatrix{ \oM^{1/2,a}_{g}\ar[r]^{\sigma^{1/2}_{a,k}}\ar[d]_{\epsilon_{a}}
                & \mathcal{J}^{qst}_{g,n}\ar[d]^{\otimes2} \\
           \oM^{a}_{g}\ar[r]_{\sigma_{a,k}} & \oJ_{g,n}.}
\end{equation}
\noindent Moreover, $\epsilon_{a}\colon \oM^{1/2,a}_{g}\to \oM^{a}_{g}$ 
restricts to a proper morphism $\DRL^{1/2}\to \DRL$.
\end{proposition}
\begin{proof} The first assertion follows from the
discussion above. For the second assertion, we observe that
under $\epsilon_{a}$ an element of $\DRL^{1/2}$ lands in 
$\DRL$. To show this, let 
$((\mathcal{C}\to S,\cL,\phi),\alpha)$ be a geometric point of
$\DRL^{1/2}$. 
Then we have 
\begin{align*}
    \cO_{\mathcal{C}}&\cong \left(\cL\otimes\omega_{\rm log}^{\frac{k-1}{2}}\otimes\cO(\alpha)\right)^{\otimes2}\\ 
    &\cong \omega^{\otimes k}_{\rm log}\otimes\cO(\beta)\otimes\cO(2\alpha).
\end{align*}
As we noted before, the line bundle $\cO(2\alpha+\beta)$
on $\mathcal{C}$ has degree $0$ on the exceptional components, and hence it is a pullback under the stabilization
morphism $\mathbf{st}\colon \mathcal{C}\to C$. We have
\begin{equation*}
    \cO_{\mathcal{C}}\cong \mathbf{st}^*\left(\omega^{\otimes k}_{\rm log}\otimes\cO(\alpha')\right),
\end{equation*}
which implies that $\omega^{\otimes k}_{\rm log}\otimes\cO(\alpha')\cong \cO_C$
as desired.
Furthermore, $\DRL^{1/2}$ is proper over $\oM^{a}_{g}$. 
Indeed, we have the following commutative diagram
\begin{equation*}
    \xymatrix{\oM^{1/2,a}_{g}\ar[r]\ar[d]_{\epsilon_{a}} & \oM^{1/2}_{g,n}\ar[d]^{\epsilon} \\
           \oM^{a}_{g}\ar[r] & \oM_{g,n}.} 
\end{equation*}
\noindent We already know that $\DRL^{1/2}$ is proper
over $\oM^{1/2}_{g,n}$, and that the
morphism $\epsilon$ is proper in general. Thus, by the commutativity
of the diagram above we have that $\DRL^{1/2} 
\to \oM_{g,n}$ is proper. The proof is completed
by considering
the commutative triangle
\begin{equation*}
    \xymatrix{ \DRL^{1/2}\ar[d]\ar[dr]& \\
           \DRL\ar[r]& \oM_{g,n}}
\end{equation*}
and using ~\cite[Tag01W6]{stacks-project}.
\end{proof}

\begin{corollary}\label{mu_2DR:gerbe} The morphism 
\begin{equation*}
    \epsilon_{a}|_{\DRL^{1/2}}\colon \DRL^{1/2}\to \DRL 
\end{equation*}
is a $\mu_2$-gerbe, and we have
\begin{equation*}
\epsilon_{a*}|_{\DRL^{1/2}}\text{DRC}^{1/2}=\frac{1}{2}\text{DRC}.
\end{equation*}
\end{corollary}
\begin{proof}
    To see that
$\epsilon_{a}$  restricts to a $\mu_2$-gerbe on $\DRL^{1/2}$, one has to observe that twisting an $S$-point of 
$\DRL^{1/2}$ by the pullback of a $2$-torsion line bundle from $S$
yields another $S$-point of $\DRL^{1/2}$. 
\end{proof}

\begin{remark} The gerbe structure of $\DRL^{1/2}$ over $\DRL$ is the reason
that we need to implement a factor of $2$ in the definition of $\DR^{\pm}_{g}(a,k)$. Moreover, again the gerbe structure
is the only obstruction to making the diagram~\ref{diagram: prop3.26}
a pullback diagram when we restrict the left vertical map to $\DRL^{1/2}\to \DRL$.
Finally, it is important to remark that the above result also shows
that the map $\epsilon_{a}$ restricted to $\DRL^{1/2}\to \DRL$ is
actually an isomorphism on the coarse spaces.
\end{remark}
\begin{corollary}\label{mult:root} The irreducible components of the 
substack $\DRL^{1/2}$ are in bijection with those of $\DRL$. Moreover,
the lengths of the Artin local rings at generic points of $\DRL^{1/2}$ lying
over  the generic point of $\DRL_{f}^{(\Gamma,I)}$ via
the composition $\DRL^{1/2}\to \DRL\to \DRL_{f}$
are equal.
\end{corollary}

\section{Parity of the Double Ramification locus} 
We start this section by setting up some notation.
First, we fix a $k$-simple star graph $(\Gamma,I)$. 
We use the same symbol $p$ for the generic points of 
irreducible components $\DRL_{f}$ 
whose generic dual graph and $k$-twist
correspond to $(\Gamma,I)$.
Now, using~\ref{irrcomp:DRL} and~\ref{mult:root}, by abuse of
notation, we denote by $\DRL^{\widetilde{p}}_{\Gamma,I}$
and $\DRL^{1/2,\widetilde{p}}_{\Gamma,I}$ the irreducible components of $\DRL$ and $\DRL^{1/2}$ corresponding to
some point $\widetilde{p}$ in $\DRL$ lying over a generic
point $p$ as above. Furthermore, we define the closed substacks
\begin{equation*}
   \DRL_{\Gamma, I}:=\bigcup_{p,\ \widetilde{p}\to p}\DRL_{\Gamma,I}^{\widetilde{p}}\ \ \text{and}\ \  \DRL^{1/2}_{\Gamma,I}:=\bigcup_{p,\ \widetilde{p}\mapsto p}\DRL^{1/2,\widetilde{p}}_{\Gamma,I}
\end{equation*}
of $\DRL$ and $\DRL^{1/2}$, respectively. Here, $p$ runs over all generic points of $\DRL_{f}$ whose generic dual graph and $k$-twist correspond to $(\Gamma,I)$, and $\widetilde{p}$ runs over
points in $\DRL$ lying over $p$.

Furthermore, by~\ref{rem:DRL-length} the lengths of the Artin local rings at all generic points 
$\widetilde{p}$ over generic points of irreducible components of $\DRL_{f}$ with generic dual graph and 
twist given by a fixed $k$-simple star graph $(\Gamma,I)$ 
are equal. Therefore,
the cycle associated to the closed
substack $\DRL_{\Gamma,I}^{1/2}$ is given by 
\begin{equation*}
    \DR^{1/2}_{\Gamma,I}:={\rm length}(\widetilde{p})\sum_{p,\ \widetilde{p}\mapsto p}[\DRL^{1/2,\widetilde{p}}_{\Gamma,I}], 
\end{equation*}
and similarly for $\DRL_{\Gamma,I}$.
Additionally, since 
$\DRL^{1/2}$ lies over $\oM_{g,n}^{1/2}$, it splits according to parity
\begin{equation*}
    \DRL^{1/2}=\DRL^{1/2,+}\coprod \DRL^{1/2,-}.
\end{equation*}
This also applies to the substack $\DRL^{1/2}_{\Gamma,I}$.
This section is devoted
to establishing a $\mu_{2}$-action on $\DRL^{1/2}_{\Gamma,I}$
whenever $I$ admits an even value, which exchanges the two components
$\DRL^{1/2,+}_{\Gamma,I}$ and $\DRL^{1/2,-}_{\Gamma,I}$ 
(see~\ref{lem:exchpar}).

\subsection{Description of the locus $\oM^{a}_{\Gamma,I}$}
\begin{definition}Let $(\Gamma,I)$ be a $k$-simple star graph and let $(\Gamma,w,J)$
be the associated triple given by~\ref{roottwist}. We denote by $\oM^{a}_{\Gamma,I}$ (resp. $\oM^{1/2,a}_{\Gamma,I}$) the substack of $\oM^{a}_{g}$ (resp. $\oM^{1/2,a}_{g}$) 
consisting of curves whose topological type is given by $\Gamma$ (resp. $(\Gamma,w)$)
and the slopes of $\alpha$ are given by $I$ (resp. $J$).
\end{definition}
\begin{lemma}\label{cor:rank-of-mininal-log} Let $(\Gamma,I)$ be a 
$k$-simple star graph. Then for all points of $\oM_{\Gamma,I}^{a}$, 
the characteristic monoid of the
minimal log structure is free of rank $|V_{\rm out}|$.
\end{lemma}

\begin{proof}The proof of this lemma is demonstrated in two steps: 
we reduce the general case to the case of a 
$k$-simple star graph $(\Gamma,I)$ with only one outlying vertex and then
we prove this simpler case.

Let $(C\to S,\alpha)$ be an $S$-point of 
$\oM_{\Gamma,I}^{a}$, and let $s\in S$ be a geometric point. To prove
our lemma, we may 
work in a sufficiently small neighborhood around $s$, so that $C$ 
has constant topological type, i.e., it has a well-defined
dual graph $\Gamma$, $\alpha$ is given by
a PL function on the tropicalization of $C_{s}$, and that 
$\alpha$ is (up to addition of constant) 
$0$ on the central vertex. Using the description of the minimal log
structure of points of $\textbf{Div}_{g,n}$ discussed 
in~\cite[Theorem 4.2.4]{Marcus2017LogarithmicCO}, which is also the one we are using throughout the previous sections, 
we may also assume that the minimal log structure
in this sufficiently small neighborhood is \emph{isomorphic} to
\begin{equation}\label{eq:proof-of-rank}
    M:=\mathbb{N}^{E}\langle\sum_{e\in\gamma}
    \frac{1}{\gcd(\gamma)}I^{\gamma}_e e\rangle_{\gamma \in Y}
\end{equation}
(see also~\ref{lem: minimal-logstr} and the text above it). 
Moreover, since $\Gamma$ is a star graph we obtain a 
decomposition 
\begin{equation*}
    \mathbb{N}^{E}\langle\sum_{e\in\gamma}
    \frac{1}{\gcd(\gamma)}I^{\gamma}_e e\rangle_{\gamma \in Y}\cong\bigoplus_{v\in V_{\rm out}}
    \mathbb{N}^{E_{v}}\langle\sum_{e\in\gamma}
    \frac{1}{\gcd(\gamma)}I^{\gamma}_e e\rangle_{\gamma \in Y_{v}},
\end{equation*}
where $E_v$ denotes the set of edges incident to the outer 
vertex $v$, and $Y_v$ denotes the subset of $Y$ consisting of cycles
going through edges in $E_v$. We denote by
$M_{v}$ each direct summand in the isomorphism above. 
Then, in order to pass to the characteristic monoid,
we have to quotient 
by the group of invertible elements. We observe that 
\begin{equation*}
    M/M^{\times}\cong \bigoplus_{v\in V_{\rm out}}M_{v}/M^{\times}_{v}.
\end{equation*}
Thus, to show that the characteristic monoid $M/M^{\times}$ is free 
of rank $|V_{\rm out}|$, it suffices to show that for each 
$v\in V_{\rm out}$ the monoid $M_{v}/M_{v}^{\times}$ 
is free of rank $1$, which finishes the reduction to the case of
a single outlying vertex. 

We will use the description
of the minimal log structure of points of a log blow-up of 
$\textbf{Div}$ given 
in~\cite[Definition 4.3 and Lemma 4.5]{Chen2022ATO}. 
The minimal log structure at points of this log-blow up whose topological type 
is given by $\Gamma$ and whose PL function has slopes 
given by a $k$-twist $I$ compatible with $a$,
such that the pair $(\Gamma,I)$ is a $k$-simple
star graph with a single outlying vertex
coincides with
that of $\oM_{\Gamma,I}^{a}$. 
There, the authors prove this in that case the characteristic monoid of
the minimal log structure of a point of $\oM_{\Gamma,I}^{a}$
is free of rank\footnote{In~\cite{Chen2022ATO},
the authors describe the characteristic monoid of the 
minimal log structure as the free monoid
generated by the \emph{horizontal edges} (see~\cite[Definition 2.10]{Chen2022ATO}), i.e. edges such that $I(e)=0$, and the number of levels given by the \emph{normalized level function} 
(see below~\cite[Definition 2.1]{Chen2022ATO}). 
In the case of single outlying vertex we have only
$1$ level and no horizontal edges.}
\begin{equation*}
    \#\{e\in E(\Gamma)\ |\ I(e)=0\ \}+ 1.
\end{equation*}
The fact that $(\Gamma,I)$ is a $k$-simple star graph implies
that $I(e)\neq0$ for all $e\in E(\Gamma)$, which completes the proof.
\end{proof}

\begin{lemma}\label{lem:normal_bundle_length}
    Let $C/S$ be a log curve with a persistent node $e$, with length $\ell(e) \in \overline{M}_S(S)$. Let $\nu\colon \tilde C \to C$ be the partial normalisation along $e$, and let $h$, $h'$ be the sections of $\tilde C$ lying over $e$. Then there is a canonical isomorphism 
    \begin{equation}
    \cO_S(\ell(e)) \cong h^*\cO_{\tilde C}(p_h)\otimes (h')^*\cO_{\tilde C}(p_{h'})
    \end{equation}
    of line bundles on $S$. 
\end{lemma}
\begin{proof}
Working locally around $e$, there is a unique PL function $\alpha$ on $C$ taking value $0$ on the branch carrying $p_h$, the value $\ell(e)$ on the branch carrying $p_{h'}$, and having slope $1$ along the edge. Then $\nu^*\cO_C(\alpha)$ is naturally identified with $\cO_{\tilde C}(p_h)$ on the branch of $\tilde C$ carrying $e^+$, and with $\cO_{\tilde C}(\ell(e) - p_{h'})$ on the branch of $\tilde C$ carrying $p_{h'}$. We compute
\begin{equation}
e^*\cO_C(\alpha) = h^*\nu^*\cO_C(\alpha) = (h')^*\nu^*\cO_C(\alpha), 
\end{equation}
the second term is $h^*\cO_{\tilde C}(p_h)$ and the third is $\cO_S(\ell(e)) \otimes (h')^*\cO_{\tilde C}(-p_{h})$. 
\end{proof}

In this section we denote by 
$\underline{X}$ the underlying scheme of a log scheme $X$. 
Let $(\Gamma,I)$ be a $k$-simple star graph with two vertices. We define 
\begin{equation*}
    L:= \lcm_e(I(e))\ \text{and}\ L_e:= L/I(e). 
\end{equation*}
As shown in~\ref{cor:rank-of-mininal-log}, in that case
the log structure of the substack $\oM^{a}_{\Gamma,I}$ has rank $1$ everywhere, 
say with characteristic monoid generated by $\delta$. 
In particular, given a geometric point $(C,\alpha)$ of
$\oM_{\Gamma,I}^{a}$, for each edge $e$ of the dual graph of $C$
we have
\begin{equation*}
    \ell(e) = L_e\delta. 
\end{equation*} 
Indeed, for such a geometric point and
for all edges $e\in E(\Gamma)$, the equation
$I(e)\ell(e)=L\delta$ holds. These are exactly 
the relations which define the minimal 
log structure over which a PL function with slopes $I(e)$ exists.
We write $\mathcal{T}$ for the $\mathbb{G}_m$-torsor on $\oM_\Gamma$ given by 
$\cO^\times(\delta)$.
\begin{lemma}\label{loc:div}
Let $(\Gamma,I)$ be a $k$-simple star graph with two vertices. 
There is an isomorphism 
\begin{equation}
\oM^{a}_{\Gamma,I} \stackrel{\sim}{\longrightarrow} \left(\bigoplus_{e\in E} \mathcal{T}^{\otimes L_e}\right)/\mathbb G_m,
\end{equation}
where the action is given by
\begin{equation}\label{eq:aut_action}
\lambda \cdot (u_e)_{e \in E} = (\lambda^{L_e}u_e)_{e \in E}. 
\end{equation}
\end{lemma}
\begin{proof}
The map $\oM^{a}_{\Gamma,I}\to \oM_\Gamma$ is a log monomorphism, that is, a representable monomorphism of log stacks. 
A logarithmic map $x\colon X \to \oM_{g,n}$ factors via 
$\oM^{a}_{\Gamma,I}$ if and only if the topological type is 
$\Gamma$ and the sections 
$x^*{I(e)}\ell(e)$ of $\overline{M}_X$ are the same for all edges $e$. 

We fix a scheme $\underline{X}$ over $\oM_\Gamma$, and compute the maps 
from this to $\oM^{a}_{\Gamma,I}$. This comes down to computing the 
\emph{minimal} maps from log schemes $X$, with underlying scheme $\underline{X}$, 
to $\oM^{a}_{\Gamma,I}$, up to automorphisms of the log structure on $X$. 
Since $\oM^{a}_{\Gamma,I}$ has constant log structure $\NN$, 
the same will be true for any minimal $X$. Thus, we may 
assume without loss of generality that the log structure of 
$X$ is $\NN \oplus \cO_X^\times$. 
The log maps $X \to \oM^{a}_{\Gamma,I}$ are then determined by 
their action on the monoids; this amounts to specifying, for 
each edge $e$, a map of torsors 
$\cO_X^\times(x^*\ell(e)) \to \cO_X^\times$; dually, to specifying a 
section of 
$$\cO_X^\times(-x^*\ell(e)) = \cO_X^\times(-L_ex^*\delta).$$ 
The lemma now follows by observing that the automorphisms of the 
log structure of $X$ (a $\mathbb{G}_m$) act on $\cO_X^\times(n)$ 
by $n$th powers.
\end{proof}

Let $\underline{u} = (u_e)_{e \in E} \in \bigoplus_{e \in E} \mathcal{T}^{\otimes L_e}$. 
The fibre of the line bundle $\cO(\alpha)$ over the point $\underline{u}$  satisfies
\begin{equation*}
    \cO(\alpha)|_{C_0} = \cO_{C_0}(-\sum_{h\to v_0} I(h)p_h) 
    \text{ and } \cO(\alpha)|_{C_v} = \cO_{C_v}(- \sum_{h\to v} I(h) p_h), 
\end{equation*}
with gluing data at node $e=(h,h')$ given by the section 
$u_e^{I(e)} \in \cO(I(e)\ell(e))$, which by~\ref{lem:normal_bundle_length} 
is identified with an element in 
\begin{equation*}
    {\rm Isom}(\cO_{C_0}(-I(h) p_h)|_{p_h}, \cO_{C_v}(-I(h') p_{h'})|_{p_{h'}}) 
\end{equation*}

Different $\underline{u}$ mapping to the same element in $\oM^{a}_{\Gamma,I}$ 
evidently give isomorphic bundles. 
Let $\mathcal{G}'$ be the group scheme $\bigoplus_E \mathbb{G}_m$, acting on 
$\oM^{a}_{\Gamma,I}$ coordinate-wise: $(\lambda_e)_e \cdot (u_e)_e = (\lambda_e u_e)_e$.
Let $\mathcal{G}$ be the quotient of $\mathcal{G}'$ by $\mathbb{G}_m$ acting on 
the $e$-th factor via an $L_e$-th power; then the action of $\mathcal{G}'$ 
descends to an action of $\mathcal{G}$ on $\oM^{a}_{\Gamma,I}$.
\begin{lemma}\label{lem:twisting_glueing}
Given $\lambda = (\lambda_e)_e \in \mathcal{G}'$ and $\underline{u} \in \oM^{a}_{\Gamma,I}$, the fibres 
\begin{equation*}
    \cO_{\underline{u}}(\alpha) \text{ and }\cO_{\lambda\cdot \underline{u}}(\alpha)
\end{equation*}
    are related by changing the gluing at node $e$ by $\lambda_e^{I(e)}$. 
\end{lemma}
\begin{lemma}
The subgroup of $\mathcal{G}'$ consisting of elements which act trivially on $\cO(\alpha)$ is given by $\prod_e \mu_{I(e)}$. The image of this subgroup in $\mathcal{G}$ is given by 
\begin{equation*}
\mu_I := (\prod_e \mu_{I(e)})/\mu_L, 
\end{equation*}
where the implicit map is given by 
\begin{equation*}
\mu_L \to \prod_e \mu_{I(e)}; u \mapsto (u^{L_e})_e. 
\end{equation*}
\end{lemma}

\begin{corollary} The action of $\mu_{I}$ restricts to an action on 
$\DRL_{\Gamma,I}$.
\end{corollary}

\subsection{Description of the locus $\oM^{1/2,a}_{\Gamma,I}$}

As before, similar arguments apply to $\oM^{1/2,a}_{\Gamma,I}$. 
In particular, if $(\Gamma,I)$ is a $k$-simple star
graph, the minimal log structure at all points of 
$\oM^{1/2,a}_{\Gamma,I}$ is also of rank $|V_{\rm out}|$. 
Indeed, by the description of the minimal log structure
of points of $\textbf{Div}$ given in~\cite[Theorem 4.2.4]{Marcus2017LogarithmicCO}
one can compute that of points
of $\oM_{\Gamma,I}^{1/2,a}$; restricting to a sufficiently small neighborhood of a geometric point $s\in S$ of an $S$-point of 
$\oM_{\Gamma,I}^{1/2,a}$,
the minimal log structure is isomorphic to
\begin{equation*}
    \mathbb{N}^{E(\Gamma)}\langle\sum_{e\in\gamma}
    \frac{1}{2\gcd(\gamma)}I^{\gamma}_e e\rangle_{\gamma \in Y},
\end{equation*}
which in turn is isomorphic to the minimal log structure of points
of $\oM_{\Gamma,I}^{a}$.

Let $(\Gamma,I)$ be a $k$-simple star graph with two vertices, and let $(J,w)$ be the associated half $k$-twist  given by~\ref{roottwist}.
We denote by $\overline{\delta}$ the generator of the minimal 
characteristic monoid. Let 
$\textbf{x}:=((\mathcal{C}\to S,\cL,\phi),\alpha)$
be a geometric $S$-point of $\oM^{1/2,a}_{\Gamma,I}$.
We will be
interested in the case where there exists an edge $e\in E(\Gamma)$ such 
that $I(e)\equiv0\mod 2$, which implies that the corresponding $2$-weighting $w$
of the associated pair $(J,w)$ (i.e., the reduction of $I$ modulo $2$) satisfies $w(e)=0$. 
We will make this assumption for the rest of this 
subsection. Moreover, we define
\begin{equation*}
    L_J:=\lcm_{e}(\underbrace{J(e)}_{e\ :\ w(e)=0},\underbrace{I(e)}_{e\ :\ w(e)=1}),
\end{equation*}
where $J(e)=I(e)/2$ for all edges such that $w(e)=0$.
By the correspondence in~\ref{roottwist}, we have $2L_J=L$.

Furthermore, without loss of generality, we may choose 
$\alpha$ such that it takes the value $0$ at the central vertex. 
By definition of the PL function $\alpha$, 
moving along an edge with $w(e)=0$, we have  
\begin{equation*}
    \alpha(v)=J(e)\ell(e),
\end{equation*}
where $v$ is the outer vertex. Similarly, moving along an edge with 
$w(e)=1$, we obtain
\begin{equation*}
    \alpha(v)=I(e)\ell(e)/2.
\end{equation*}
Then for all edges $e$ with
$w(e)=0$, we have
\begin{equation}\label{eq:w(e)=0}
    J(e)\ell(e)=\frac{L}{2}\overline{\delta},
\end{equation}
and since $2J(e)=I(e)$, we obtain $\ell(e)=L_e\overline{\delta}$. For the
edges $e$ with $w(e)=1$, we obtain that
\begin{equation}\label{eq:w(e)=1}
    I(e)\frac{\ell(e)}{2}=\frac{L}{2}\overline{\delta},
\end{equation}
and so $\frac{\ell(e)}{2}=\frac{L_e}{2}\overline{\delta}$.
We denote by $\overline{\mathcal{T}}$ the $\mathbb{G}_{m}$-torsor on $\oM_{\Gamma}$ given
by $\cO^{\times}(\overline{\delta})$.
\begin{lemma} \label{lem:fibre_rooted} Let $(\Gamma,I)$ be a $k$-simple star graph
with two vertices,
and let $w$ be the reduction of $I$ modulo $2$.
There is an isomorphism
\begin{equation*}
    \oM^{1/2,a}_{\Gamma,I}\stackrel{\sim}{\longrightarrow} \left(\bigoplus_{w(e)=0}
    \overline{\mathcal{T}}^{L_e}\oplus\bigoplus_{w(e)=1}\overline{\mathcal{T}}^{\frac{L_e}{2}}\right)/\mathbb{G}_{m},
\end{equation*}
where the action is given by
\begin{equation*}
    \lambda\cdot(u_{e})_{e\in E(\Gamma)}= \left\{\begin{matrix}
 \lambda^{L_{e}}u_{e}& w(e)=0 \\
 \lambda^{\frac{L_{e}}{2}}u_{e}& w(e)=1. \\
\end{matrix}\right.
\end{equation*}
\end{lemma}
\noindent The proof closely mirrors that of~\ref{loc:div}. 

\begin{proof} Let $x\colon X\to \oM_{g,n}^{1/2}$ be a log map. Then
$x$ lifts to $\oM^{1/2,a}_{\Gamma,I}$ if and only if the topological type is given by $(\Gamma,w)$ and the elements $x^*(I(e)\ell(e)/2)$
are the same in $\overline{M}_{X}$ for all $e\in E(\Gamma)$. 
We endow $X$ with the log 
structure $\NN\oplus\cO_{X}^{\times}$, and fix a schematic map
$x\colon \underline{X}\to \oM^{1/2}_{g,n}$ with topological type given by 
$(\Gamma,w)$. To lift a minimal log map 
$X\to \oM^{1/2,a}_{\Gamma,I}$
is to specify a section of $\cO_{X}^{\times}(-x^*(\ell(e)))$ for all edges
$e$ with $w(e)=0$ and a section of
$\cO_{X}^{\times}(-x^*(\ell(e)/2))$ for all edges
$e$ with $w(e)=1$. By equations \ref{eq:w(e)=0} and \ref{eq:w(e)=1}, we obtain
\begin{equation*}
    \cO_{X}^{\times}(-x^*\ell(e))=\cO_{X}^{\times}(-L_{e}
x^*\overline{\delta}),
\end{equation*}
and 
\begin{equation*}
    \cO_{X}^{\times}(-x^*\frac{\ell(e)}{2})=\cO_{X}^{\times}(-\frac{L_{e}}{2}
x^*\overline{\delta})
\end{equation*}
respectively.
The lemma follows by observing that the automorphisms of the log structure of $X$ act on $\cO_{X}^{\times}(n)$ by $n$-th powers.
\end{proof}

Let $\underline{u}=(u_e)_{e\in E(\Gamma)}$ be any representative of an 
equivalence class of $\oM^{1/2,a}_{\Gamma,I}$.
The fiber of $\cO(\alpha)$ over the point $\underline{u}$ satisfies
\begin{equation*}
    \cO(\alpha)|_{\mathcal{C}_0} = \cO_{\mathcal{C}_0}(-\sum_{h\to v_0} J(h) p_h)\ \text{ and }\ \cO(\alpha)|_{\mathcal{C}_v} = \cO_{\mathcal{C}_v}(- \sum_{h\to v} J(h) p_h), 
\end{equation*}
on the central and outer vertices respectively. If $e=(h,h')$ is such that 
$w(e)=0$ then the gluing data is given by $u_e^{J(e)}\in \cO(J(e)\ell(e))$,
which is identified with 
\begin{equation*}
    {\rm Isom}\left(\cO_{\cC_0}(-J(h) p_{h})|_{p_{h}}, \cO_{\cC_v}(-J(h') p_{h'})|_{p_{h'}}\right) 
\end{equation*}
by~\ref{lem:normal_bundle_length}. On the other hand, if $e$ is such that $w(e)=1$, then this edge is
subdivided into two edges on the quasi stable model,
corresponding to the two nodal points on the exceptional component.
Additionally, by the correspondence in~\ref{roottwist}, 
we obtain that $I(e)$ is an odd number. 
If we write $J(e)=(I(e)-1)/2$, then the gluing data at the two nodal 
points are given by 
\begin{equation*}
    u_{e}^{J(e)}\in\cO(J(e)\frac{\ell(e)}{2})\ \ \text{and}\ \  u_{e}^{J(e)+1}\in\cO((J(e)+1)\frac{\ell(e)}{2}) 
\end{equation*}
which are identified with the corresponding ${\rm Isom}$ spaces,
by~\ref{lem:normal_bundle_length}, as before.

\subsection{An action on $\DRL_{\Gamma,I}^{1/2}$} 
Let $e$ be an edge such that $I(e)\equiv0\mod 2$, and let $w$ be the 
$2$-weighting corresponding to the associated pair $(J,w)$. In particular, $w$
satisfies $w(e)=0$. We let $\mu_{I(e)}$ act on $\oM^{1/2,a}_{\Gamma,I}$ by scalar multiplication on the summand in the presentation of~\ref{lem:fibre_rooted} corresponding to $e$, in 
a similar fashion to~\ref{lem:twisting_glueing}
and in the text above it. 
Let $t_e\in \mu_{I(e)}$, and let 
\begin{equation*}
    \underline{u}\in\bigoplus_{w(e)=0}\overline{\mathcal{T}}^{L_e}\oplus\bigoplus_{w(e)=1}\overline{\mathcal{T}}^{\frac{L_e}{2}}.  
\end{equation*}
Then the fibers of   
\begin{equation*}
    \cO(\alpha)_{\underline{u}}\ \text{and}\ \cO(\alpha)_{t_{e}\cdot \underline{u}}
\end{equation*}
are related by changing the gluing at the node $e$ by $t_{e}^{J(e)}\in \mu_{2}$. The action by $\mu_{I(e)}$ 
restricts to an action on $\DRL_{\Gamma,I}^{1/2}$. Indeed, it suffices to observe that the map $\oM^{1/2,a}_{\Gamma,I}\to \oM^{a}_{\Gamma,I}$ squares the gluing data at all nodes $p$ corresponding to 
an edge $e_{p}$ such that $w(e_p)=0$, and that $t_{e}^{J(e)}\in \mu_{2}$.
Thus, we obtain  
\begin{equation*}
    \underline{u}\in \DRL^{1/2}_{\Gamma,I}\Rightarrow t_{e}\cdot\underline{u}\in \DRL^{1/2}_{\Gamma,I}.
\end{equation*}
\begin{lemma}\label{lem:exchpar} Let $(\Gamma,I)$ 
be a $k$-simple star graph with
two vertices
such that $I$ admits an even value at an edge $e$. Let $t_{e}\in \mu_{I(e)}$ 
such that $t_{e}=\zeta_{I(e)}^i$ for $i\equiv1\mod2$.
Then the isomorphism 
\begin{equation*}
    \theta_{t_{e}}\colon \DRL^{1/2}_{\Gamma,I}\to \DRL^{1/2}_{\Gamma,I} 
\end{equation*}
exchanges the two components  $\DRL^{1/2,+}_{\Gamma,I}$
and $\DRL^{1/2,-}_{\Gamma,I}$.
\end{lemma}

\begin{proof} We work over geometric points since the
parity is invariant in families. If we fix $\textbf{x}=((\mathcal{C},\cL,\phi),\alpha)\in\DRL^{1/2}_{\Gamma,I}$, 
we have to show that the parities of $\textbf{x}$ and $t_{e}\cdot \textbf{x}$
are opposite. In particular, the parity of $\textbf{x}$ is the parity 
of $\cL$ and so it suffices to show that the spin structure
corresponding to $t_{e}\cdot\textbf{x}$ has opposite parity. 
We write $i=2l+1$, where $l\in \ZZ$, and recall that 
$2J(e)=I(e)$. Then we have
\begin{equation*}
    t_{e}^{J(e)}=\left(\zeta_{I(e)}^{2l+1}\right)^{\frac{I(e)}{2}}
    =\left(\zeta_{I(e)}^{\frac{I(e)}{2}}\right)=-1
\end{equation*}
Hence, $t_{e}$ acts on the gluing of $\cO(\alpha)$
at the node corresponding to $e$ by changing 
the sign. We write $\cF$ 
to denote the line bundle 
\begin{equation*}
    \omega_{\rm log}^{\frac{1-k}{2}}\left(\sum_{i=1}^n\frac{a_i-1}{2}\right)\otimes\cO(\alpha)^{\vee}.
\end{equation*}
In particular, since $\textbf{x}\in \DRL^{1/2}_{\Gamma,I}$, we have that 
$\cF\cong \cL$. Recall that the associated $2$-weighting $w$
satisfies $w(e)=0$. Now, similarly to the proof of~\ref{prop:cancellation},
since we are interested in parity calculations, we may pass to the
non-exceptional sub-curve of $\mathcal{C}$.
Hence, passing to the non-exceptional sub-curve $\widetilde{\cC}$, if 
\begin{equation*}
    \nu_{e}\colon \widetilde{\cC}'\to \widetilde{\cC}    
\end{equation*}
is the partial normalization of the node corresponding
to $e$ we obtain that $\nu^*\widetilde{\cF}$ is an honest square root of 
$\omega_{\widetilde{\cC}'}(p_h+p_{h'})$, where $\widetilde{\cF}$ is
the restriction of $\cF$ to $\widetilde{\cC}$.
By the same reasoning
as in~\ref{prop:cancellation} there are $2$ possible identifications of $\nu^*\widetilde{\cF}$
resulting in $2$ different spin structures, which have opposite parities.
The two possible identifications differ up to sign and the action we described 
switches from one identification to the other.
\end{proof}

\begin{corollary} Let $(\Gamma,I)$ be a $k$-simple
star graph as in the lemma above.
Then the following equality holds in $A^*(\oM_{g,n})$
    \begin{equation*}
        \rho_{*}\left(
 \epsilon_{a*}\left(\DR^{1/2}_{\Gamma,I}\cdot\rho^*_{\frac{1}{2}}[\pm]\right)\right)=0.
    \end{equation*}
\end{corollary}

\subsection{Multiple outlying vertices}
Suppose now that $(\Gamma,I)$ is a $k$-simple star graph with more than one
outlying vertex, and let $w$ be the $2$-weighting of the pair $(J,w)$
associated to $I$. Then the log structure of 
$\oM_{\Gamma,I}^{1/2,a}$ is free of rank equal 
to $|V_{\rm out}|$. 
Working as before, we obtain a generator 
$\overline{\delta}_{v}$ for each $v\in V_{\rm out}$. Let $\overline{\mathcal{T}}_{v}$ denote the $\mathbb{G}_{m}$-torsor on $\oM_{\Gamma}$ given by $\cO^{\times}(\overline{\delta}_{v})$.
By applying a similar analysis as for the case of 
$k$-simple star graphs with two vertices, we have an isomorphism
\begin{equation*}
    \oM^{1/2,a}_{\Gamma,I}\stackrel{\sim}{\longrightarrow} \prod_{v\in V_{\rm out}}\left(\bigoplus_{\substack{e\to v \\ w(e)=0}}
    \overline{\mathcal{T}}_{v}^{L^v_e}\oplus\bigoplus_{\substack{e\to v \\ w(e)=1}}\overline{\mathcal{T}}_{v}^{\frac{L^v_e}{2}}\right)/\mathbb{G}_{m},
\end{equation*}
where $L^v=\lcm_{e\to v}(I(e))$ and $L^v_{e}=L^{v}/I(e)$. 
The action described above and the corollaries carry over unchanged for such graphs.

\begin{corollary}\label{DRcancel} Let $(\Gamma,I)$ be a
$k$-simple star graph such that the $k$-twist $I$ admits an even value at an
edge $e$.
Then the following equality holds in $A^*(\oM_{g,n})$
    \begin{equation*}
        \rho_{*}\left(
 \epsilon_{a*}\left(\DR^{1/2}_{\Gamma,I}\cdot\rho^*_{\frac{1}{2}}[\pm]\right)\right)=0.
    \end{equation*}
\end{corollary}

\subsection{Proof of~\ref{thm:DR}}
In this section we prove~\ref{thm:DR}. We start with a 
simple remark about the classes $[\oM_{\Gamma,I}]^{\pm}$.

\begin{remark}\label{rem:tensor-decomp} Using standard theory of tensor products,
we can decompose the class $[\oM_{\Gamma,I}^{\pm}]$ as follows;
\begin{align*}
    [\oM_{\Gamma,I}]^{\pm}=\sum_{\zeta\colon V(\Gamma)\to\{1,-1\}}
    \left(\prod_{v\in V}\zeta(v)\right) \times[\oM_{g(v_0)}(I(v_0),k)^{\zeta(v_0)}]\otimes
    \prod_{v\in V_{\rm out}}[\oM_{g(v)}(I(v)/k)^{\zeta(v)}].
\end{align*}
\end{remark}

 Now, consider the following commutative diagram:
\begin{equation*}
    \xymatrix{\DRL^{1/2} \ar[r]^{\rho_{\frac{1}{2}}}\ar[d]_{\epsilon_{a}} & \oM^{1/2}_{g,n}\ar[d]^{\epsilon} \\
            \DRL\ar[r]_{\rho} & \oM_{g,n}.}
\end{equation*}
Then, by the commutativity of the diagram above, we obtain
\begin{align*}
    \DR^{\pm}_{g}(a)&=2\epsilon_{*}\left(\rho_{\frac{1}{2}*}\mathrm{DRC}
    ^{1/2}\cdot[\pm]\right) \\
                    &=2\epsilon_{*}\left(\rho_{\frac{1}{2}*}\left((\mathrm{DRC}
    ^{1/2}\cdot\rho_{\frac{1}{2}}^*[\pm]\right)\right) \\
                    &=2\rho_{*}\left(\epsilon_{a*}\left(\mathrm{DRC}
    ^{1/2}\cdot\rho_{\frac{1}{2}}^*[\pm]\right)\right).
\end{align*}
By definition, the cycle associated to $\DRL^{1/2}$ is given by
\begin{equation*}
    [\DRL^{1/2}]=\sum_{(\Gamma,I)\in \mathrm{SStar}_{g}(a,k)}
    \DR^{1/2}_{\Gamma,I}.
\end{equation*}
\begin{lemma} The cycle associated to the closed substack $[\DRL^{1/2}]$ coincides with the cycle ${\rm DRC}^{1/2}$. 
\end{lemma}
\begin{proof} First we recall that, by~\ref{lem:l.c.i-charts},
the local rings of $\oM_{g}^{a}$ at generic points 
of irreducible components of $\DRL$ are local complete 
intersections, and hence Cohen-Macaulay.
Moreover, by~\ref{mu_2DR:gerbe}, the local rings
at generic points of $\DRL$ are isomorphic to those of $\DRL^{1/2}$.
Then it follows from~\cite[Proposition 7.1b)]{Fulton} that
the lengths at generic points of irreducible components of 
$\DRL^{1/2}$ agree with the cycle
theoretic multiplicity of $\mathrm{DRC}^{1/2}$, and so we obtain 
\begin{equation*}
    \mathrm{DRC}^{1/2}=[\DRL^{1/2}].
\end{equation*}
\end{proof}
\vspace{-3mm}

Furthermore, using the lemma above, we obtain 
\begin{align*}
    \rho_{*}\left(\epsilon_{a*}\left(\mathrm{DRC}
    ^{1/2}\cdot\rho_{\frac{1}{2}}^*[\pm]\right)\right)&=\sum_{(\Gamma,I)\in\mathrm{SStar}_{g}(a,k)}
 \rho_{*}\left(\epsilon_{a*}\left(\DR^{1/2}_{\Gamma,I}\cdot\rho^*_{\frac{1}{2}}[\pm]\right)\right).
\end{align*}
However, as we proved in the previous subsection, the contribution
above arising from $k$-simple star graphs $(\Gamma,I)$ 
such that the $k$-twist $I$ admits an even value
on an edge vanishes (see~\ref{DRcancel}). 
Hence, we can restrict to the case that
$I$ takes only odd values.
We will study the contribution of the even and odd components individually. 
The fact that the associated $2$-weighting $w$ 
is the constant function $1$ allows us to describe the 
parity of a point in $\DRL^{1/2}_{\Gamma,I}$. Indeed, in this case, i.e.,
if a geometric point $((\mathcal{C},\cL,\phi),\alpha)$ 
of $\DRL^{1/2}$ has topological type given by $(\Gamma,w)$, we have
\begin{equation*}
    h^0(\cC,\cL)=\sum_{v\in V(\Gamma)}h^0(\cC_v,\cL|_{C_v}),
\end{equation*}
where $\cC_v$ runs over the non-exceptional components of $\cC$ 
(see~\ref{rem:exc-parity}).
Thus, the line bundle $\cL$ defines a parity function
\begin{align*}
    \zeta_{\cL}\colon V(\Gamma)&\to \{\pm1\} \\ 
    v&\mapsto (-1)^{h^0(\cC_v,\cL|_{\cC_v})}.
\end{align*}
In multiplicative notation, the parity of $\cL$ is even if and only if 
$\prod_{v\in V(\Gamma)}\zeta_\cL(v)=1$, and odd otherwise.
Additionally, if 
the slopes of $\alpha$
are given by a half $k$-twist $J$ on the dual graph of $\cC$
we have
\begin{equation*}
    \cL|_{\cC_v}\cong \cO_{\cC_v}\left(\sum_{h\in H(v)}J(h)
    \right)\otimes\omega_{\rm log}^{\frac{1-k}{2}}
\end{equation*}
for all vertices $v\in V(\Gamma)$. Furthermore,  
using the reuslts of~\cite[Section 2.6]{HolSch} and~\ref{mu_2DR:gerbe},  we have that both $\DRL$ and $\DRL^{1/2}$ map surjectively 
to $\widetilde{\cH}_{g}^{k}(a)$. Moreover,
the composition
\begin{equation*}
    \DRL^{1/2}\to \DRL\to\widetilde{\cH}_{g}^{k}(a)
\end{equation*}
restricts to
\begin{equation*}
    \DRL^{1/2}_{\Gamma,I}\to \DRL_{\Gamma,I}\to \zeta_{\Gamma}(\oM_{\Gamma,I}).
\end{equation*}
In particular, this implies that if $v\in V_{\rm out}(\Gamma)$
we also have that 
\begin{equation*}
    \cO_{\cC_v}\left(\sum_{h\in H(v)}\frac{2J(h)+1}{k}\right)\cong \omega_{\rm log}.
\end{equation*}
Hence, in that case, combining the above, we obtain
\begin{equation*}
    \cL|_{\cC_v}\cong \cO_{\cC_v}\left(\sum_{h\in H(v)}\frac{2J(h)+1-k}{2k}
    \right).
\end{equation*}
Therefore, the locus of points in $\DRL_{\Gamma,I}^{1/2}$ 
with parity function $\zeta$ maps onto the closed subset
\begin{equation*}
    \zeta_{\Gamma}\left(\oM_{g(v_0)}(I(v_0),k)^{\zeta(v_0)}\times\prod_{v\in V_{\rm out}}\oM_{g(v)}(I(v)/k)^{\zeta(v)}\right).
\end{equation*}
Suppose we consider  
the even component $\DRL^{1/2,+}_{\Gamma,I}$. The pushforward 
$2\rho_*(\epsilon_{a*}([\DRL^{1/2,+}_{\Gamma,I}]))$ will be a multiple of 
the sum 
\begin{equation*}
    \sum_{\substack{\zeta\colon V(\Gamma)\to\{1,-1\} \\ \prod_{v\in V(\Gamma)}\zeta(v)=1}}\frac{1}{|\Aut(\Gamma)|}
    \zeta_{\Gamma*}\left([\oM_{g(v_0)}(I(v_0),k)^{\zeta(v_0)}]\otimes
    \prod_{v\in V_{\rm out}}[\oM_{g(v)}(I(v)/k)^{\zeta(v)}]\right).
\end{equation*}
On the other hand, using a similar argument for the 
odd component, we obtain that the pushforward $2\rho_*(\epsilon_{a*}([\DRL^{1/2,-}_{\Gamma,I}]))$
will be a multiple of the sum
\begin{equation*}
     \sum_{\substack{\zeta\colon V(\Gamma)\to\{1,-1\} \\ \prod_{v\in V(\Gamma)}\zeta(v)=-1}}\frac{1}{|\Aut(\Gamma)|}
    \zeta_{\Gamma*}\left([\oM_{g(v_0)}(I(v_0),k)^{\zeta(v_0)}]\otimes
    \prod_{v\in V_{\rm out}}[\oM_{g(v)}(I(v)/k)^{\zeta(v)}]\right).
\end{equation*}
Altogether, by~\ref{rem:tensor-decomp}, we conclude 
that the pushforward $2\rho_{*}\left(\epsilon_{a*}\left(\mathrm{DRC}
^{1/2}\cdot\rho_{\frac{1}{2}}^*[\pm]\right)\right)$ will be of the form
\begin{equation}\label{eq:form-of-spDR}
    \sum_{(\Gamma,I)\in\text{SStar}^{\rm odd}_{g}(a,k)}c_{\Gamma,I}\frac{1}{|\Aut(\Gamma)|}\zeta_{\Gamma*}[\oM_{\Gamma,I}]^{\pm},
\end{equation}
where $\text{SStar}_{g}^{\rm odd}(a,k)$ is the subset of 
$k$-simple star graphs
with only odd values.

\begin{theorem}[\ref{thm:DR}]~\cite[Conjecture 2.5]{CosSauSch} Let $k\in\ZZ_{\geq1}$ be
odd, and let $a\notin k\ZZ^n_{>0}$ be a vector of odd integers 
such that $|a|=k(2g-2+n)$. 
Then we have 
\begin{equation*}
     \mathsf{H}_{g}^{\pm}(a,k)=\DR^{\pm}_{g}(a,k).
\end{equation*}
\end{theorem}
\begin{proof}  As we mentioned above, the class $\DR^{\pm}_{g}(a,k)$
is of the form~\ref{eq:form-of-spDR}.
Therefore, we only have to compute the numbers
$c_{\Gamma,I}$. By definition of the proper pushforward,
these numbers are the sum of the lengths of the Artin local
rings at generic points of $\DRL^{1/2}$ 
over the generic points of 
$\zeta_{\Gamma}(\oM_{\Gamma,I})\subset
\widetilde{\cH}_{g}^{k}(a)$. Using~\ref{mu_2DR:gerbe} 
we reduce this computation to computing 
the lengths at such generic points of $\DRL$. Then applying~\cite[Lemma 2.12]{HolSch}
on the proper birational morphism 
$\oM_{g}^{a}\to \oM_{g,f}^{a}$ 
reduces this computation to computing the length of the Artin local rings at generic points
of $\DRL_{f}$.
In that case, these lengths are computed 
in~\cite{HolSch} and they are equal to
\begin{equation*}
    c_{\Gamma,I}=\frac{\prod_{e\in E(\Gamma)}I(e)}{k^{|V_{\rm out}|}}.
\end{equation*}
\end{proof}

\vspace{-4mm}
\section{Tautological calculus of moduli spaces of differentials}   The overall strategy to compute classes of strata of differentials developed in~\cite{Sau} imposes to work with an alternative compactification (the incidence variety compactification) than the one presented in the introduction.  Besides,  we need to consider strata of differentials supported on disconnected curves,  and with linear conditions on residues at the poles.  We recall this language here and prove \ref{lem:computation3}: the main~\ref{thm:main} is a consequence of the spin DR cycle formulas (\ref{thm:DR} and~\ref{spDR = spPix} proven above) and the expression of signed Segre classes of cones of spin sections (\ref{thm:segre}) that will be established in the last section. The last part of the section is dedicated to proving the identity~\ref{for:pm1main} that will be required in the next section for the proof of~\ref{thm:segre}. 

\subsection{Incidence variety compactification,  and generalized strata} 

If $P=(p_1,\ldots,p_n)$ is a vector of non-negative integers, then we denote by $\oH_{g}[P]\to \oM_{g,n}$ the vector bundle with fiber over $(C,x_1,\ldots,x_n)$ given by
$$H^0(C,\omega_C(p_1x_1+\ldots +p_nx_n)).$$
This space is the moduli space of marked curves with a meromorphic differential, and generalizes the Hodge bundle as $\oH_{g,n}=\oH_g[0,\ldots]$. Let $\ell\in \NN^*,$ and let $\bg=(g^1,\ldots,g^\ell)$ and $\bn=(n^1,\ldots,n^\ell)$ be vectors of non-negative integers satisfying $2g^j-2+n^j>0$ for all $j\in \{1,\ldots,\ell\}$.    Besides,  let $\mathbf{P}=(P^j=(p^j_i)_{1\leq  i\leq n^j})_{ 1\leq j\leq \ell}$ be a vector of vectors of non-negative integers.  We denote
 $$
 \M_{\bg,\bn}= \prod_{j=1}^\ell \M_{g^j,n^j}, \,\text{ and }  \oH_{\mathbf{g}}[\mathbf{P}]=  \prod_{j=1}^\ell \oH_{g^j}[P^j].
$$ 
If we fix a vector of vectors of integers $\ba=(a^j=(a_i^j)_{1\leq  i\leq n^j})_{ 1\leq j\leq \ell}$   such that $-p^j_{i} \leq a^j_{i} -1 $ for all pairs $(j,i)$,  then we denote by  $\H_\bg(\ba)$ the sub-cone of $\oH_{\bg}[\bP]$ of differentials supported on smooth curves and with singularities of order $a^j_{i}-1$ at $x^j_{i}$ for all pairs $(j,i)$. The {\em incidence variety} $\oH_\bg(\ba)$ is  the Zariski closure of $\H_\bg(\ba)$ in $\H_\bg[\bP]$ (different choices of $\bP$ provide isomorphic cones).  

\subsubsection*{Residue conditions} We denote by ${\rm Pol}(\ba)$ the set of pairs $(j,i)$ such that $a_{i}^j\leq 0$.  The space of residues $\fR(\ba)$ is the linear subspace of $\CC^{|{\rm Pol}(\ba)|}$ defined as the set of vectors $(r^j_{i})_{{\rm Pol}(\ba)}$ satisfying 
$$
\sum_{\begin{smallmatrix} 1\leq i\leq n^j,  \\ \text{s. t. } (j,i)\in {\rm Pol}(\ba) \end{smallmatrix}} r^j_{i} = 0, \text{  for all $j\in \{1,\ldots,\ell\}$.}
$$
A {\em space of residue conditions} is a linear subspace of $R\subset \fR(\ba)$ defined by linear conditions of the form $\sum_{(j,i)\in E} r^{j}_{i} = 0$ for a strict subset $E$ of ${\rm Pol}(\ba)$.  We denote by $\oH_{\bg}(\ba,R)\subset \oH_{\bg}(\ba)$ the sub-cone of differentials with residues in $R$.

\subsection{Bi-colored graphs} Here we recall the description of (part) of the boundary of $\oH_\bg(\ba,R)$.  First,  a marked stable graph $\Gamma$ of type $(\bg,\bn)$ is the datum of stables graphs $\Gamma^1,\ldots,\Gamma^\ell$ such that the legs of $\Gamma^j$ are labeled $(j,i)$ for  $i\in \{1,\ldots,n_j\}$, and the genus of $\Gamma^j$ is $g^j$. As for connected stable graphs, we have a gluing morphism
$
\zeta_\Gamma \colon \oM_{\Gamma} = \prod_{v\in V} \oM_{g(v),n(v)} \to \oM_{\bg,\bn}.
$ 
 Moreover, a twist on a disconnected graph is a twist on each component. 

\begin{definition}
A {\em bi-colored} graph  is a twisted graph $(\Gamma,I)$ together with a non-trivial partition of the set of vertices  $V=V_0\sqcup V_{-1}$ such that:
\begin{enumerate}
    \item For each vertex $v$ we have $\sum_{h\mapsto v} I(h) \leq 2g(v)-2+n(v)$.
    \item If $(h,h')$ is an edge, then $I(h)\neq 0$. If we assume that $I(h)>0,$ then $h$ is incident to a vertex in $V_0$ and $h'$ to a vertex in $V_{-1}$. 
\end{enumerate}  The {\em multiplicity} of $\oGamma=(\Gamma,I,V_0)$ is the integer
$
m(\oGamma)=\prod_{(h,h') \in E} \sqrt{-I(h)I(h')}.
$
\end{definition}

A bi-colored graph $\oGamma$  determines vectors $\bg[i]$ and $ \bn[i]$ for $i=0$ or $-1$ (the genera and number of half-edges at each vertex of level $i$),  while the twist function determines vectors $\ba[i]$.
As explained in~\cite{BCGGM} and~\cite{Sau}, given a space of residue condition $R$ (possibly the full space of residues), we can construct a boundary component 
$$\zeta_{\oGamma}\colon \oH(\oGamma,R)\to  \oH_\bg(\ba,R).
$$
The space is $\oH(\oGamma,R)$ is isomorphic to 
$$
\oH_{\bg[0]}(\ba[0],R[0]) \times p_{-1} \left(\PP\oH_{\bg[-1]}(\ba[-1],R[-1]) \right) \subset \oH_{\bg[0]}(\ba[0]) \times \oM_{\bg[-1],\bn[-1]},
$$
where $p_{-1}\colon \PP\oH_{\bg[-1]}(\ba[-1])\to \oM_{\bg[-1],\bn[-1]}$ is the forgetful morphism of the differential, while the vector spaces $R[i]$ are determined explicitly by $R$.  The morphism $\zeta_{\Gamma}$ associates to each point, a differential that vanishes on the vertices of level $-1$. It is finite of degree $|{\rm Aut}(\oGamma)|$.

\subsection{Intersection formulas including spin signs}\label{ssec:intersectionformulas}  From now on,  we assume that $\ba$ is odd. 
Generalizing the notation of the introduction, we set $\mathcal{SQ}_\bg(\ba,R)\subset \oH_\bg(\ba,R)$ to be the locus of differentials with singularities (zeros or poles) of even orders. We denote by $\overline{\mathcal{SQ}}_\bg(\ba,R)$ its closure. As the space ${\mathcal{SQ}}_\bg(\ba,R)$ admits a partition according to parity, we denote by 
$$[\PP\overline{\mathcal{SQ}}_\bg(\ba,R)]^\pm = [\PP\overline{\mathcal{SQ}}_\bg(\ba,R)^+] - [\PP\overline{\mathcal{SQ}}_\bg(\ba,R)^-].\footnote{This class will be  considered as a class in  $A^*(\PP\oH_\bg[\bP],\QQ)$ for sufficiently large $\bP$ as we will only be interested in classes of the form $p_*\left(\xi^q[\PP\overline{\mathcal{SQ}}_\bg(\ba,R)]^\pm\right)$ that do not depend on $\bP$.}$$
 If $\oGamma$ is an odd bi-colored graph $\oGamma$ (the twist only takes odd values), then we denote 
\begin{eqnarray*}
    \overline{\mathcal{SQ}}(\oGamma,R) &=&\overline{\mathcal{SQ}}_{\bg[0]}(\ba[0],R[0]) \times p_{-1} \left(\PP\overline{\mathcal{SQ}}_{\bg[-1]}(\ba[-1],R[-1]) \right) \subset  \overline{\mathcal{H}}(\oGamma,R), \\
    \text{and } [\PP\overline{\mathcal{SQ}}(\oGamma,R)]^\pm &=&\left[\PP\overline{\mathcal{SQ}}_{\bg[0]}(\ba[0],R[0])\right]^\pm \otimes p_{-1*}\left[\PP\overline{\mathcal{SQ}}_{\bg[-1]}(\ba[-1],R[-1]) \right]^\pm.
\end{eqnarray*}
We emphasize that the locus $\PP\overline{\mathcal{SQ}}(\oGamma,R)$ can be of co-dimension greater than 1 in $\PP\overline{\mathcal{SQ}}_\bg(\ba,R)$. However, the class $[\PP\overline{\mathcal{SQ}}(\oGamma,R)]^\pm$  is the (weighted) Poincar\'e-dual class of $\PP\overline{\mathcal{SQ}}(\oGamma,R)$ if it is of co-dimension 1 and 0 otherwise. Finally, we denote by ${\rm Bic}_{\bg,\bn}(i,j)$ the set of odd bi-colored graphs such that the leg $(j,i)$ is incident to a vertex of level $-1$.   Besides, if  $R'\subset R$ is a sub-space of residue conditions of co-dimension 1,  then we denote by ${\rm Bic}_{\bg,\bn}(R')$ the set of bi-colored graphs such that $\PP\oH(\oGamma,R)$ lies in $\PP\oH_\bg(\bP, R')$ (i.e. the linear residue constraints defined by $R'$ are ensured by the topological type of the graph). 

\begin{proposition} \cite[Propositions~5.9 and 5.10]{Wong} \label{prop:induction1} Here we assume that $\ba$ is {\em full}, i.e. for all $j$ we have $|a^j|=2g^j-2+n^j$. Then for all labels $(j,i)$, we have
\begin{equation}\label{for:induction1}
(\xi + a^j _{i}  \psi^j_{i})  [\PP\overline{\mathcal{SQ}}_\bg({\ba},R)]^\pm =  \sum_{\oGamma \in {\rm Bic}_{\bg,\bn}(i,j)} \frac{m(\oGamma)}{|\Aut(\oGamma)|} \zeta_{\oGamma\, *} [\PP\overline{\mathcal{SQ}}(\oGamma,R)]^\pm
\end{equation}
(here $\psi^{j}_i$ is the class of the co-tangent line the marking with label $(j,i)$). Besides, if $R'\subset R$ is a sub-space of residue conditions of co-dimension $1$,  then  
\begin{equation}\label{for:induction2}
\xi   [\PP\overline{\mathcal{SQ}}_\bg({\ba},R)]^\pm = [\PP\overline{\mathcal{SQ}}_\bg({\ba},R')]^\pm + \sum_{\oGamma \in {\rm Bic}_{\bg,\bn}(R')} \frac{m(\oGamma)}{|\Aut(\oGamma)|} \zeta_{\oGamma\, *} [\PP\overline{\mathcal{SQ}}(\oGamma,R)]^\pm.
\end{equation}
\end{proposition}
The first consequence of this proposition is that classes of strata of meromorphic differentials with poles and residue constraints are determined by classes of strata of holomorphic differentials, namely.

\begin{lemma}\label{lem:computation2}
Let $g_0\geq 0$. We assume that all the classes $p_*[\PP\oSQ_g(a)]^\pm$ with $g\leq g_0$ are tautological and computable.  Then the class 
$p_*\left(\xi^{q}\PP\oSQ_{[\bg]}(\ba,R)^\pm\right)$ is tautological and can be explicitly computed if the associated vector of genera $\bg$ has coordinates at most equal to $g_0$. 
\end{lemma}

\begin{proof} The images of tautological classes under push-forward along the forgetful morphism of markings are tautological and can be explicitly computed.  Therefore, we may assume that $\ba$ is full. We set $\chi(\ba)=\sum_{j=0}^\ell (2g_j-2+n_j)$.   If $\oGamma$ is a bi-colored graph compatible with $\ba$, then $\chi(\ba[0])$ and $\chi(\ba[-1])$ are smaller than $\chi(\ba)$,  so we will prove the lemma by induction on $\chi(\ba)$.  The base is given by the a single connected component with $(g,n)=(0,3)$: in this case the class is $1$ if $q=0$ and the residue conditions are trivial (i.e.  $\fR=R$), and 0 otherwise. 
\smallskip

The first step is to reduce the statement to the case $R=\fR$ by induction on the co-dimension of $R$  by using formula~\ref{for:induction2}
$$
    \xi^{q}  [\PP\overline{\mathcal{SQ}}_\bg({\ba},R')]^\pm =   \xi^{q+1}  [\PP\overline{\mathcal{SQ}}_\bg({\ba},R)]^\pm + \!\!\!\! \sum_{\oGamma \in {\rm Bic}_{\bg,\bn}(R')} \frac{m(\oGamma)}{|\Aut(\oGamma)|} \xi^{q}\zeta_{\oGamma\, *} [\PP\overline{\mathcal{SQ}}(\oGamma,R)]^\pm,
$$
where $R'\subset R$ is of co-dimension $1$. If the coordinate of $\bg$ are smaller than or equal to $g_0$, then the coordinates of $\bg[0]$ and $\bg[-1]$ are also smaller than or equal to $g_0$.  Then, for each term $\oGamma$ in the sum,  we use the projection formula
\begin{align*}
    p_*\left(\xi^{q-1}\zeta_{\oGamma\, *} [\PP\overline{\mathcal{SQ}}(\oGamma,R)]^\pm\right) \\   &&\!\!\!\!\!\!\!\!\!\!\!\!\!\!\!\!\!\!\!\!\!\!\!\!\!\!\!\!\!\!\!\! =  \zeta_{\Gamma\, *} \left(p_{0\, *}\left(\xi^{q-1}  [\PP\oSQ_{\bg[0]}(\ba[0],R[0])]^\pm\right) \otimes p_{-1\, *}  [\PP\oSQ_{\bg[1]}(\ba[-1],R[-1])]^\pm \right).
\end{align*}
The above equality is valid because we have $p \circ \zeta_{\oGamma}=\zeta_\Gamma \circ p_0$ and $ \zeta_{\oGamma}^*\mathcal{O}(1)\simeq \mathcal{O}(1)$. Finally, we reduce the to the case $q=0$ using the induction formula~\ref{for:induction1}: 
$$
    \xi^q  [\PP\overline{\mathcal{SQ}}_\bg({\ba})]^\pm =  - a^j _{i}  \psi^j_{i} p_*\left(\xi^{q-1}  [\PP\overline{\mathcal{SQ}}_\bg({\ba})]^\pm\right) + \!\!\!\! \sum_{\oGamma \in {\rm Bic}_{\bg,\bn}(i,j)} \frac{m(\oGamma)}{|\Aut(\oGamma)|} \xi^{q-1}\zeta_{\oGamma\, *} [\PP\overline{\mathcal{SQ}}(\oGamma)]^\pm.
$$
\end{proof}

\subsection{Relations with non-marked zeros}  Here, we extend the relation~\ref{for:induction1} to certain non-full vectors. Until the end of the section we assume that $\ell=1$, $a_1\geq -1,$ and $a_j>0$ for all $j>i$, i.e. we allow only one pole of order at most $2$ at the first marking. We set $N=g-1-(|a|-n)/2$: it is the number of non-marked zeros of order $2$ for a generic differential, and, we denote by $\pi_N\colon  \oM_{g,n+N}\to \oM_{g,n}$ the forgetful morphism of the last $N$ markings, and by $\widetilde{a}$ the vector  obtained from $a$ by adding $N$ times $3$.  The morphism $\pi_N$ lifts to a finite morphism $\pi_N\colon\PP\oSQ_{g}(\widetilde{a})\to\PP\oSQ_g(a)$ of degree $N!$ for which $\pi_N^*(\mathcal{O}(1))=\mathcal{O}(1)$. 

\begin{proposition}\label{pr:indformulanonmarked}
    In this setup, we have 
    \begin{equation}\label{for:induction3}
(\xi + a_{1}  \psi_{1})  [\PP\overline{\mathcal{SQ}}_g(a)]^\pm = 2\,[\PP\overline{\mathcal{SQ}}_g(a')]^\pm + \frac{1}{N!}\sum_{\oGamma \in {\rm Bic}_{g,n+N}^*(1)} \frac{m(\oGamma)}{|\Aut(\oGamma)|} \pi_{N*} \zeta_{\oGamma\, *} [\PP\overline{\mathcal{SQ}}(\oGamma)]^\pm.
\end{equation}
where $a'=(3,a_2,\ldots)$ if $a_1=-1$, and $(a_1+2,a_2,\ldots)$ otherwise. Besides, in this formula ${\rm Bic}_{g,n+N}^*(i)\subset {\rm Bic}_{g,n+N}(i)$ stands for set of bi-colored graphs for which the stabilization of the graph obtained by removing the last $N$ legs is not trivial.
\end{proposition}

\begin{proof}  If $E\subset \{2,\ldots,n+N\}$ is a non-trivial subset, then we denote by $\Gamma_E$ the stable graph with one edge separating a genus $0$ vertex carrying the legs in $\{1\}\cup E$. There is only one bi-colored graph $\overline{\Gamma}_E$ with underlying stable graph $\Gamma_E$: the vertex of genus 0 is in $V_{-1}$, and the twist at the unique edge is equal to $$
a_E=a_1+2 |E\cap \{n+1,\ldots,n+N\}| + \sum_{j \in E\cap \{2,\ldots,n\}} (a_j-1) . 
$$ For this graph, we have 
$$
\PP\overline{\mathcal{SQ}}(\oGamma_E)\simeq \PP\overline{\mathcal{SQ}}_g(\ldots,a_E,\ldots) \times Z_{-1},
$$
where $Z_{-1}\subset \oM_{0,2+|E|}$ is a locus of co-dimension 1 if $a_1=-1$ (and it  and $0$ otherwise.
If $E$ is included in $\{n+1,\ldots,n+N\}$, then after forgetting the last $N$ markings we obtain:
\begin{equation}\label{for:proofpositive1}
\pi_{N*}\, \zeta_{\oGamma_E*}[\PP\overline{\mathcal{SQ}}(\oGamma_E)]^\pm = \left\{\begin{array}{cl}(N-1)! [\PP\overline{\mathcal{SQ}}_g(a')] &\text{ if $a_1>0$ and $E$ is a singleton,}\\
(N-2)! [\PP\overline{\mathcal{SQ}}_g(a')] &\text{ if $a_1=-1$ and $E$ is of size 2,}\\0 & \text{otherwise.} \end{array} \right.
\end{equation}

Let $\oGamma$ be a bi-colored graph in ${\rm Bic}_{g,n+N}(1)\setminus {\rm Bic}^*_{g,n+N}(1)$. We claim that the class of $\PP\oSQ(\oGamma)$ vanishes unless $\oGamma=\oGamma_E$ for some set $E$. Indeed, as the $a_j$ are positive for $j\neq 1$,  there can be no vertices of genus 0 in $V_0$. Therefore the graph $\oGamma$ is obtained by attaching a genus $g$ 
 vertex in $V_0$ to vertices of genus 0 in $V_{-1}$. The locus $\PP\overline{\mathcal{SQ}}(\oGamma)$ is of co-dimension equal to the number of vertices of genus 0 in $\PP\overline{\mathcal{SQ}}_g(\widetilde{a})$ (by a dimension count). So $[\PP\overline{\mathcal{SQ}}(\oGamma)]^\pm=0$ if there are at least two vertices of genus 0. Then, the last identities needed to complete the proof of the proposition are
\begin{align} \label{for:proofpositive2}
    \pi_N^*\psi_1=\, \, &\psi_1 + \sum_{E} \zeta_{\Gamma_E*}[\oM_{\Gamma_E}], \text{ and } \\ 
    \label{for:proofpositive3}
     \zeta_{\Gamma_E}^* [\PP\overline{\mathcal{SQ}}_g(\widetilde{a})]^\pm =\, \,  & [\PP\overline{\mathcal{SQ}}(\oGamma_E)]^\pm.
\end{align}
 The second identity follows from the fact that $\zeta_{\Gamma_E}(\oM_{\Gamma_E})$ and $\PP\overline{\mathcal{SQ}}_g(a)$ intersect transversally along $\zeta_{\oGamma_E}(\PP\overline{\mathcal{SQ}}(\oGamma_E))$ (see~\cite{CMSZ}). With these identities at hand, we apply formula~\ref{for:induction1} to $\widetilde{a}$, and then the operator $\frac{1}{N!}\pi_{N*}$. The left-hand side is then given by
\begin{eqnarray*}
    \frac{1}{N!}\pi_{N*}(\xi + a_1 \psi_1)[\PP\overline{\mathcal{SQ}}_g(\widetilde{a})]^\pm &=& (\xi+a_1 \psi_1) [\PP\overline{\mathcal{SQ}}_g({a})]^\pm +  \sum_{E} \frac{a_1}{N!} \pi_{N*}\left(\delta_E\cdot [\PP\overline{\mathcal{SQ}}_g(\widetilde{a})]^\pm\right)\\
    &=& (\xi+a_1 \psi_1) [\PP\overline{\mathcal{SQ}}_g({a})]^\pm \\ && + \frac{a_1}{N!} \times \binom{N}{1+\delta_{a_1=-1}} \times (N-1-\delta_{a_1=-1})!  [\PP\overline{\mathcal{SQ}}_g({a'})]^\pm.\\
    &=& (\xi+a_1 \psi_1) [\PP\overline{\mathcal{SQ}}_g({a})]^\pm  + \frac{a_1}{1+\delta_{a_1=-1}}   [\PP\overline{\mathcal{SQ}}_g({a'})]^\pm.
\end{eqnarray*}
For the equality in the first line we have used the projection formula and the expression of $\pi_N^*\psi_1$ given by~\ref{for:proofpositive2}, and from the first line to the second we have used the expression~\ref{for:proofpositive2} and the push-forward formula~\ref{for:induction1}. Finally the proposition follows from the following expression
\begin{eqnarray*}
    &&\!\!\!\!\!\!\!\!\!\!\!\!\!\!\!\!\!\!\!\!\!\!\! \sum_{\oGamma \in {\rm Bic}_{g,n+N}(1)} \frac{m(\oGamma)}{|\Aut(\oGamma)|} \frac{1}{N!} \pi_{N*} \zeta_{\oGamma\, *} [\PP\overline{\mathcal{SQ}} (\oGamma)]^\pm  \\ &=& \sum_{E} \frac{a_1+2|E|}{N!} \pi_{N*} \zeta_{\oGamma_E*}[\PP\overline{\mathcal{SQ}}(\oGamma_E)]^\pm + \sum_{\oGamma \in {\rm Bic}_{g,n+N}^*(1)} \frac{m(\oGamma)}{|\Aut(\oGamma)|} \frac{1}{N!} \pi_{N*} \zeta_{\oGamma\, *} [\PP\overline{\mathcal{SQ}}(\oGamma)]^\pm  \\
    &=& \frac{a_1+2+2\delta_{a_1=-1}}{1+\delta_{a_1=-1}} [\PP\overline{\mathcal{SQ}}_g({a'})]^\pm + \sum_{\oGamma \in {\rm Bic}_{g,n+N}^*(1)} \frac{m(\oGamma)}{|\Aut(\oGamma)|} \frac{1}{N!} \pi_{N*} \zeta_{\oGamma\, *} [\PP\overline{\mathcal{SQ}}(\oGamma)]^\pm .
\end{eqnarray*}
From the first to the second line we have used the fact that the only graphs in the complement of ${\rm Bic}_{g,n+N}^*(1)$ with non-trivial contribution are the graphs $\oGamma_E$, and from the second to the third line we have used the push-forward formula~\ref{for:induction1}.  
\end{proof}

We are now ready to prove the main lemma of the present section. 
\begin{lemma}\label{lem:computation3}
    \ref{thm:main} is a consequence of~\ref{thm:DR} and~\ref{thm:segre}.
\end{lemma}

\begin{proof}
    First, we use~\ref{lem:computation1} to reduce the computation of the classes of strata of $k$-differentials to the computation of the $[\oM_{g}(a)]^\pm=p_*[\PP\oSQ_g(a)]$ with $a$ positive. Then we work by induction on $(g,n)$ to prove the proposition $(P_{g,n})$:  all classes of the form  $p_*\left(\xi^q[\PP\oSQ_g(a)]^\pm\right)$ are tautological and computable . The base of the induction is trivial as  there are no strata of holomorphic differentials for $g=0$, so  we assume that $(P_{g',n'})$ holds for $(g',n')$ smaller than a fixed $(g,n)$ with $g$ positive. In order to prove $(P_{g,n})$ we work by induction on $|a|$. The base of the induction $a=(1,\ldots,1)$ is given by~\ref{thm:segre}. Then  we use formula~\ref{for:induction3} to obtain
    \begin{equation}
2\xi^q [\PP\overline{\mathcal{SQ}}_g(a')]^\pm = p_*\left(\xi^{q+1}[\PP\overline{\mathcal{SQ}}_g(a)]^\pm \right) \xi + a_{i}  \psi_{i} p_*\left(\xi^{q}  [\PP\overline{\mathcal{SQ}}_g(a)]^\pm\right) - p_*(\xi^q\Delta)
\end{equation}
where $\Delta$ is a linear combination of terms defined by bi-colored graphs. For these graphs we use the induction hypothesis for the upper vertices in $V_0$. For the vertices in $V_{-1}$, we remark that the genus of each vertex is smaller than $g$ (as vertices in $V_0$ are of positive genus), and use~\ref{lem:computation2} to prove that their contribution is tautological and computable.   
\end{proof}

\subsection{Relation between $[\PP\oH_g(-1,1^n)]^\pm$ and $[\PP\oH_g(1^{n+1})]^\pm$} 
Here we fix some $n\geq 0$, and consider the vectors $(1^{n+1})$ and $=(-1,1^n)$. These two vectors play a special role as 
$$
{\rm dim}(\PP\oSQ_g(1^{n+1}))={\rm dim}(\PP\oSQ_g(-1,1^n))={\rm dim}(\oM_{g,n+1}).
$$
In particular, by analogy with the notation~\ref{for:notations} we set
$$
s_g^{1,\pm}(t)= \sum_{c\geq 0} t^c p_*\left(\xi^c[\PP\oSQ_g(-1,1^n)]^\pm\right).
$$
The main purpose of the section is to compare $s_g^\pm$ and $s_g^{1,\pm}$. To do so, we begin with computations in degree 0. For all $(g,m)\in \NN^2$, there exists $d(g,m)$ in $\QQ$ satisfying: for all $n\in \NN$ such that $2g-2+n+m>0$, we have
    $$p_*[\PP\oSQ_g(((-1)^m,1^{n},\{0\})]^\pm= 
 d(g,m) \cdot [\oM_{g,n+m}],$$
where the notation $\{0\}$ stands for the trivial space of residue conditions.
\begin{proposition}\label{pr:pm1}
    For all $(g,m)$ we have $d(g,m)=(-1)^{m+1} 2^{2g-1} + 2^{g-1}.$
\end{proposition}

\begin{proof}
    We will prove that the function $d$ satisfies the following induction formulas
    \begin{align}
        & d(g,m+2)=d(g,m), \label{for:inddfunction1}\\&d(g+1,m)=d(g,m+1)- 3\cdot d(g,m).\label{for:inddfunction2}
    \end{align}
    These two identities determine uniquely $d$ together with the base cases
    $$
    d(0,0)=0, \text{ and } d(0,1)=1.
    $$
      To prove these identities, we will fix a  stable graph $\Gamma \in {\rm Stab}_{g,n+m}$ with one separating edge. If we set $N=(g-1)+m$ (the number of non-marked zeros),  
    then the intersection 
    $$\pi_N^{-1}(\zeta_{\Gamma}[\oM_{\Gamma}]) \times_{\oM_{g,n+m+N}} \PP\oSQ_{g}(((-1)^{m},1^n,3^N),\{0\})$$ is the union of the divisor of $\PP\oSQ_{g}(((-1)^{m},1^n,3^N),\{0\})$ defined by twisted graphs $(\Gamma',I)$ for which  the stabilization  of $\Gamma'$ after forgetting the last $N$ legs is a specialization of $\Gamma$. These graphs are either: a graph with one horizontal edge or a bi-colored graph. The  morphism $p$ has fibers of positive dimension along the first type of divisor so they contribute trivially.  We denote by ${\rm Bic}_{g,n+m+N}(\Gamma)$ the set of bi-colored defining a component of this intersection. 
    Then, with this notation we have
    $$
    \pi_N^*\zeta_{\Gamma*}[\oM_{\Gamma}]\cdot p_*\left(\PP\oSQ_{g}(((-1)^{m},1^n),\{0\}\right)= \sum_{\oGamma \in {\rm Bic}_{g,n+m}(\Gamma)} c_{\oGamma} \cdot p_* \zeta_{\oGamma*}[\oSQ({\oGamma},\{0\})]^\pm 
    $$
    where $c_{\oGamma}$ is some coefficient in $\QQ$. Then, the function $d$ is  computed by the projection formula
    \begin{equation}
        \label{for:inductiond3}
    d(g,m) \cdot \zeta_{\Gamma*}[\oM_{\Gamma}]=\frac{1}{N!}\sum_{\oGamma \in {\rm Bic}_{g,n+m+N}(\Gamma)} c_{\oGamma}\cdot  \pi_{N*} p_*\left(\zeta_{\oGamma*}[\oSQ({\oGamma},\{0\})]^\pm\right).
    \end{equation}
    Let $\oGamma=(\Gamma',I) \in {\rm Bic}_{g,n+m+N}(\Gamma)$ be a graph with non-trivial contribution to this expression. We can check that after removing the last $N$ legs and stabilizing $\Gamma'$  we obtain  precisely $\Gamma$: indeed, otherwise the associated contribution would be supported on a locus of co-dimension greater than $1$ in $\oM_{g,n+m}$. This implies in particular that $h^1(\Gamma')=0$. As the residue condition defining our stratum is trivial, there can be only vertex in $V_{-1}$ and $V_0$ (to have a non-trivial push-forward under $p$). Therefore, the graph $\Gamma'$ is obtained from $\Gamma$ by a assigning the last $N$ legs to each vertex, and the twist is uniquely determined by this choice. The contribution of this graph is of the form 
    $$
    p_*(\zeta_{\oGamma}[\PP\oSQ(\oGamma,\{0\})]^\pm= \zeta_{\Gamma*}\left(p_{0*}[\PP\oSQ_{g[0]}(a[0],\{0\})]^\pm\otimes p_{-1*}[\PP\oSQ_{g[-1]}(a[-1],\{0\})]^\pm\right).
    $$
     This class vanishes after applying $\pi_{N*}$ if a coordinate of $a[-1]$ is smaller than $-1$. So, once we have chosen which vertex of $\Gamma'$ is in $V_0$,  then the number of legs incident to each vertex is determined and the twist at the edge has absolute value 1.   Then, the identities~\ref{for:inddfunction1} and~\ref{for:inddfunction2} will be derived from~\ref{for:inductiond3} for two different choices of graphs $\Gamma$: for the first one, $\Gamma$ has a vertex of genus $0$ carrying only the first two legs, while for the second one $\Gamma$ has a vertex of genus 1 without legs. Indeed for these choices of graphs~\ref{for:inductiond3} implies
    \begin{eqnarray*}
        d(g,m+2)&=&d({0,3}) \cdot d(g,m) + d({0,2}) \cdot d(g,m+1), \text{ and }\\
        d(g+1,m)&=&d(1,1)\cdot d(g,m)+d(1,0)\cdot d(g,m+1).
    \end{eqnarray*}
    We have already mentioned that $d(0,2)$ vanishes so the identities~\ref{for:inddfunction1} and~\ref{for:inddfunction2} follow from
    \begin{align*}
        d({0,3})=1, 
        d(1,0)=-3, \text{ and } d(1,0)=1.
    \end{align*}
    For any 3 points on a genus 0 curve, there is exactly one differential with poles of order $2$ without residues at each marking and zeros of order 2, so $d(0,3)=1$. To compute $d(0,1)$ we remark that that only the trivial (and even) spin structure has an holomorphic section. Finally, a point in   $\PP\SQ_1(-1,3)$ is the data of  $(E,x_1,x_2)$ such that  $2(x_2)-2(x_1)$ is a nontrivial 2-torsion divisor (which is odd)  so $d(1,1)=-3$. 
\end{proof}

We turn now to the higher degrees. For the second statement, we consider the set of {\em back-bones graphs} ${\rm BB}_{g,1+n} \subset {\rm Stab}_{g,1+n}$, i.e. the set of star graphs with no loops, and such that the central vertex  carries the first leg.  To a non-trivial back-bone graph $\Gamma$, we associate two bi-colored graphs $\oGamma_{-1}$ and $\oGamma_{1}$ compatible with $(-1,1^n)$ and $(1^{n+1})$ respectively: the central vertex is in $V_{-1}$ and the twist function has value $1$ at each half-edge in $V_{\rm out}$. 

\begin{lemma}\label{lem:pm1}
    For $a_1=-1$ or $1$, we have
    \begin{equation}\label{for:pm1main}
    \left(\xi+ a_1\psi_1\right) [\PP\oSQ_g(a_1,1^n)]^\pm = 2 [\PP\oSQ(3,1^n)] + \underset{\text{non-trivial}}{\sum_{\Gamma \in {\rm BB_{g,1+n}}}} \frac{1}{|{\rm Aut}(\Gamma)|} \zeta_{\oGamma_{a_1} *} [\PP\oSQ(\oGamma_{a_1})]^\pm
 \end{equation}
\end{lemma}

\begin{proof}
    The present lemma is a consequence of~\ref{pr:indformulanonmarked} applied to $a=(a_1,1^n)$. Let $\oGamma$ be a bi-colored graph in ${\rm Bic}_{g,1+n+N}^*(1)$ such that $\pi_{N_*}\zeta_{\oGamma *}[\PP\oSQ(\oGamma)]^\pm\neq 0$. First,  the dimension of $\oSQ_{\bg[-1]}(\ba[-1],R[-1])$ is equal to 
    $$
    -{\rm codim}(R[-1])+\sum_{v \in V_{-1}} (2g(v)-2+n(v)),
    $$
    so the co-dimension of $p_{-1}(\PP\oSQ_{\bg[-1]}(\ba[-1],R[-1]))$ in $\oM_{\bg[-1],\bn[-1]}$ is at most
    $$
    {\rm codim}(R[-1])-1+ \sum_{v \in V_{-1}} (g(v) +1) \leq |V_0| + \delta_{a_1=-1}-1+ \sum_{v \in V_{-1}} (g(v)+1).
    $$
    This inequality follows from the following observation: apart from the vanishing residue at $x_1$, each vertex in $V_0$ is responsible for one linear constraint of residues at edges although two such constraints can be redundant (see~\cite{BCGGM}). Besides, for each vertex we denote by $N(v)$ the number of markings in $\{n+2,\ldots,n+1+N\}$ incident to $v$. Then we have 
    $$
    \sum_{v\in V_{-1}} N(v)  = \delta_{a_1=-1} + \left(\sum_{v\in V_{-1}} g(v)-1\right) + \underset{h\mapsto v \in V_{-1} }{\sum_{(h,h')\in \Gamma }} \frac{I(h')+1}{2}  \leq \delta_{a_1=-1} + \left(\sum_{v\in V_{-1}} g(v)-1\right) + |E(\Gamma)|
    $$
    where the second sum runs over all edges, and $I(h')$ is the value of the twist at the vertex in $V_0$. Then the morphism $\pi_N$ has fibers of positive dimension unless
    $$
    |V_0| + |V_{-1}| + \delta_{a_1=-1}+ \sum_{v \in V_{-1}} g(v) \geq \delta_{a_1=-1}+ \left(\sum_{v \in V_{-1}} g(v)\right) - |V_{-1}| + |E(\Gamma)|. 
    $$
    This last inequality holds if and only if $h^1(\Gamma)=0$ and $|V_{-1}|=1$, i.e. if and only if $\Gamma$ is a back-bone graph. Finally, the above inequalities are equalities if and only if the the twist function has value $1$ at every half-edge in $V_{\rm out}$. If we denote by ${\rm BB}_{g,1+n+N}^* \subset {\rm BB}_{g,1+n+N}$ the sub-set of non-trivial graphs for which the central vertex is stable after forgetting the last $N$ legs, then we have 
    \begin{equation*}
    \left(\xi+ a_1\psi_1\right) [\PP\oSQ_g(a_1,1^n)]^\pm = 2 [\PP\oSQ(3,1^n)] + \frac{1}{N!} \sum_{\Gamma \in {\rm BB_{g,1+n+N}^*}} \frac{1}{|{\rm Aut}(\Gamma)|} \pi_{N_*}\zeta_{\oGamma_{a_1} *} [\PP\oSQ(\oGamma_{a_1})]^\pm.
 \end{equation*}
 The lemma follows as the set ${\rm BB}_{g,1+n+N}^*$ is the set of non-trivial graphs in ${\rm BB}_{g,1+n}$ plus a choice of labeling of the last $N$ legs.
 \end{proof}

Combining~\ref{pr:pm1} and~\ref{lem:pm1}, we obtain the main identity that will be used in the following section.
\begin{proposition}\label{pr:pm1main}
     For all $g\geq 0$, we have
    \begin{equation}
        (1-t\psi_1)s_g^{1,\pm}(t)= (1+t\psi_1) s_{g}^\pm(t)+  \sum_{\Gamma\in {\rm BB}_{g,1+n}} \frac{(-t)^{E(\Gamma)}2^{2g(v_0)}}{|{\rm Aut}(\Gamma)|}\zeta_{\Gamma*}\left([\oM_{g(v_0),n(v_0)}]\bigotimes_{v\in V_{\rm out}} s_{g(v)}^\pm(t)\right).
    \end{equation}
\end{proposition}

\begin{proof}
If $\Gamma$ is a back bone graph then we denote by  
\begin{eqnarray*}
    \oSQ_\Gamma&=& \oM_{g(v_0),n(v_0)}\times \left(\prod_{v\in V_{\rm out}} \oSQ_{g(v),n(v)} \right), \\
    {[}\PP\oSQ_\Gamma{]}^\pm &=& [\oM_{g(v_0),n(v_0)}]\times \bigotimes_{v\in V_{\rm out}} 
 [\PP\oSQ_{g(v),n(v)} ]^\pm, \text{ and} \\
 s_\Gamma^\pm(t) &=& \sum_{c\geq 0} t^c p_*(\xi^c [\PP\oSQ_\Gamma]^\pm). 
\end{eqnarray*}
As $\oSQ_\Gamma$ is a product of cones, we can compute its Segre class in terms of each piece:
$$
s_\Gamma^\pm(t)= t^{E(\Gamma)-1} \left([\oM_{g(v_0),n(v_0)}]\bigotimes_{v\in V_{\rm out}} s_{g(v)}^\pm(t)\right).
$$
With this notation, for $a_1=-1$ or $1$, we have  $$[\PP\oSQ(\oGamma_{a_1})]^\pm=d(g(v_0),|E(\Gamma)|+\delta_{a_1=-1})\cdot [\PP\oSQ_\Gamma]^\pm.$$
So, combining~\ref{pr:pm1} and~\ref{lem:pm1}, we obtain the following identity:
\begin{eqnarray*}
(\xi-\psi_1)[\PP\oSQ_g(-1,1^n)]^\pm &=& (\xi+\psi_1)[\PP\oSQ_g(1^{n+1})]^\pm + \sum_{\Gamma \in {\rm BB}_{g,1+n}} \frac{f(\Gamma)}{|{\rm Aut}(\Gamma)|} \zeta_{\Gamma*}[\oSQ_\Gamma]^\pm, \text{ with }\\
    f(\Gamma)&=&d(g(v_0),|E|+1)-d(g(0),|E|) = (-1)^{|E|} 2^{2g}.
\end{eqnarray*}
Summing over all degree this identity implies
\begin{equation*}
        (1-t\psi_1)s_g^{1,\pm}(t)= (1+t\psi_1) s_{g}^\pm(t)+ C_0 + \sum_{\Gamma\in {\rm BB}_{g,1+n}} \frac{(-t)^{E(\Gamma)}2^{2g(v_0)}}{|{\rm Aut}(\Gamma)|}\zeta_{\Gamma*}\left([\oM_{g(v_0),n(v_0)}]\bigotimes_{v\in V_{\rm out}} s_{g(v)}^\pm(t)\right),
    \end{equation*}
where $C_0$ is a chow class of degree 0. We use again~\ref{pr:pm1} to deduce that $C_0=2^{2g}$. 
\end{proof}

\section{Cones of spin sections}\label{sec:cone}

We begin with a very general discussion of the cone of sections of a line bundle on a curve. 
\subsection{Cones of sections}\label{sec:cones_of_sections}
Let $\pi\colon C \to S$ be a prestable curve and $\ca L$ a line bundle on $C$. We define a functor $H_\ca L$ from $S$-schemes to sets by the formula 
$$T \mapsto H^0(T, \pi_{T*}\ca L_T)$$
where $\ca L_T$ is the pullback of $\ca L$ to $C_T = C \times_S T$. We claim that this functor is representable by an affine $S$-scheme which carries a natural $\bb G_m$-action (a `cone' over $S$). 

More concretely, we show that $H_\ca L$ is representable by the relative spectrum of the sheaf of graded $\ca O_S$-algebras defined by the formula $Sym^\bullet R^1\pi_*\ca F$ where $\ca F = \mathcal{H}om(\ca L, \omega_\pi)$. Indeed, for any $T\to S$ we have
    \begin{equation}\label{eq:representing_cone_of_sections}
        \begin{split}
            H_\ca L(T) & = H^0(T, \pi_{T*}\ca L_T)\\
            & = H^0(T, \pi_{T*}\mathcal{H}om(\ca F_T, \omega_{\pi_T}))\\
            & = H^0(T, \mathcal{H}om(R^1\pi_{T*}\ca F_T, \ca O_{T}))\\
            & = H^0(T, \mathcal{H}om((R^1\pi_{*}\ca F)_T, \ca O_{T}))\\
            & = Hom_{\ca O_T-mod}((R^1\pi_*\ca F)_T, \ca O_{T})\\
            & = Hom_{\ca O_T-alg}(Sym_{\mathcal O_T}^\bullet ((R^1\pi_*\ca F)_T), \ca O_{T})\\
            & = Hom_{\ca O_T-alg}((Sym_{\mathcal O_S}^\bullet R^1\pi_*\ca F)_T, \ca O_{T})\\
            & = Hom(T, \operatorname{Spec}_S (Sym_{\ca O_S}^\bullet R^1\pi_*\ca F);
        \end{split}
    \end{equation}
    here the third equality is by Grothendieck duality (in the form found on \cite[p243]{liu2002algebraic}), and we use that formation of $R^1\pi_*$ commutes with base change for curves, and that formation of the symmetric algebra commutes with base-change. 

\subsection{The cone of meromorphic spin sections}
In order to prove~\ref{thm:segre}, we will study the moduli spaces of (meromorphic) spin sections defined  for $P=(p_1,\ldots,p_n)\in \NN^n$ by 
$$
\wH_{g}^\half[P] = \{(C,x_1,\dots, x_n, \phi\colon L^{\otimes 2}\to \omega_{C},s), \text{with } (C,x_i,L,\phi)\in \oM^{1/2}_{g,n}, \text{ and } s\in H^0(C,L(p_1x_1+\ldots))\}.
$$
This space is a cone over $\oM_{g,n}^\half$. Besides we have a canonical morphism
\begin{eqnarray*}
    {\rm sq}\colon \PP\wH_{g}^\half[P]&\to&\PP\oH_g[2P]\\
    (\phi,s)&\mapsto& \phi(s^{2})
\end{eqnarray*}
(the underlying curve is mapped to the stabilization, and the spin structure is forgotten). Using this morphism, we will describe the irreducible components of (some of the) $\PP\wH_{g}^\half[P]$  in terms of the $\PP\oSQ_{g}(Z)$ of the previous sections. On the other hand, we will express the Segre classes of the $\PP\wH_{g}^\half[P]$ in terms of Chiodo classes. Altogether this will allow for the computation of the Segre classes of the $\oSQ_{g,n}$ and the proof of~\ref{thm:segre}.

\subsection{Resolution and virtual dimension  of $\wH_g^\half[P]$}\label{ssec:resolution} Our first task is to describe the cone $\wH_{g}^\half[P]$  as the kernel of a morphism between vector bundles. We recall the following lemma.
\begin{lemma}[{\cite[Lemma 43]{BHPSS}}] Let $(L\to C\to B,x_1,\ldots,x_n\colon B\to C)$ be a family of line bundles over a family of  pre-stable curves. There exist $m\in  \NN$,  a family of pre-stable curves $(C'\to B',x_1,\ldots,x_{n+m}\colon B'\to C')$, and morphisms $\beta\colon B'\to B$ and $\beta^\#\colon C'\to C\times_{B} B'$ commuting with the $n$ first sections such that:
\begin{itemize}
\item $\beta$ is an alteration (i.e. a proper, surjective and generically finite morphism);
\item $\beta^\#$ is a de-stabilization (i.e. it contracts part of the rational components);
\item for each fiber of $C'$, the complement of the components of $C'$ without markings are chains of $\PP^1$ for which $\beta^{\#*}{L}$ is trivial.
\end{itemize}
\end{lemma}
Applying this lemma to the family $\L\to \oC^\half_{g,n}\to \oM_{g,n}^\half$ we obtain a family of stable curves $(\pi\colon \C\to \M, x_1,\ldots,x_{n+m})$ and a de-stabilization $(\beta^\#,\beta)\colon (\C\to\M)\to (\oC^{\half}_{g,n}\to \oM_{g,n}^{1/2})$ such that $\beta$ is an alteration and each component of a fiber of $\C$ has a marking. By abuse of notation, we still denote by $\L\to \C$ the pull-back of the universal spin structure along $\beta^{\#}$. With this notation, for $M$ sufficiently large the group $H^1(C,\mathcal{L}(Mx_1+\ldots+Mx_{n+m}))$ is trivial for all fibers of $\C$, and we can consider the vector bundle
$$
E\coloneqq  \pi_*\mathcal{L}(M(x_1+\ldots+x_{n+m})).
$$

The second type of vector bundles that we consider here are the {\em jet bundles} $J_{i}^{q_1,q_2} \to \M$ whose fiber over a point $(C,L,x_i)$ is given by 
$$
H^0(C,L(q_2 x_i)/L(q_1 x_i)).
$$ Then for all $P\geq 0$ there exists a canonical residue morphism 
$$
r_P\colon E \to F_P \coloneqq   \bigoplus_{i=1}^n J^{M,p_i}_{i} \bigoplus_{i=n+1}^{n+m} J^{M,0}_{i}
$$
defined by mapping a meromorphic section of $L$ to the coefficients of the pole at each markings. By construction, the kernel of $r_P$ is precisely the cone $\wH=\wH_{g}^\half[P]\times_{\oM_{g,n}^\half} \M$. In particular we define the {\em virtual dimension} of $\wH_{g}^\half[P]$ to be 
$$
{\rm vdim}[P]= {\rm dim}(\oM_{g,n}^\half) + {\rm rk}(E)-{\rm rk}(F_P)= 3g-3+n + \sum_{i=1}^n p_i .
$$
\begin{lemma} The dimension of $\wH_g^\half[P]$ is greater than or equal to ${\rm vdim}[P]$.
\end{lemma}

\begin{proof}
    Let $H$ be an irreducible component of $\wH^{\half}_g[P]$, with image $B$ in $\Mbar_{g,n}^{1/2}$. As $\beta^{-1}(B)\to B$ is proper,  
    there exist a space $\M'$ sitting in the following diagram 
    $$
    \xymatrix{\M' \ar[r]^f \ar[rd]_{\beta'} & \beta^{-1}(B) \ar[d]\\ & B}
    $$
    where $h$ is an alteration~\cite[\href{https://stacks.math.columbia.edu/tag/0DMN}{Tag 0DMN}]{stacks-project}. Then we have $\beta'^*\wH_g[P]=f^*\wH$, and 
    \begin{align*}
        {\rm dim}(\wH_g[P])= {\rm dim}(\beta'^*\wH_g[P])= {\rm dim}(f^*\wH)\geq {\rm vdim}[P].
    \end{align*}
    For the last inequality, we use that $f^*\wH$ is the kernel of the pull-back of $r_P$ along $f$.
\end{proof}

\subsection{Irreducible components of $\wH_{g}[p,0^{n-1}]$} 
We denote by 
$\H_{g}^\half[P] \subset \wH_g^\half[p]$ the sub-cone of sections supported on a smooth curve, and with pole of order exactly $p_i$ at $x_i$. We denote by $\oH_{g}^\half[P]$ its closure. The squaring morphism restricts to a morphism
$$
{\rm sq}: \PP\oH_{g}^\half[P] \to \PP\oSQ_{g}((-2p_1+1,-2p_2+1\ldots))
$$
which is finite. Thus, the dimension of $\oH_g^\half[P]$ is given by 
$$
d[P]= {\rm vdim}[P]+ \delta_{P=(0^n)}.
$$

From now on, we fix $p\geq 0$, and we restrict our attention to the case of $P=(p,0^{n-1})$. We denote by ${\rm Tree}_{g,n}$ the set of stable graphs with $h^1(\Gamma) = 0$. If $\Gamma$ is a graph in ${\rm Tree}_{g,n}$, then there is only one weighting function on this graph, and it takes the value 1 at each edge, so will omit it in the notation. We set
$$
\oH^{1/2}_{\Gamma}[p]= \oH^{\half,\mu(v_0)}_{g(v_0)}[(p, 0^{n(v_0)-1})] \times \prod_{v\neq v_0} \oH_{g(v),n(v)}^{\half,\mu(v)}.
$$
where $v_0$ is the vertex carrying the first leg. This space is a cone over $\oM^{\half}_\Gamma$ and the gluing morphism $\zeta_{\Gamma}\colon \oM_\Gamma^\half\to \oM_{g,n}^\half$ lifts to a morphism $\zeta_{\Gamma}\colon \oH^{1/2}_{\Gamma}[p]\to \oH_{g,n}^\half[p]$ as the sections defined component by component glue to define a unique global section of the spin structure. Besides, we have
$$
{\rm dim}(\oH^{1/2}_{\Gamma}[p]) = d[P].
$$
In particular, by this dimension count, we see that  $\wH_{g,n}^\half[p]\neq \oH_{g,n}^\half[p]$ in general. In other words,  the cone $\wH_{g,n}^\half[p]$ has irreducible components supported on the boundary of $\oM_{g,n}^\half$.

\begin{proposition}
    For all $p\geq 1$,  the space  $\wH_{g}^\half[(p,0^{n-1})]$ is of pure dimension $d[p,0,\ldots]$, and 
    $$
    \wH_{g}^\half[p,0^{n-1}] = \bigcup_{\Gamma \in {\rm Tree}_{g,n}} \zeta_{\Gamma}\left(\oH^{1/2}_{\Gamma}[p]\cup \oH^{1/2}_{\Gamma,\mu}[0] \right). 
    $$
\end{proposition}

\begin{remark} Using the classification of irreducible components of strata of meromorphic differentials by Boissy~\cite{Boi}, we see that the irreducible components of $\oH_{\Gamma}^\half[p]$ are completely determined by assigning a sign to each vertex of $\Gamma$. Therefore, this proposition provides a full classification of the irreducible components of $\wH_{g}^\half[p,0,\ldots]$.
\end{remark}

\begin{proof}  Let $H$ be an irreducible component of $\wH_{g}^\half[(p,0^{n-1})]$. The generic point of this component determines the following numerical data: 
\begin{itemize}
    \item a weighted graph $(\Gamma,w)$;
    \item a set $V_0\subset V(\Gamma)$ of vertices where the section does not vanish identically;
    \item for any half-edge $h$ incident to a vertex in $V_0$, a positive integer $I_h$ equal to $1$+the order of the square of the section at this marking.
\end{itemize}
(we cannot say that the $I_h$ define a twist because they are not defined at all half-edges). If $(h,h')$ is an edge, and $w(h)=0=I_h=0$, then we also have $w(h')=0$ and $I_{h'}=0$ by the residue condition. Then, the generic point of $H$ is obtained as a limit of spin sections on curves with one less node so $H$ cannot be an irreducible component. Therefore we have the constraint $I_h>0$ for all half-edges. Besides, if an edge connects two vertices in $V\setminus V_0$ then the node can also be smoothed, and by the same reasoning, we can assume that such edges are not allowed.
\smallskip

In order to conclude, we remark that the image of $H$ under the squaring morphism is contained in the image (after gluing) of the following space
$$
\M_H=\prod_{v\in V_0} \SQ_{g(v)}(I(v)) \prod_{v\notin V_0} \oM_{g(v),n(v)}.
$$
The dimension of $\M_H$ is equal to 
\begin{equation*} \sum_{v\notin V_0} 3g(v)-3+n(v) + \sum_{v\in V_0} 3g(v)-2+n(v) - (|I'(v)|-n(v)) + \delta_{1\mapsto v} \cdot (p_1-1)
\end{equation*}
where $\delta_{1\mapsto v}= 1$ if the unique leg is incident to $v$, while $I'(v)$ is the vector of the $I_h$ for $h$ incident to $v$ and different from the leg.  This expression can be re-organized as follows
\begin{eqnarray*}
    {\rm dim}(\M_H)&=& \delta_{1\mapsto V_0}\cdot (p_1-1) +   \left(\sum_{v\in V} 3g(v)-3+n(v)\right) - \sum_{v \in V_0} (|I(v)|-n(v)-1) \\
    &=& \delta_{1\mapsto V_0}\cdot (p_1-1) + (3g-3+n) - |E|+|V_0|,
\end{eqnarray*}
where $\delta_{1\mapsto V_0}$ is equal to 1 if the leg is incident to any vertex in $V_0$, and 0 otherwise. This dimension is equal to ${\rm vdim}[p_1]$ if and only if  $V=V_0$, and $|V|=|E|+1$, i.e. $h^1(\Gamma)=0$. Finally, as $V=V_0$, for a generic point of $H$, the restriction of the section at any vertex is non trivial. This imposes that the vertex is of positive genus unless it carries the unique leg.
\end{proof}

\subsection{Computing the signed Segre  classes of $\wH_{g}[P]$} For all $(g,n)$, we set
\begin{eqnarray*}
\widetilde{s}^{0,\pm}_g(t)&=&\sum_{\Gamma \in {\rm Tree}_{g,n}} \frac{t^{|E(\Gamma)|}}{|{\rm Aut}(\Gamma)|} \zeta_{\Gamma *} \left( \bigotimes_{v} s_{g(v)}^\pm(t)\right), \\
    \widetilde{s}^{1,\pm}_g(t)&=&\sum_{\Gamma \in {\rm Tree}_{g,n}} \frac{t^{|E(\Gamma)|}}{|{\rm Aut}(\Gamma)|} \zeta_{\Gamma *} \left(s_{g(v_0)}^{1,\pm}(t) \bigotimes_{v\neq v_0} s_{g(v)}^\pm(t)\right), \text{and} \\
    \widetilde{s}^{\pm}_g(t)&=& \widetilde{s}^{0,\pm}_g(t)+ \widetilde{s}^{1,\pm}_g(t).
\end{eqnarray*}
Note that the first and the second sum could be indexed by ${\rm Tree}_{g,n}^\star\subset {\rm Tree}_{g,n}$ as the other graphs have a vertex of genus 0, and thus contribute trivially. We keep these trivial summands to simplify the algebraic manipulations. 

\begin{proposition}
     For all $g\geq0,$ we have \begin{equation}\label{for:pm10}
\widetilde{s}^{\pm}_g(t)= 2^{g}\frac{\Lambda(-2t)\Lambda(t)}{1-t\psi_1}.
\end{equation}
\end{proposition}

\begin{proof} The restriction of the squaring morphism to $\wH_{g}^{\half}[P]$  is of degree $1/2$ on its image due to presence of automorphisms. Besides, the pull-back of $\mathcal{O}(-1)$ is equal to $\mathcal{O}(-1)^{\otimes 2}$ so 
    \begin{eqnarray*}
        \sum_{c\geq 0} \epsilon_* p_*\left((t\xi)^c[\PP\wH_{g}^{\half}[1,0^{n-1}]]^\pm\right) &=& \sum_{\Gamma \in {\rm Star}_{g,n} } \frac{1}{|{\rm Aut}(\Gamma)|} \zeta_{\Gamma*}\left((t\xi)^c\left([\PP\oH_\Gamma[0]]^\pm+ [\PP\oH_\Gamma[1]]^\pm\right)\right)\\&=&\frac{1}{2}\left(\widetilde{s}^{0,\pm}_g(t/2)+\widetilde{s}^{1,\pm}_g(t/2)\right).
    \end{eqnarray*}
    Then, we consider the line bundle $\mathcal{L}(-x_1)\to \oC_{g,n}^\half\overset{\pi}{\to} \oM_{g,n}^\half$. As $h^0(C, L(-x_1))=0$ for a generic point of each component of $\oM_{g,n}^\half$, the sheaf $\pi_*\mathcal{L}(-x_1)$ is trivial. By \ref{eq:representing_cone_of_sections} we know that 
    \begin{equation}
        \wH_g^\half[1,0^{n-1}] = Spec (Sym^\bullet R^1\pi_*\ca L(-x_1)), 
    \end{equation}
    hence according to \cite[Example 4.1.7]{Fulton} the Segre classes for $\wH_g^\half[1,0^{n-1}]$ are given by those of $Sym^\bullet R^1\pi_*\ca L(-x_1)$. The same then holds for the Chern classes, yielding 

    \begin{align*}
        ch_d(\wH_{g}^{\half}[1,0^{n-1}])=(-1)^{d}ch_d(R^1\pi_* \mathcal{L}(-x_1)) = (-1)^{d+1} ch_d(R\pi_*\L(-x_1)),
    \end{align*}
    and we can use Chiodo's formula to compute these characters:
    \begin{eqnarray*}
        ch_d\big( R\pi_* \mathcal{L}(-x_1)\big)&=& \frac{B_{d+1}(1/2)}{(d+1)!} \left(\kappa_d -\sum_{i=2}^n \psi_i^d \right) - \frac{B_{d+1}(3/2)}{(d+1)!} \psi_1^d + \sum_{i=0,1}\sum_{j=0}^{d-1} \frac{B_{d+1}(i/2)}{(d+1)!} \zeta_{i*}(\psi_h^j\psi_{h'}^{d-1-j})\\
        &=&\frac{(2^{-d}-1)B_{d+1}}{(d+1)!} \left(\kappa_d -\sum_{i=1}^n \psi_i^d \right) - \frac{2^{-d}}{d!} \psi_1^d + \sum_{i=0,1}\sum_{j=0}^{d-1} \frac{B_{d+1}(i/2)}{(d+1)!} \zeta_{i*}(\psi_h^j\psi_{h'}^{d-1-j})
    \end{eqnarray*}
where $\zeta_i$ is the  gluing morphism from the union of boundaries defined by a graph with a single edge of weight $i$, and $\psi_h$ and $\psi_{h'}$ are the $\psi$-classes attached to each half-edge (see~\cite{Chiodo2006TowardsAE}). From the first to the second line we have used the identities $B_{d+1}(3/2)-B_{d+1}(1/2)=(d+1)2^{-d}$ and $B_{d+1}(1/2)=(2^{-d}-1)B_{d+1}$. After multiplying this expression by the parity cycle and  pushing-forward to $\oM_{g,n}$ the contribution of $\zeta_0$ becomes trivial by~\ref{prop:cancellation} so we obtain
\begin{eqnarray*}
    \widetilde{s}_g^\pm(t/2)&=& 2\epsilon_* \left( [\pm]\cdot  {\rm exp}\left(\sum_{d\geq 1} \frac{(-t)^d}{(d-1)!} ch_d(\wH^{1/2}[1,0^{n-1}])\right)\right)\\
&=&   2 \epsilon_* \left( [\pm]\cdot  {\rm exp}\left(\sum_{d\geq 1} \frac{(t\psi_1/2)^d}{d} -\frac{(2^{-d}-1)B_{d+1}t^d}{d(d+1)}\left(\kappa_d - \sum_{i=1}^n \psi_i^d + \sum_{j=0}^{d-1} \zeta_{0*}(\psi_h^j\psi_{h'}^{d-1-j}) \right)\right)\right) \\ 
    &=& 2^{g}\cdot {\rm exp}(-{\rm log}(t\psi_1/2)) \cdot (\Lambda(t/2)\Lambda(-t))^{-1} =  2^{g}\frac{\Lambda(t)\Lambda(-t/2)}{1-t\psi_1/2},
\end{eqnarray*}
where the last equality is obtained from Mumford's identity $\Lambda(t)^{-1}=\Lambda(-t)$~\cite{Mum}. 
\end{proof}

This proposition provides a first relation between $\widetilde{s}^{0,\pm}_g$ and $\widetilde{s}^{1,\pm}_g$. 
Now we use~\ref{pr:pm1main} from the previous section to prove an independent relation.

\begin{proposition}
    For all $(g,n),$ we have
    \begin{equation}\label{for:relationpm1}
    (1-t\psi_1)\widetilde{s}^{1,\pm}_g(t)=(1+t\psi_1)\widetilde{s}^{0,\pm}_g(t) +2^{2g}. 
    \end{equation}
\end{proposition}
    
\begin{proof} For all $(g,n)$ we write
\begin{eqnarray*}
    &&(1-t\psi_1)\widetilde{s}^{1,\pm}_g(t)-(1+t\psi_1)\widetilde{s}^{0,\pm}_g(t) \\
    &=& \sum_{\Gamma\in {\rm Tree}_{g,n}} \frac{t^{|E(\Gamma)|}}{|{\rm Aut}(\Gamma)|}\zeta_{\Gamma*}\left( \left((1-t\psi_1)s^{1,\pm}_{g(v_0)}(t)-(1+t\psi_1) s_{g(v_0)}^\pm(t) \right) \right) \bigotimes_{v\neq v_0} s_{g(v)}^{\pm}(t)\\
    &=& \sum_{\Gamma\in {\rm Tree}_{g,n}} \sum_{\Gamma'\in {\rm BB}_{g(v_0),n(v_0)}} \frac{t^{|E(\Gamma)|}(-t)^{|E(\Gamma')|}}{|{\rm Aut}(\Gamma)|\cdot|{\rm Aut}(\Gamma')|}\zeta_{\Gamma*}\left( \zeta_{\Gamma'*}\left(2^{2g(v_0')} \bigotimes_{v'\neq v_0'} s_{g(v')}^\pm(t) \right)  \bigotimes_{v\neq v_0} s_{g(v)}^{\pm}(t)\right).
\end{eqnarray*}
From the first line to the second we have used~\ref{pr:pm1main} at the vertex $v_0$, and 
in the second sum the notation $v_0'$ stands for the central vertex of the back-bone graphs.  This sum can be indexed differently as the data of the pair $(\Gamma,\Gamma')$ is equivalent to the data of $(\Gamma'', E'')$ where $\Gamma''$ is a graph in ${\rm Tree}_{g,n}$ and $E''\subset E(v_0'')$ is a subset of the edges incident to the root (the vertex carrying the first leg).  Indeed, the graph $\Gamma''$ is constructed by replacing the root of $\Gamma$ by $\Gamma'$, and the set of $E''$ is the set of edges coming from $\Gamma'$. Conversely, the graph $\Gamma$ is constructed from $\Gamma''$ by contracting the edges in $E''$, and $\Gamma'$ is the connected graph containing $v_0''$ and the edges in $E''$. Then the last line can be re-written as \begin{eqnarray*}
    &&\sum_{\Gamma''\in {\rm Tree}_{g,n}} \frac{t^{|E(\Gamma'')|}}{|{\rm Aut}(\Gamma'')|} \left(\sum_{E''\subset E(v_0'')} (-1)^{|E''|}\right)\cdot \zeta_{\Gamma*}\left(2^{2g(v_0)} \bigotimes_{v\neq v_0''} s_{g(v)}^\pm(t) \right)\\
    &=& \sum_{\Gamma''\in {\rm Tree}_{g,n}} \frac{t^{|E(\Gamma'')|}}{|{\rm Aut}(\Gamma'')|} 0^{|E(v_0'')|} \zeta_{\Gamma*}\left(2^{2g(v_0)} \bigotimes_{v\neq v_0''} s_{g(v)}^\pm(t) \right)  = 2^{2g}.
\end{eqnarray*}
The last equality is due to the presence of the factor $0^{|E(v_0'')|}$. It vanishes unless $\Gamma''$ is trivial.  
\end{proof}

With these two propositions we are now ready to complete the proof of~\ref{thm:segre}.

\begin{proof}[Proof of~\ref{thm:segre}]
    We fix $g\geq 0$, and we multiply formula~\ref{for:pm10} by $(1-t\psi_1)$ to get
    \begin{eqnarray*}
2^{g}\Lambda(-2t)\Lambda(t)&=& (1-t\psi_1)\widetilde{s}^{1,\pm}_g(t)+(1-t\psi_1)\widetilde{s}^{0,\pm}_g(t)\\
        &=& 2^{2g} +2 \widetilde{s}^{0,\pm}_g(t) \\ &=& 2\left(2^{2g-1}+ \sum_{\Gamma \in {\rm Tree}_{g,n}} \frac{t^{|E(\Gamma)|}}{|{\rm Aut}(\Gamma)|}\zeta_{\Gamma *}\left(\bigotimes_{v\in V} s_{g(v)}(t) \right)\right).
    \end{eqnarray*}
    From the first to the second line we have used the identity~\ref{for:relationpm1}, and from the second to the third we have used the fact that the sum defining  $\widetilde{s}^{0,\pm}_g(t)$ can be restricted to graphs in ${\rm Tree}^\star_{g,n}\subset {\rm Tree}_{g,n}$ as graphs with a vertex of genus 0 contribute trivially. This completes the proof of the first identity of the theorem.
    
    \smallskip
    
    The second identity is derived from the first one by induction the so-called Möbius inversion. More precisely, we work by induction on $(g,n)$ (the base case $(g,n)=(0,3)$ is trivial). For simplicity, we write all sums over ${\rm Trees}_{g,n}^\star$ as sums over ${\rm Trees}_{g,n}$ as both $s_0=L_0=0$. We remark that the data of a graph in ${\rm Tree}_{g,n}$ is equivalent to the data of $(\Gamma',(\Gamma_v)_{v\neq v_0})$ where $\Gamma'$ is a graph in ${\rm BB}_{g,n}$ and $\Gamma_v$ is a graph in ${\rm Tree}_{g(v),n(v)}$ for each vertex $v\neq v_0$ (we simply say that a rooted tree is the same as a root attached to a set of rooted trees). Therefore we can write 
    \begin{eqnarray*}
        L_g(t)\!\!\!\! &=&{\sum_{\Gamma \in {\rm Tree}_{g,n}}} \frac{t^{|E(\Gamma)|}}{|{\rm Aut}(\Gamma)|}\zeta_{\Gamma *}\left(\bigotimes_{v\in V} s^\pm_{g(v)}(t) \right) \\
         &=&{\sum_{\Gamma' \in {\rm BB}_{g,n}}} \frac{t^{|E(\Gamma')|}}{|{\rm Aut}(\Gamma')|}\zeta_{\Gamma' *} \left(s^\pm_{g(v_0)}(t) \bigotimes_{v\neq v_0} \left( \sum_{\Gamma_v \in {\rm Tree}_{g(v),n(v)}} \frac{t^{|E(\Gamma_v)|}}{|{\rm Aut}(\Gamma_v)|} \zeta_{\Gamma_v*}\left(\bigotimes_{v' \in V(\Gamma_v)}s_{g(v')}^\pm(t) \right) \right)\right)\\
         &=& {\sum_{\Gamma' \in {\rm BB}_{g,n}}} \frac{t^{|E(\Gamma')|}}{|{\rm Aut}(\Gamma')|}\zeta_{\Gamma' *} \left(s^\pm_{g(v_0)}(t) \bigotimes_{v\neq v_0} L_{g(v)}(t) \right) \\
         &=& s_g^\pm(t)+\!\!\!\! \underset{\text{non-trivial}}{{\sum_{\Gamma' \in {\rm BB}_{g,n}}}}\sum_{\Gamma''\in {\rm Tree}_{g(v_0),n(v_0)}} \frac{t^{|E(\Gamma')|}(-t)^{|E(\Gamma'')|}}{|{\rm Aut}(\Gamma'')||{\rm Aut}(\Gamma_0)|}
         \zeta_{\Gamma' *} \left(\left(\bigotimes_{v'\neq v_0''} L_{g(v')}(t)\right)  \bigotimes_{v \in V(\Gamma'')}\! L_{g(v)}(t) \right). 
    \end{eqnarray*}
    From the second to the third line we have applied the first identity of the theorem to each vertex $v\neq v_0$ in $\Gamma'$, while for the last equality we have applied the induction hypothesis to the central vertex. Here we re-index  again this sum by observing that the datum of $(\Gamma',\Gamma'')$ is equivalent to the datum of $(\Gamma,V'')$ where $\Gamma$ is non-trivial graph in ${\rm Tree}_{g,n}$ and $V''$ is a non-trivial subset of the leaves of this graph (considered as a rooted tree). Then we can write
    \begin{eqnarray*}
        s_g^\pm(t) &=& L_g(t) -  \underset{\text{non-trivial}}{{\sum_{\Gamma \in {\rm Tree}_{g,n}}}} \left(\underset{V''\neq \emptyset}{\sum_{V''\subset \text{leaves($\Gamma$)}}} (-1)^{|V''|} \right) \frac{(-t)^{|E(\Gamma)|}}{|{\rm Aut}(\Gamma)|}
         \zeta_{\Gamma *} \left(\bigotimes_{v \in V(\Gamma'')} L_{g(v)}(t) \right)\\
         &=&L_g(t) +  \underset{\text{non-trivial}}{{\sum_{\Gamma \in {\rm Tree}_{g,n}}}}  \frac{(-t)^{|E(\Gamma)|}}{|{\rm Aut}(\Gamma)|}
         \zeta_{\Gamma *} \left(\bigotimes_{v \in V(\Gamma'')} L_{g(v)}(t) \right).
    \end{eqnarray*}
\end{proof}


\newcommand{\etalchar}[1]{$^{#1}$}

\end{document}